\documentclass[11pt]{article}

\usepackage{latexsym,amsfonts,amsmath,amssymb,amsthm}
\usepackage[T1]{fontenc}
\usepackage{graphics,graphicx}
\usepackage{color}
\usepackage{subcaption}

\usepackage[font={footnotesize}]{caption}

\usepackage{url}

\newtheorem{lem}{Lemma}[section]
\newtheorem{prop}[lem]{Proposition}
\newtheorem{rem}[lem]{Remark}

\newtheorem{theo}[lem]{Theorem}

\newtheorem{cor}[lem]{Corollary}

\newtheorem*{ack}{Acknowledgement}

\newtheorem*{algopatience}{Heap sorting for a sequence $(U_i,\nu_i)$}

\renewcommand{\P}{\mathbb{P}}
\newcommand{\R}{\mathbb{R}}
\newcommand{\N}{\mathbb{N}}

\newcommand{\E}{\mathbb{E}}

\newcommand{\defeq}{:=}

\newcommand{\dr}{\mathbf{R}}

\begin{document}

\title{From Hammersley's lines to Hammersley's trees.}
\author{\textsc{Basdevant A.-L.\footnote{Laboratoire Modal'X, Université Paris Ouest, France. email: anne-laure.basdevant@u-paris10.fr}, Gerin L.\footnote{CMAP, Ecole Polytechnique, France. email: gerin@cmap.polytechnique}, Gouéré J.-B.\footnote{Laboratoire de Mathématiques, Université de Tours, France. email: jean-baptiste.gouere@lmpt.univ-tours.fr}, Singh A.\footnote{Laboratoire de Mathématiques d'Orsay, Univ. Paris-Sud, CNRS, France. email: arvind.singh@math.u-psud.fr}}}
\date{\today}   
\maketitle

\begin{abstract}

We construct a stationary random tree, embedded in the upper half plane, with prescribed offspring distribution and whose vertices are the atoms of a unit Poisson point process. This process which we call \emph{Hammersley's tree process} extends the usual \emph{Hammersley's line process}. Just as Hammersley's process is related to the problem of the longest increasing subsequence, this model also has a combinatorial interpretation: it counts the number of heaps (\emph{i.e.} increasing trees) required to store a random permutation.\\ This problem was initially considered  by Byers \emph{et. al} (2011) and Istrate and Bonchis (2015) in the case of regular trees. We show, in particular, that the number of heaps grows logarithmically with the size of the permutation.
\end{abstract}

\bigskip

\begin{figure}[h]
\begin{center}
\includegraphics[width=7cm]{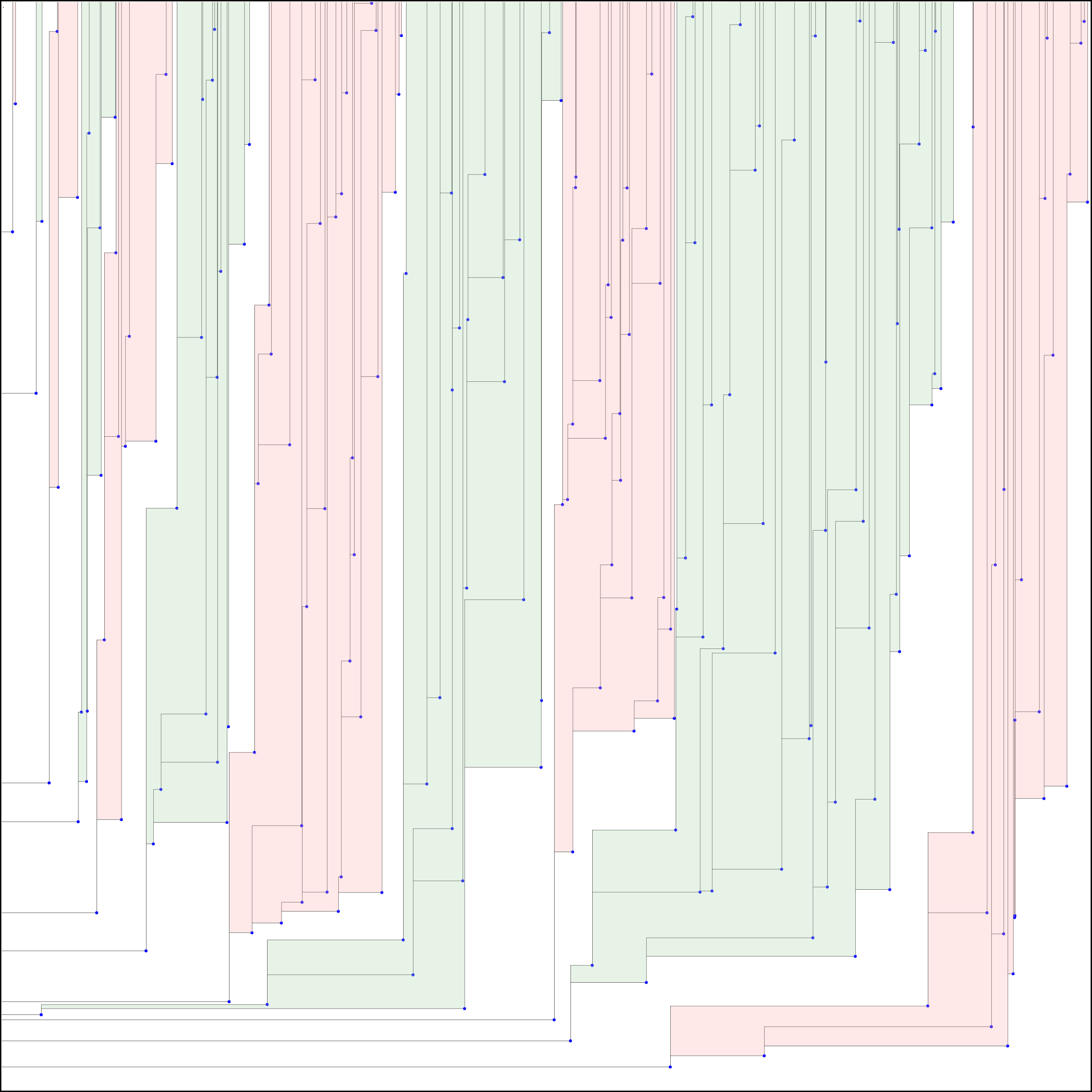}
\end{center}
\caption{Sorting a random permutation of size n=200 into binary heaps. Here, $11$ trees coloured alternatively in red and green are required.}
\end{figure}

\vfill

\noindent{\bf {\textsc MSC 2010 Classification}:} 60K35, 60G55.\\
\noindent{\bf Keywords:} Hammersley's process; Heap sorting, Patience sorting; Longest increasing subsequences; Interacting particles systems. 

\vfill

\newpage

\section{Introduction}

Given a permutation $\sigma$ of $\{1,\ldots,n\}$, a decreasing sub-sequence of $\sigma$ of length $\ell$ is a sequence $(\sigma_{i_1},\ldots,\sigma_{i_\ell})$ such that 
\begin{equation*}
i_1 < i_2 < \ldots < i_\ell \quad \hbox{and} \quad \sigma_{i_1} > \sigma_{i_2} > \ldots > \sigma_{i_\ell}.
\end{equation*}
The question of estimating the length $\ell(\sigma)$ of the longest decreasing sub-sequence of $\sigma$ when the permutation is sampled uniformly among all permutation of $\{1,\ldots,n\}$ also known as Ulam's problem, has a long standing history. In a seminal paper, Hammersley \cite{Hammersley} proved that, as $n$ goes to infinity, the almost-sure asymptotic $\ell(\sigma) \sim c\sqrt{n}$ holds for some constant $\pi/2 < c < e$. The bounds on $c$ were improved by Kingman in \cite{Kingman} and it was later shown independently by Logan and Shepp \cite{LoganShepp} and Ver{\v{s}}ik and Kerov \cite{VershikKerov} that, in fact, $c=2$. In recent years, important progress has been made in understanding finer properties of $\ell(\sigma)$. In particular, thanks to a remarkable connection with random matrices Baik, Deift and Johansson \cite{BaikDeiftJohansson} showed that the fluctuations of $\ell(\sigma)$ around its mean are of order $n^{1/6}$ and, correctly normalized, converge to the Tracy-Widom distribution. In fact, so much work was done on this subject that there is now a whole book devoted to the properties of the longest increasing sub-sequences of a random permutation \cite{Romik}. We refer the reader to it for additional details.

One remarkable property of this model is the following self-duality result: the length $\ell(\sigma)$ of the longest \emph{decreasing} sub-sequence of $\sigma$ is also equal to the minimum number of disjoint \emph{increasing} sub-sequences of $\sigma$ required to partition $\{1,\ldots,n\}$. This results follows from the fact that a minimal partitioning may be constructed by 
grouping the $\sigma_i$'s in ``stacks'' via the following greedy algorithm. Start with a single stack containing only $\sigma_1$. Now, if $\sigma_2$ is larger than $\sigma_1$, put it on top of $\sigma_1$ (\emph{i.e.} in the same stack), otherwise put $\sigma_2$ apart in a new stack. By induction, at each step $i$, put $\sigma_i$ on top of the stack whose top value is the largest one below $\sigma_i$. If all the stacks have a top value larger than $\sigma_i$, then put $\sigma_i$ in a new stack.

This procedure outputs a set of $\dr(\sigma)$ increasing stacks (equivalently a partition of $\sigma$ into $\dr(\sigma)$ increasing sub-sequences). It is easy to check that the number of stacks obtained by this algorithm is indeed minimal among all such partitioning and that, furthermore, we can construct a decreasing sub-sequence with maximal length by picking out exactly one element in each stack. See Figure \ref{fig:stack} for an example and some more details. Therefore, the random variables $\ell(\sigma)$ and $\dr(\sigma)$ are, in fact, equal.

It is useful to couple permutations with different lengths on the same space. We do so by considering a sequence $(U_i,\, i\geq0)$ of i.i.d. random variables with uniform distribution on $[0,1]$. For each $n$, let $\sigma^n$ to be the permutation of $\{1,\ldots,n\}$ which re-orders the $n$ first $U_i$'s increasingly \emph{i.e.} such that $U_{\sigma^n_1} < \ldots < U_{\sigma^n_n}$. By construction, $\sigma^n$ is uniformly distributed among all permutations of length $n$. Furthermore, the greedy algorithm described above is compatible in $n$ hence the random variables $\dr(n) \defeq \dr(\sigma^n)$ can be computed iteratively: 
\begin{itemize}
\item Start at time $1$ with $\dr(1) =1$ and a single stack containing only $U_1$.
\item At time $n$, we have $\dr(n)$ stacks. Each stack contains values in increasing order with the largest value on top. 
\item At time $n+1$, we add $U_{n+1}$ on top of the stack which has the largest value smaller than $U_{n+1}$. If there is no such stack, we put $U_{n+1}$ in a new stack.
\end{itemize}
This algorithm is called \emph{patience sorting} in reference to the eponymous card game \emph{c.f.} \cite{AldousDiaconis99}. One can think of it as the optimal way to stream incoming resources into ordered stacks while keeping the number of stacks minimum at all time. Hammersley's result is thus equivalent to the following almost-sure asymptotic on the number of stacks:
$$
\lim_{n\to\infty} \frac{\dr(n)}{\sqrt{n}} = 2 \hbox{ a.s.}
$$
In \cite{Byersetal}, Byers \emph{et al.} proposed a generalization of the patience sorting algorithm where the incoming $U_i$'s are not put into stacks anymore but instead into another classical data structure called a \emph{heap}. A heap is labelled tree satisfying the condition that the label of any node is always greater or equal than the label of its parent. Heaps are extremely common in computer science, especially since they allow implementation of efficient algorithms for searching and sorting. Of course, stacks are particular kind of trees so the classical patience sorting into stacks corresponds to sorting into unary heaps. However, as soon as the heaps have exponential growth, one can expect the patience sorting algorithm to behave differently and require asymptotically a much smaller number of heaps. This question was raised by  Istrate and Bonchis in \cite{IstrateBonchis} who predicted that the number $\dr(n)$ should be of logarithmic order. More precisely, when all heaps are binary trees, they conjectured that 
\begin{equation}\label{gdconj}
\dr(n) \sim \frac{1}{2}(1+\sqrt{5})\log n
\end{equation}
as $n$ goes to infinity. Yet, concerning this asymptotic, they only pointed out the crude asymptotic lower bound $\dr(n) \geq \log n$ which comes from the fact that each new running minimum $U_i$ must necessarily be placed into a new heap.

\begin{figure}
\begin{center}
\includegraphics[width=11cm]{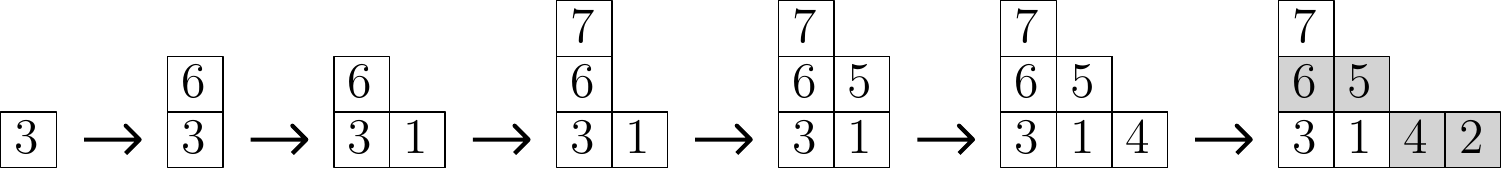}
\caption{
Example of the stack sorting algorithm for $\sigma = (3,6,1,7,5,4,2)$. The procedure outputs $4$ stacks. The elements in grey constitute a maximal decreasing sub-sequence of $\sigma$. It is obtained by picking the smallest element from the last stack and then recursively picking in the previous stack the smallest element larger than the one previously chosen.
}\label{fig:stack}
\end{center}
\end{figure}

In this paper, we consider a generalization of this model where we allow the heaps to be random Galton-Watson trees without leaves. Let $\mu$ denote a probability measure on $\N=\{1,2\ldots,\}$ which will correspond to the reproduction law of our Galton-Watson trees.  Let $((U_i,\nu_i),i\ge 1)$ be i.i.d. random variables where $U_i$ and $\nu_i$ are independent,  $U_i$ is uniform on $[0,1]$ and $\nu_i$ has distribution $\mu$. We consider the following generalized patience sorting algorithm.
\begin{algopatience}\
\begin{itemize}
\item Start at time $1$, with $\mathbf{R}(1) =1$ and a single tree containing a unique vertex $U_1$ with $\nu_1$ lives. 
\item At time $n$, we have a forest of $\dr(n)$ trees.  To each vertex of the trees is associated a pair 
$(U,\nu)$. The variable $U$ is the label of the vertex whereas $\nu$ is the maximum number of children that the vertex can have. We call number of lives remaining the difference between $\nu$ and its current number of children. A vertex is said to be \emph{alive} if it still has at least one life remaining. 
\item Each tree in the forest has the heap property \emph{i.e.}, along any branch of the tree, the labels $(U_{i_1},U_{i_2},\ldots)$ of the vertices are increasing. 

\item At time $n+1$, we add $U_{n+1}$ as the children of the vertex which is still alive and has the largest label smaller than $U_{n+1}$. If there is no such vertex, we put $U_{n+1}$ at the root of a new tree. The new vertex starts with $\nu_{n+1}$ lives whereas the number of lives of its father (if it exists) is decreased by one.
\end{itemize}
\end{algopatience}

Figure \ref{fig:heaparbre} illustrates the procedure. If $\mu$ is the dirac measure $\delta_1$, we recover the classical patience sorting algorithm into stack. For $\mu = \delta_2$, this correspond to sorting into binary heaps as considered in \cite{Byersetal,IstrateBonchis}. Just like the classical patience algorithm, it is easy to verify that this algorithm is optimal in minimizing number of heaps. In particular, since the algorithm is online, it means that knowing all the $(U_i,\nu_i)$ beforehand cannot help reducing the number of heaps. 

\begin{figure}
\begin{center}
\includegraphics[width=15cm]{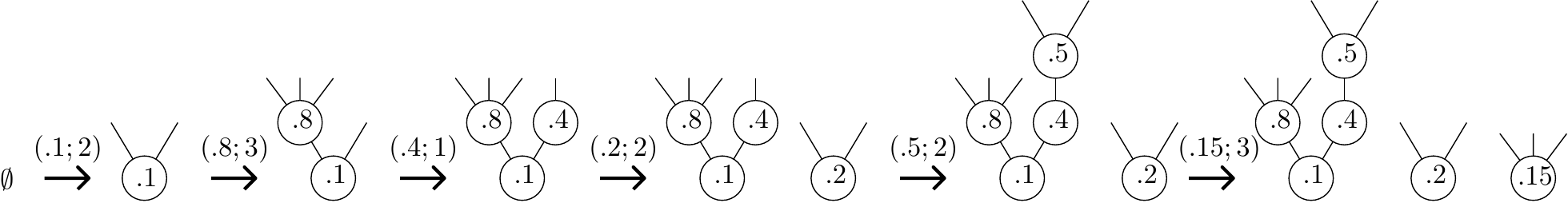}
\caption{ \label{fig:heaparbre} Example of heap patience sorting for the sequence $((.1,2),(.8,3),(.4,1),(.2,2),(.5,2),(.15,3))$. Here we have $ \dr(6) = 3$.
} 
\end{center}
\end{figure}

Our first result states that, as predicted in \cite{IstrateBonchis}, the number $\dr(n)$ of heaps grows logarithmically whenever the trees are non-degenerated.
\begin{theo}\label{theoesperance} 
For every distribution $\mu$ on $\N$ which is not the Dirac mass at $1$, there exists a constant $c_\mu \in (1,\infty)$  such that
$$\E[\dr(n)]\sim c_\mu \log n.$$
\end{theo}

Its turns out that this model is exactly solvable when the trees have a geometric reproduction law. In this case, we obtain a more precise result.  

\begin{theo}\label{theogeometric} Suppose that $\mu$ is a geometric distribution with mean $1/\alpha \in (1,\infty)$. We have
$$
\frac{\dr(n)}{\log n} \underset{n\to\infty}{\longrightarrow} c_\mu = \frac{1}{1-\alpha}
$$
where the convergence holds a.s. and in $L^1$.
\end{theo}

Computing the value of $c_\mu$ for a general reproduction law seems a challenging question and we do not know if there  exists a simple closed expression for this constant. However, thanks to a comparison argument with the geometric case, we can still deduce some bounds on its value in the case of regular trees.
\begin{theo}\label{theoregulier} Suppose that $\mu  = \delta_k$ for some $k\ge 2$. Then
$$1< c_\mu \le \frac{k}{k-1}.$$ 
\end{theo}
Unsurprisingly, the constant converges to $1$ when $k$ goes to infinity which means that for big trees, new heaps are created mostly when encountering new running minimum. Let us also remark that the bounds above are compatible with conjecture \eqref{gdconj} of Istrate and Bonchis \cite{IstrateBonchis} in case of binary trees. Furthermore, numerical simulations indeed indicate a value for the constant relatively close to that of the golden ratio. Yet, we do not believe it to be the correct value. We conjecture that the constant is, in fact, strictly smaller than the golden ratio.
\medskip

In \cite{Hammersley}, Hammersley introduced what is now called \emph{Hammersley's line process} as a way of embedding Ulam's problem in continuous time via a Poissonization procedure. This approach was developed by Aldous and Diaconis \cite{AldousDiaconis95}, providing a probabilistic proof of $\ell(\sigma)\sim 2\sqrt{n}$.
Hammersley's line process turned out to be interesting in its own right. Indeed it is a prime example of an exactly solvable conservative interacting particle system and subsequent works (\emph{e.g.} \cite{Sepp,Sepp2,Groeneboom,CatorGroeneboom,BEGG}) have highlighted the remarkable properties it exhibits.

Here, we follow a similar approach by considering the analogue of Hammersley's process for heap sorting. This gives an alternative representation of the initial problem in term of an \emph{Hammersley's trees process}. The study of this new process is the key to proving Theorem \ref{theoesperance}-\ref{theoregulier}. As already mentioned, the geometric distribution which, in a way, is the most natural extension of Hammersley's line process plays here a special role and deserves a separate study since computation can be carried out to their full extend in this case. Just like Hammersley's line process, these tree processes have some remarkable probabilistic properties that we think are worth exploring, even without reference to the combinatorial problem of counting heaps in a random permutation.

The paper is organized as follows:
\begin{itemize}
\item In Section \ref{Sec:vision}, we formally define this Hammersley's tree process on boxes of $\R^2$. We study the dynamics when the rectangles grow in different directions. This provides several ways of looking at the quantity $\dr(n)$ which are complementary.
\item Section \ref{Sec:GeometricCase} 
is devoted to the geometric case. 
\begin{itemize}
\item[(a)] First, following Cator and Groeneboom \cite{CatorGroeneboom}, we introduce a modification of the process where \emph{sources} and \emph{sinks} are added. This allows us to compute the stationary measures of the process in both directions (Theorems \ref{Th:SourcesStationnaires} and \ref{Th:RacinesStationnaires}).
\item[(b)] The next step consists in defining the stationary process on the whole upper-half plane. We show that the local limit exists as we take boxes that grow to fill the whole space (Theorem \ref{Theo:HalfPlane}). 
\item[(c)] We consider the environment seen from a particle and show that it follows a Poisson Point Process at all time (Proposition \ref{Prop:StationaritePalmGeometrique}). This is the last ingredient to prove the logarithmic growth of $\dr(n)$ in the geometric case.
\end{itemize}
\item In Section \ref{Sec:CasRegulier}, we go back to the case of a general offspring distribution $\nu$. Here, we use a partial coupling argument to compare with the geometric case and complete the proof of Theorems \ref{theoesperance}-\ref{theoregulier}. Once this is done, we can, somewhat remarkably, bootstrap the arguments used for the geometric case to show the existence of the infinite stationary random tree for an arbitrary progeny distribution, even though we cannot compute the stationary distribution explicitly (Theorem \ref{Theo:HalfPlaneGeneral}). 
\end{itemize}

\section{The $\mu$-Hammersley Process and its graphical representations}\label{Sec:vision}
In this section, we describe an alternative way of looking at the problem. 
Easy considerations will then show that $\E[\mathbf{R}(n)]/\log n$ converges to some constant $c_\mu\in(1,\infty]$ as $n$ goes to infinity while still leaving out the really challenging part that is to show $c_\mu<\infty$.

\subsection{Definitions}\label{section-poisson}

We follow the idea of Hammersley \cite{Hammersley}. Consider a Poisson Point Process (PPP)
$$\Xi=(U_i,T_i,\nu_i)_{i\geq 1}$$ 
on $(0,1)\times (0,+\infty) \times \N$ with intensity $du\otimes  dt\otimes \mu$ where $du$, $dt$ are Lebesgue measures and $\mu$ is the progeny distribution of the trees. An atom $(U,T,\nu)$ of $\Xi$ corresponds to a label $U$ with $\nu$ lives arriving at time $T$. We index the atoms of this measure by their time of arrival, \emph{i.e.} such that 
$$0<T_1 <T_2 < \ldots < T_n < \ldots$$

We use the heap sorting algorithm defined in the introduction to sort the sequence $(U_i,\nu_i)$ and we denote by $R(t)$ the number of trees required up to time $t$ (\emph{i.e.} the number of trees needed to sort all the atoms that occur before time $t$). Thus the process $(R(t),t\ge 0)$ is simply a time change of the discrete time process $(\mathbf{R}(n),n\in \N)$ defined in the previous section: 
\begin{equation}\label{eq-discretcontinu}
(R(t),t\ge 0)=(\mathbf{R}(K(t)),t\ge 0),
\end{equation}
 where 
$$
K(t) \defeq |\{i,\, T_i \leq t\}|
$$ 
is the Poisson process counting the number of atoms of $\Xi$ before time $t$. Notice that the process $K$ is independent of $(\mathbf{R}(n),n\in \N)$. Therefore, estimates on $R$ can easily be transferred on $\mathbf{R}$ and \emph{vice-versa}.
We will deal mostly with $R(t)$ and Theorem \ref{theoesperance} - \ref{theoregulier} will follow easily once we have proved the following result concerning the continuous time process $R$:
\begin{theo}\label{theocontinu} 
For every distribution $\mu$ on $\N$ which is not the Dirac mass in $1$, there exists a constant $c_\mu \in (1,\infty)$  such that
$$\lim_{n\to \infty}\frac{\E[\mathbf{R}(n)]}{\log n}=\lim_{t\to \infty}\frac{\E[R(t)]}{\log t}=c_\mu.$$
\begin{itemize}
\item[(i)]   If $\mu$ is a geometric distribution with mean $1/\alpha$, then $c_\mu=1/(1-\alpha)$ and $\frac{R(t)}{\log t}$ converges a.s. and in $L^1$ to $1/(1-\alpha)$.
\item[(ii)]  If $\mu$ is the Dirac mass in $k\ge 2$,  then
 $1< c_\mu \le \frac{k}{k-1}$. 
\end{itemize}
\end{theo}
In order to study the process $R(t)$, we need to keep track of all the labels of the vertices of the trees together with  their remaining lives at any given time $t$. To do so, we introduce a particle system $\bar{H}$ with the following properties. At any time $t$, the set of particles on the system is the set of atoms in $\Xi$ that occurred before $t$ and the positions of the particles are given by the labels $U$ of the atoms. Moreover, each particle is also assigned a number of lives that  corresponds to the number of lives remaining in the heap sorting procedure. Formally, we can represent $\bar{H}$ as a measure-valued process on $(0,1)\times \{0,1,\ldots\}$ 
$$\bar{H}(t)\defeq\sum_j \delta_{(u_j(t),\bar{\nu}_j(t))}$$
where the $u_j(t)$ are the labels of the vertices of the trees that appeared before time $t$ and $\bar{\nu}_j(t)\ge 0$ are the corresponding number of lives at time $t$. Here we use the notation $\bar{\nu}$ instead of $\nu$ to distinguish between its current number of lives and the initial number of lives $\nu$ it had when it appeared. Let us stress that $\bar{H}$ keeps track of the particles with $0$ life. We say that such particles are \emph{dead} whereas particles with at least one life remaining are called \emph{alive}. Dead particles cannot interact in any way (the corresponding vertex is full) hence they can be discarded without affecting the evolution of the other particles. We will denote by $H$ the process where we keep only the particles alive:
$$
H(t)\defeq\sum_{j,\, \bar{\nu}_j(t) > 0} \delta_{(u_j(t),\bar{\nu}_j(t))}.
$$
We call this process (and by extension also $\bar{H}$) the \emph{$\mu$-Hammersley} process or \emph{Hammersley's Tree process}. By convention, for each time $t$, we index the particles in $H(t)$ (resp. $\bar{H}(t)$) by increasing order of their positions \emph{i.e.} such that $u_j(t)<u_{j+1}(t)$. We point out that this is not the order of their arrival time so the index of a particle located a given position may change in time. Looking back at the sorting algorithm, the following proposition is straightforward.
\begin{prop}\label{Prop:transitionH}
The process $\bar{H}$ is a Markov process with initial state and transition kernel given by:
\begin{itemize}
\item There is no particle at time $0$ \emph{i.e.} $\bar{H}(0) = 0$
\item Given $\bar{H}(t^-)$, an atom $(u,t,\nu)$ in $\Xi$ creates in $\bar{H}(t)$ a new particle at position $u$ with $\nu$ lives. Furthermore, the particle alive in $H(t^-)$ with the largest label smaller than $u$ loses one life (if such a particle exists).
\end{itemize}
The same statement holds with $H$ in place of $\bar{H}$.
\end{prop}
When there is no particle in $H(t^-)$ on the left of an atom $(u,t,\nu)$, we say that a \emph{root} is created at time $t$ since it corresponds to creating a new tree in the heap sorting algorithm. With this terminology, the number of heaps $R(t)$ is exactly the number of roots created by $H$ up to time $t$.

\begin{figure}
\begin{center}
\includegraphics[height=5.5cm]{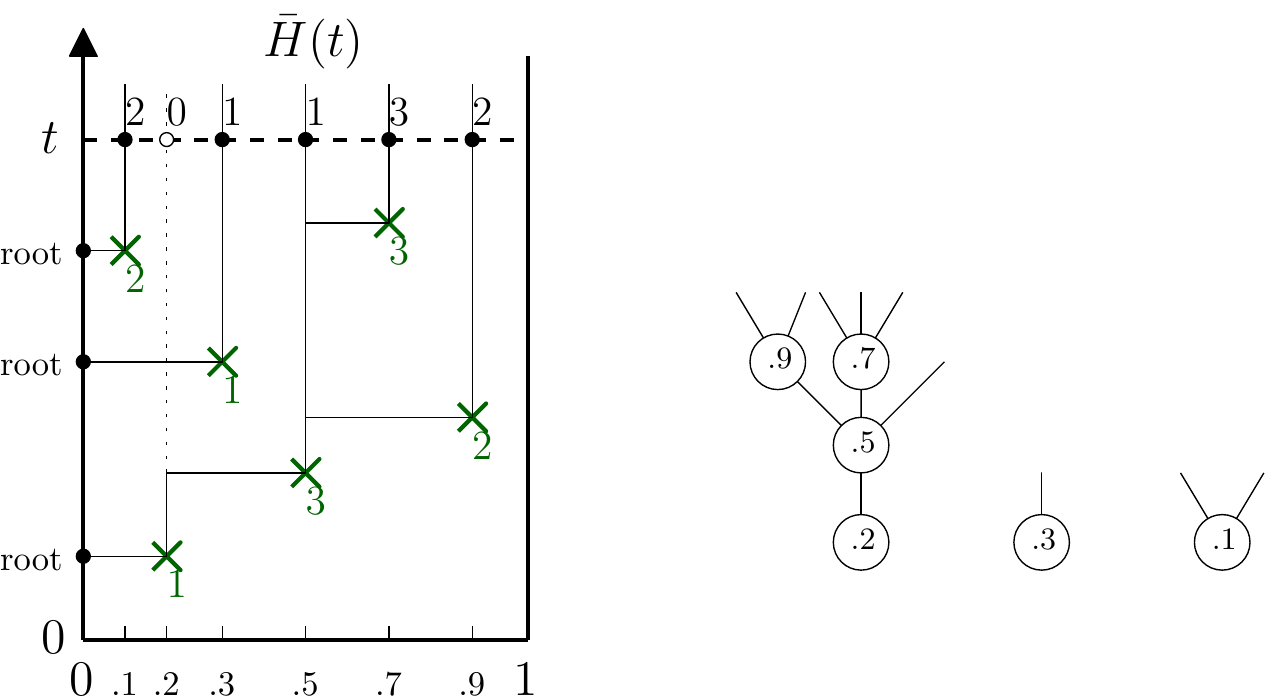}
\caption{ \label{fig:defHammersley} On the left, the graphical representation of $\bar{H}$. Crosses in green represent the atoms of $\Xi$ with their initial number of lives written below. Dead particles are signified with dotted lines. In this example, the process $\bar{H}$ at time $t$ has three roots ($R(t) = 3$) and consists of $6$ particles at positions $.1, .2, .3, .5, .7, .9$ with respective remaining lives a time $t$ equal to $2,0,1,1,3,2$. On the right, the corresponding trees obtained by the heap patience sorting algorithm.}
\end{center}
\end{figure}

Let us note that in the particular case where $\mu=\delta_1$, every particle has exactly one life when it is born and $H$ is the classical Hammersley's line process where particles appear according to a PPP of unit intensity and each new particle kills the particle immediately on its left. In our setting, the difference is that particles may be born with more than one life so they do not necessarily die when a new particle appears on their right.

Just as in the case of the usual Hammersley process, there is a very useful graphical representation of the genealogy of the particles in $H$ in term of broken lines. Consider the strip $[0,1]\times[0,\infty)$. We draw a cross at each position $(u,t)$ corresponding to an atom of $\Xi$. These crosses mark the position and time at which a new particle appears. We represent the genealogy of the particles using vertical and horizontal lines:
\begin{itemize}
\item Vertical lines mark the positions of particles through time. 
\item Horizontal line connects each newly arrived particle to its father or to the vertical axis in case of a new root.
\end{itemize}
In the case of Hammersley's line process, this representation yields a set of non intersecting broken lines going up and right. In our setting, we get a set of non-intersecting trees, each tree starting from the left boundary of the box and going right and upward. When dealing with $\bar{H}$, it is convenient to represent also the position of the dead particles. We do so by drawing position of dead particles through time using dotted lines. See Figure \ref{fig:defHammersley} for an example.

\subsection{Markov property for the number of lives process}\label{sect-augmentationhaut}

In \cite{IstrateBonchis}, Istrate and Bonchis made the observation that, if we forget the exact position of the particles in $\bar{H}$, keeping track only their relative order together with their number of lives, then the resulting process is still Markov. More precisely, recall that $T_0<T_1<\ldots$ denote the times of arrival of a new particle and define 
$$\mathcal{V}_n=(\bar{\nu}_j(T_n))_{1\le j\le n}$$
\emph{i.e.} $\mathcal{V}_n$ is the vector containing the number of lives of the $n$ first particles at the time when the $n$-th particle appears, ordered by increasing value of their labels. Then, $(\mathcal{V}_n, n\geq 0)$ is a Markov process on finite words with alphabet $\N\cup\{0\}$ and with the following properties.
\begin{itemize}
\item $\mathcal{V}_0 = \emptyset$ and $\mathcal{V}_n \in (\N\cup\{0\})^n$.
\item Conditionally on $\mathcal{V}_n = (\bar{\nu}_1,\ldots,\bar{\nu}_n)$, the law of $\mathcal{V}_{n+1}$ is given by the following construction: pick an index $j\in \{0,\ldots, n\}$ with uniform distribution
and insert between $\bar{\nu}_j$ and $\bar{\nu}_{j+1}$ a new random variable sampled with law $\mu$ (cases $j=0$ and $j=n$ correspond to inserting at both extremities of the word). Then, decrement the number of lives of the particle still alive with the largest index smaller than $j$ (if such an index exists). 
\end{itemize}

\begin{rem} In the construction $\mathcal{V}$, it is essential to keep track of the dead particles otherwise the Markov property is lost. 
\end{rem}

Consider the quantity
\begin{equation}\label{defD_n}
D_n \defeq\sup\{j, \bar{\nu}_j(T_n)=0\}\qquad\hbox{}
\end{equation}
with the convention $\sup\emptyset = 0$. Hence, $D_n$ is the number of $0$ at the beginning of $\mathcal{V}_n$. According to the dynamic of $\mathcal{V}$, the process $\bar{H}$ creates a new root each and every time the uniform index  $j$ chosen for the insertion of the new element in the word is such that $0 \leq j \leq D_n$. Indeed, this means that the new vertex cannot attach itself to a previously existing one, hence starts a new tree. Therefore, we obtain the simple recursion formula.
\begin{equation} \label{recDn}
 \E[\mathbf{R}(n+1)]= \E[\mathbf{R}(n)]+\frac{\E[D_n+1]}{n+1}.
 \end{equation}
 
\subsection{Growing the rectangle in different directions}

In our original definition of the $\mu$-Hammersley process, the PPP $\Xi$ is defined on the set $(0,1) \times (0,\infty)\times \N$ therefore the graphical representation is defined inside the strip $[0,1]\times [0,\infty)$. More generally, it is convenient to allow labels to be indexed by $\R$ instead of $(0,1)$. This enables us to define the $\mu$-Hammersley process on larger boxes. Thus, we extend the PPP $\Xi$ to $\R \times (0,\infty) \times \N$, still with intensity $du \otimes dt \otimes \mu$. For any $a<b$, we define the  $\mu$-Hammersley  $(H_{[a,b]}(t), t\ge 0)$ (resp.  $(\bar{H}_{[a,b]}(t), t\ge 0)$) obtained by considering only the atoms of $\Xi$ whose labels lie in $(a,b)$. Thus, the genealogy of the process is given by the aforementioned graphical representation inside the strip $[a,b]\times[0,\infty)$. We denote by $R_{[a,b]}(t)$ the associated number of roots up to time $t$.

It is clear that the mapping
$$\begin{array}{ccc}
(a,b)\times (0,T)\times \N & \to& (0,1) \times (0, (b-a)T)\times \N\\
(u,t,\nu) & \mapsto & (\frac{u-a}{b-a}, (b-a)t,\nu)
\end{array}$$
preserves $\Xi$. Hence, the same is also true for the $\mu$-Hammersley process. This means that we can map the graphical representation inside boxes of similar volume by dilation of space and inverse dilatation of time. In particular, the number of roots depends only on the volume which implies the equality 
$$R_{[a,b]}(t)\overset{\hbox{\tiny{law}}}{=}R((b-a)t).$$

As we already remarked, the entire history of the particle  system up to time $t$ is encoded by the graphical representation inside the box $[a,b]\times[0,t]$. Thus looking at the process evolving in time corresponds to growing this box from the top. Since  the heap sorting algorithm is online,  the graphical representation is compatible with this operation: for $t < t'$, the graphical representation inside $[a,b]\times[0,t]$ is the trace of the graphical representation inside the larger box $[a,b]\times[0,t']$. In the following subsection, we will instead consider the process obtained as the box grows from its right (resp. left) boundary. This corresponds to inserting atoms with larger (resp. smaller) values and looking at  how they affect the sorting procedure.

\subsubsection{Growing to the right. Second class particle}\label{sect-augmentationdroite}

We first look at the dynamic of the $\mu$-Hammersley  $\bar{H}_{[0,x]}(t)$ when $t$ is fixed and  $x$ increases, \emph{i.e.} when the atoms of $\Xi$ are taken into account from left to right. Let $((U_i,T_i,\nu_i),i\ge 1)$ denote the atoms of $\Xi$ inside the strip $(0,\infty)\times \N \times (0,t)$ indexed by increasing order of their first coordinate:
  $$U_1<U_2<\ldots$$ 
The $U_i$'s should now be seen as the ``times'' of jumps of the process  $\bar{H}_{[0,\cdot]}(t)$. Let us note that this process is not Markov since  the arrival of a new largest label can alter the existing trees in a not trivial manner. This means that the graphical representation inside boxes of the form $[0,x]\times[0,t]$ and $[0,x']\times[0,t]$ are not compatible any more.  However, the modification required to insert a new label can be understood with the help a second class particle defined as follow:

\begin{itemize}
\item the second class particle starts from the position of the new label.
\item It moves horizontally to the left until either reaching the left side of the box or reaching a solid vertical line of the $\mu$-Hammersley process.
\item  Upon reaching an Hammersley line, the particle starts moving upwards until it either reaches the top of the box or until the line becomes dotted. In the latter case, it starts again moving to the left and the procedure above is repeated.
\end{itemize}

\begin{figure}
\begin{center}
\includegraphics[height=5.5cm]{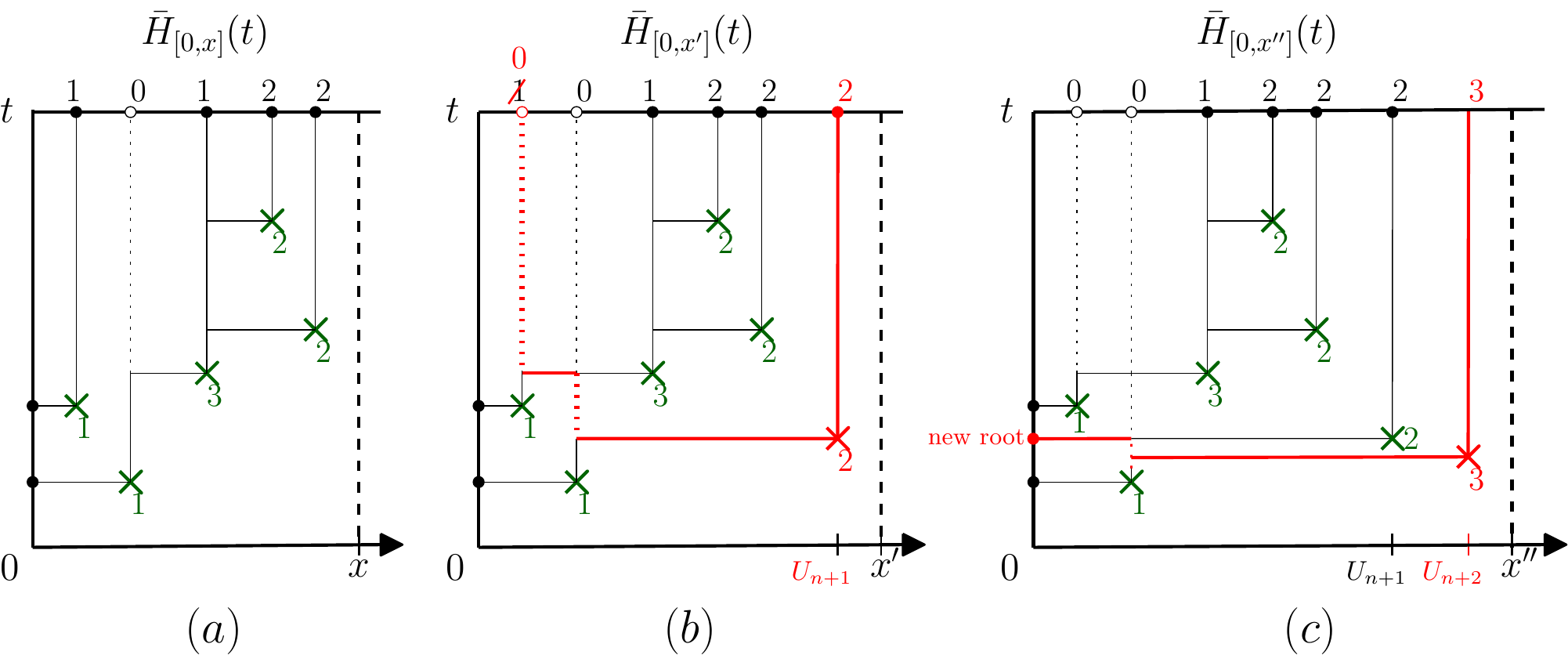}
\caption{
Effect of adding new particles with maximal label. The trajectory of the second class particles are displayed in red. $(a)$ Initial configuration. $(b)$ The particle with label $U_{n+1}$ settles inside the existing trees. $(c)$ Extending again to the right, the particle $U_{n+2}$  creates a new tree. }
 \label{fig:Dncroissant}
\end{center}
\end{figure}

Figure \ref{fig:Dncroissant} illustrates this procedure. It is easy to convince oneself that the second class particle encodes the modification required to insert the new label: when it first intersects a vertical line, the new label attached itself as a child of the corresponding particle. If this particle still has a life at time $t$ in the original process, no other modification on the tree are required (this means that the second class particle goes up to the top of the box). In the other case, the atom of $\Xi$ that previously attached last to this particle now needs to find a new father. Hence, the second class particle restarts from this atom and goes left looking for an ancestor.

This description has the following important consequence. If we look at the number of lives of all the particles at time $t$, then inserting a new atom with largest label, can alter  at most one living particle by removing one of its lives. Furthermore, a new tree is created i.f.f. the second class particle reaches the left side of the box in which case the respective number of lives of all previously existing particles is undisturbed. 

Recall that we defined $D_n$ in the previous section as the number of dead particles before finding a particle alive  at the time when the system has exactly $n$ particles when we look at the particles from left to right.  This was done with the box growing from top. We can do the same thing but now growing the box from its right side. We call $(D'_n, n\ge 0)$ this  process. The scaling invariance of the PPP implies that for each fixed $n$, the marginals $D_n$ and $D'_n$ have the same law. Furthermore, the previous remark concerning the second class particle shows that, contrarily to $D_n$, the process $D'_n$ is non-decreasing in $n$. This insures that the limit 
\begin{equation}\label{limDn}
\lim_{n\to \infty}\E[D_n+1] = \lim_{n\to \infty}\E[D'_n+1]=c_\mu
\end{equation}
exists for some $c_\mu \in (1,\infty]$. We can now make a first step towards  Theorem \ref{theoesperance}.

\begin{lem}\label{lem:cvesperance} For any distribution $\mu$ on $\N$, there exists a constant $c_\mu$ (possibly infinite) such that
$$\lim_{n\to \infty} \frac{\E[\mathbf{R}(n)]}{\log n}=\lim_{t\to \infty} \frac{\E[R(t)]}{\log t}=c_\mu.$$
\end{lem}

\begin{proof}Combining  Equations \eqref{recDn} and \eqref{limDn}, we get that 
 $$\lim_{n\to \infty} \frac{\E[\mathbf{R}(n)]}{\log n}=c_\mu.$$
Moreover, $\E[D_n]$ being non-decreasing,  $\E[\mathbf{R}(n)]\le c_\mu (1+\log n)$.
Recall also that $R(t)=\mathbf{R}(K(t))$ where $K(t)$ is an independent Poisson random variable of parameter $t$. Thus,
\begin{equation*}
 \E[R(t)]=\E\big[\E[\mathbf{R}(K(t))|K(t)]\big] \le c_\mu \E[1+\log K(t)] \le c_\mu(1+\log \E[K(t)]) =c_\mu(1+ \log t).
 \end{equation*}
On the other side, for $c<c_\mu$, let $N$ such that  for $n\ge  N$,  $\E[\mathbf{R}(n)]\ge c\log n$. Then, for $t\ge 2N$, we have
\begin{multline*}
\E[R(t)] \ge \E\big[\E[\mathbf{R}(K(t))|K(t)]\mathbf{1}_{K(t)\ge  t/2}\big] \ge c_\mu\E\big[\log(K(t))\mathbf{1}_{K(t)\ge  t/2}\big] \\ \ge c_\mu\log (t/2)\P\{K(t)\ge  t/2\}.
\end{multline*}
\end{proof}

\subsubsection{Growing to the left. The root process} \label{section-rootprocess}
We now consider the dynamic of the $\mu$-Hammersley process $H_{[-x,0]}(t)$ (or equivalently $\bar{H}_{[-x,0]}(t)$) as $x$ increases and $t$ is kept fixed, \emph{i.e.} when the atoms of $\Xi$ are taken into account from right to left. This corresponds to studying the evolution of the graphical representation inside the box $[-x,0]\times[0,t]$ as it grows by its left boundary. Let $((-U_i,T_i,\nu_i),i\ge 1)$ denote the atoms of $\Xi$ inside the strip $(-\infty,0) \times (0,t)\times \N$ indexed by decreasing order of their first coordinate:
  $$ \ldots < -U_2 < -U_1$$ 
Again, the $U_i$'s should be seen as the ``times'' of jumps of the process  $H_{[\cdot,0]}(t)$. The dynamic here is much simpler than when growing to the right. Indeed, the addition of a new label $U$ with minimal value may result in a merging of trees but does not disturb the respective sorting of atoms with larger labels. Alternatively, this means that the graphical representation possesses the compatibility property: for $x'>x$, the trace of the representation inside a box $[-x',0]\times[0,t]$ restricted to the box $[-x,0]\times[0,t]$ is the same as that obtained by considering the graphical representation inside the smaller box. Thus, in order to understand the evolution of the process, we must understand the evolution of the trace of the graphical representation on the left boundary of the box \emph{i.e.} we must understand how the set of roots evolves as time goes backward. More precisely, for each $x>0$, let $r_1^x<r_2^x<\ldots$ denote the vertical coordinate of the roots of the graphical representation inside $[-x,0]\times[0,t]$ (it is a subset of $\{T_i \leq t\hbox{ s.t.} U_i <x\}$). The \emph{root process} $(\mathcal{R}(x,t),x\ge 0)$ is formally defined by
$$\mathcal{R}(x,t)\defeq\sum_{r_i^x<t}\delta_{r_i^x}.$$
Let us note that this definition is consistent for different values of $t$, \emph{i.e.} for $t<t'$, $(\mathcal{R}(x,t),x\ge 0)$ is the restriction of $(\mathcal{R}(x,t'),x\ge 0)$ to the interval $[0,t]$. Since we are looking at the root process as $x$ varies, we say, somewhat awkwardly, that there is a \emph{root} at time $x$ located at height $s$ if $s=r_i^x$ for some $i$. The following proposition that describes the dynamic of $\mathcal{R}$ follows directly from the construction of the $\mu$-Hammersley process. 
\begin{figure}
\begin{center}
\includegraphics[height=5cm]{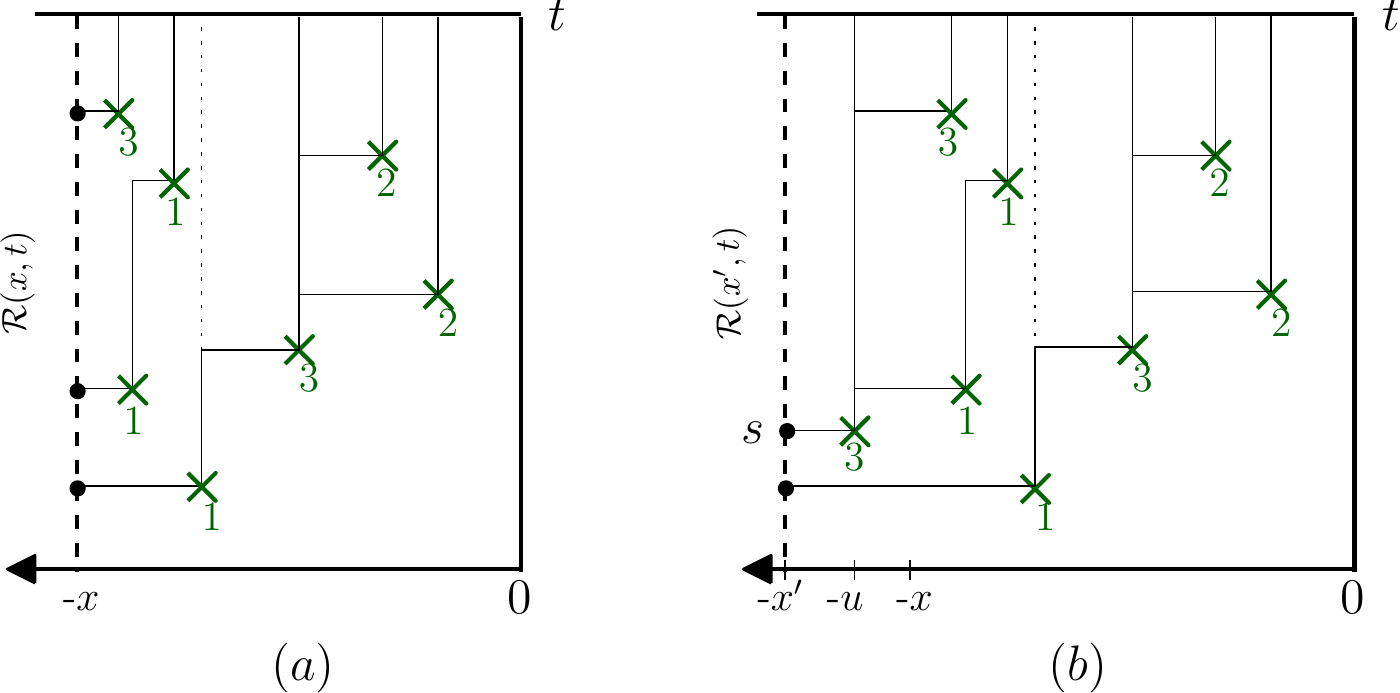}
\caption{Dynamic of the root process. $(a)$ There are three roots at time $-x$. $(b)$ The box grows to the left. The atom $(-u,s,3)$ creates a new root at time $u$ with height $s$ and tries to remove three roots above it (here, only two are removed).\label{fig:rootprocess}}
\end{center}
\end{figure}
\begin{prop}\label{Prop:transitionR}
For any fixed $t$, the root process $(\mathcal{R}(x,t),x\ge 0)$ is a Markov process with initial state and transition kernel given by
\begin{itemize}
\item There is no root at time $0$ \emph{i.e} $\mathcal{R}(0,t) = 0$.
\item Given $\mathcal{R}(u^-,t)$, an atom $(-u,s,\nu)\in \Xi$ with $s<t$ creates, at time $u$, a new root with height $s$ while removing from $\mathcal{R}$ the $\nu$ lowest particles with height greater than $s$ (or all of them if there are less than $\nu$).
\end{itemize} 
\end{prop}
Figure \ref{fig:rootprocess} illustrate the dynamic of the root process. It follows from the scaling property of Poisson measures that, for any $x,t>0$ we have the correspondence
$$
|\mathcal{R}(x,t)| \overset{\hbox{\tiny{law}}}{=} R(xt).
$$
For the sake of brevity, we shall use the notation $\mathcal{R}(x)$ for $\mathcal{R}(x,1)$.
\begin{rem}
Comparing Proposition \ref{Prop:transitionH} and Proposition \ref{Prop:transitionR}, we see that, in the  case of the classical Hammersley process $\mu =\delta_1$, the two processes $\mathcal{R}$ and $H$ have same law. This is a consequence of the self-duality property of the model  
\end{rem}

The next result looks technical but will	allow us to compare the number of trees of the graphical representation inside a box with the number of trees of the graphical representation inside more general domains.

\begin{prop}\label{Prop:couplagesousdomaine} Let $f : [0,1] \to [0,+\infty]$ be a non-decreasing function. Set
$$a_f \defeq \inf(x,\; f(x) = +\infty).$$ 
Define the domain $\mathcal{D}_f  = \{ (x,t) \in (0,a_f)\times (0,+\infty),\; t > f(x)\}$  and the mapping
\begin{equation*}
F : \begin{array}{ccc}
\mathcal{D}_f \times \mathbb{N} &\to& (0,a_f)\times(0,+\infty) \times \mathbb{N}\\
(x,t,\nu) &\mapsto& (x,t-f(x),\nu).
\end{array}
\end{equation*}
Let $\xi$ be a PPP on $\mathcal{D}_f \times \mathbb{N}$ with intensity $dx\otimes dt\otimes \nu$. The image $F(\xi)$ is a PPP on $(0,a_f)\times(0,+\infty)$ with intensity $dx\otimes dt\otimes \nu$. Given a finite set of atom $A$, we denote by $R(A)$ the number of trees of the graphical representation using the atom of $A$. Define $B_t \defeq [0,1]\times[0,t]$. We have, almost surely,
$$
R(\xi \cap B_t) \;\leq\; R(F(\xi \cap B_t)) \;\leq\; R(F(\xi) \cap B_t) \quad \hbox{for all $t\geq 0$.}
$$
\end{prop}

\begin{figure}
\begin{center}
\includegraphics[height=7cm]{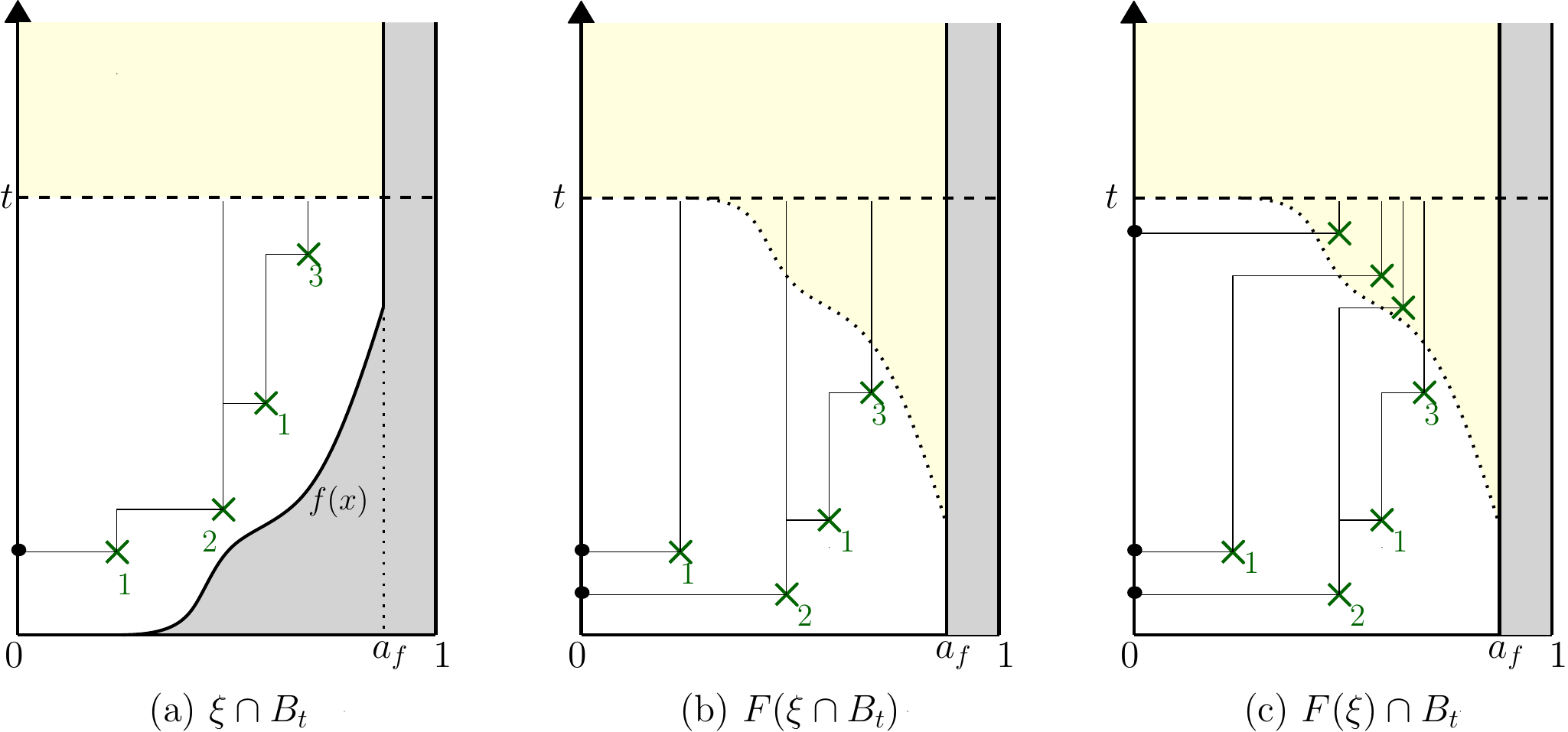}
\caption{\label{fig:falling}Illustration of Proposition \ref{Prop:couplagesousdomaine}. (a) $R(\xi \cap B_t) = 1$, (b) $R(F(\xi \cap B_t)) = 2$, (c) $R(F(\xi) \cap B_t) = 3$.} 
\end{center}
\end{figure}

\begin{rem}
\begin{itemize}
\item Graphically, the mapping $F$ corresponds to removing the hypograph of $f$  (greyed out in Figure \ref{fig:falling}) and then letting what is over it ``falls down to the ground''.
\item We need to transform the point process via the mapping $F$ because we cannot directly compare the number of trees between a domain and a sub-domain using the restriction of the same PPP. Indeed, the number of trees created by a set of atom in $[0,a_f]\times[0,t]$ may be smaller than the number of trees created when keeping only those atoms located in the sub-domain $\mathcal{D}_f \cap [0,a_f]\times[0,t]$.
\item The result of the proposition is, in fact, deterministic. For simplicity, we stated it only for a PPP with Lebesgue intensity which insures that no two atoms are aligned in either configurations (otherwise the graphical representation is not well defined).  
\end{itemize}
\end{rem}

\begin{proof}
First, the fact that $F(\xi)$ is a PPP on $(0,a_f)\times(0,+\infty)$ with intensity $dx\otimes dt\otimes \nu$ is easy so we omit the details. Second, the inequality
$$
R(F(\xi \cap B_t)) \;\leq\; R(F(\xi) \cap B_t)
$$
is also straightforward since atoms in $(F(\xi) \cap B_t) \backslash F(\xi \cap B_t)$ cannot be ancestor of atoms in $F(\xi \cap B_t)$ hence adding those atoms can only increase the number of trees in the graphical representation. 
It remains to prove the first inequality. The idea is to construct simultaneously the graphical representations for $\xi \cap B_t$ and $F(\xi \cap B_t)$ using their root processes. Thus, we start from an empty set of root on the left border $\{1\}\times [0,t]$ of $B_t$ and consider the increasing boxes $B_t(u) \defeq [1-u,1]\times [0,t]$ as $u$ goes from $0$ to $1$. Doing so for both point processes yields two root processes $\mathcal{R}^1(u,t)$ (for $\xi \cap B_t$) and $\mathcal{R}^2(u,t)$ (for $F(\xi \cap B_t)$). For $0\le a <b\le t$, let $N^i_{[a,b]}(u)$ denote the number of particles of  $\mathcal{R}^i(u,t)$ in the interval $[a,b]$. Since $u$ is running backward, let also use the notation $\tilde{f}(u)\defeq f(1-u)$, so that $\tilde{f}$ is non-increasing.
We now prove that, for all $u\in [0,1]$, we have
\begin{equation}\label{Eq: recurrencedomination}
\forall s\in [0,t-\tilde{f}(u)], \quad
N^1_{[0,s+\tilde{f}(u)]}(u) \le N^2_{[0,s]}(u),
\end{equation}
from which we will conclude that
$$
R(\xi \cap B_t) = N^1_{[0,t]}(1) \;\leq\; N^2_{[0,t-\tilde{f}(1)]}(1) 
\;\leq\; N^2_{[0,t]}(1) = R(F(\xi \cap B_t)).
$$
Note that \eqref{Eq: recurrencedomination} clearly holds for $u\leq 1- a_f$. Moreover, since $\tilde{f}$ is non-increasing, the first time \eqref{Eq: recurrencedomination}  fails must necessarily happen at a time where an atom of $\xi$ is encountered. Hence, by induction, we just need to prove that if \eqref{Eq: recurrencedomination} holds at $u^-$ and $\xi$ has an atom  at $(1-u,r,\nu)$, then \eqref{Eq: recurrencedomination} also holds at $u$. Since $\tilde{f}$ is continuous almost everywhere and $\xi$ is a PPP with Lebesgue spacial intensity, we can assume without loss of generality that $\tilde{f}$ is continuous at $u$ (in fact, this continuity assumption simplifies the proof slightly but is not required). Recall that an atom of $\xi$ located at $(1-u,r,\nu)$ corresponds to the atom $(1-u,r-\tilde{f}(u),\nu)$ in $F(\xi)$. Recall also that an atom at $(1-u,r,\nu)$ creates a new root with height $r$ (at time $u$) in the root process while removing the first $\nu$ particles above $r$ (if they exist). 

Thus, assume that $(1-u,r,\nu) \in \xi$ and that \eqref{Eq: recurrencedomination} holds up to time $u^-$. For $s<r-\tilde{f}(u)$, we have 
$$N^1_{[0,s+\tilde{f}(u)]}(u)=N^1_{[0,s+\tilde{f}(u)]}(u^-) \quad \mbox{ and }\quad N^2_{[0,s]}(u)=N^1_{[0,s]}(u^-),$$
so \eqref{Eq: recurrencedomination} holds at time $u$ for all $s<r-\tilde{f}(u)$. Let now $s\ge r-\tilde{f}(u)$. There are two cases to consider. Either  
$N^1_{[r,s+\tilde{f}(u)]}(u^-)\le N^2_{[r-\tilde{f}(u),s]}(u^-)$ in which case we use that
\begin{equation*}
N^1_{[0,s+\tilde{f}(u)]}(u) =  N^1_{[0,r]}(u^-)+1+ (N^1_{[r,s+\tilde{f}(u)]}(u^-)- \nu)^+
\end{equation*}
and 
\begin{equation*}
N^2_{[0,s]}(u) = N^2_{[0,r-\tilde{f}(u)]}(u^-)+1+ (N^2_{[r-\tilde{f}(u),s]}(u^-)- \nu)^+
\end{equation*}
showing that \eqref{Eq: recurrencedomination} holds at time $u$. Otherwise, we have $N^1_{[r,s+\tilde{f}(u)]}(u^-) > N^2_{[r-\tilde{f}(u),s]}(u^-)$ and we use instead that
\begin{equation*}
N^1_{[0,s+\tilde{f}(u)]}(u) =  N^1_{[0,s+\tilde{f}(u)]}(u^-)+1- N^1_{[r,s+\tilde{f}(u)]}(u^-)\wedge \nu
\end{equation*}
and
\begin{equation*}
N^2_{[0,s]}(u) = N^2_{[0,s]}(u^-)+1- N^1_{[r-\tilde{f}(u),s]}(u^-)\wedge \nu
\end{equation*}
showing again that \eqref{Eq: recurrencedomination} holds. 
\end{proof}

\subsection{Hammersley process and root process with sources and sinks}\label{sect:sourcesandsinks}
As we have seen in Sections \ref{section-poisson} and \ref{section-rootprocess}, the genealogies of the trees are consistent when we discover the atoms of $\Xi$ from bottom to top or from right to left. Hence the graphical representation inside a box $[a,b]\times [s,t]$ restricted to a sub-box on the form $[a',b]\times [s,t']$ for $a<a'$ and $t'<t$ coincides with that obtained by considering only the atoms inside the smaller box. This means that we can directly construct the graphical representation on a whole quarter plane $(-\infty,b]\times[0,\infty)$. However, as pointed out in Section \ref{sect-augmentationdroite}, the restriction property fails when $b$ increases \emph{i.e.} when the atoms of $\Xi$ are revealed from left to right since a new atom can alter the previously constructed trees. 

We aim to construct an infinite volume measure associated with the $\mu$-Hammersley process on the whole half plane $(-\infty,+\infty)\times[0,\infty)$. To do so, we shall need some kind of compatibility for the law of the process inside different sub-boxes in order to apply Kolmogorov's extension Theorem. This means, in particular, that we must alter the graphical representation inside a finite box in such way that its lines can now reach the bottom and the right boundary of the box (since this will necessarily happen when taking the trace from a larger box). This idea was carried out by Cator and Groeneboom in \cite{CatorGroeneboom} for the classical Hammersley process via the introduction of \emph{sources} and \emph{sinks} processes along, respectively, the bottom and right boundary of the box. 

Recall that $H$ is the $\mu$-Hammersley process constructed from the atom of $\Xi$ laying inside $(0,1)\times (0,\infty)$. Let 
$$\mathcal{C}=\{(u_i,\nu_i)\}_{1 \leq i\leq N}$$ 
be a finite set in $(0,1)\times \N$ and let 
$$\mathcal{S}=(s_i)_{i\ge 1}$$
 be a positive increasing sequence. We define the $\mu$-Hammersley process $(H^{\mathcal{C},\mathcal{S}}(t),t\geq 0)$  with \emph{sources} $\mathcal{C}$ and \emph{sinks} $\mathcal{S}$ as the Markov measure-valued process on $(0,1)\times \N$ defined by;
\begin{figure}
\begin{center}
\includegraphics[height=6cm]{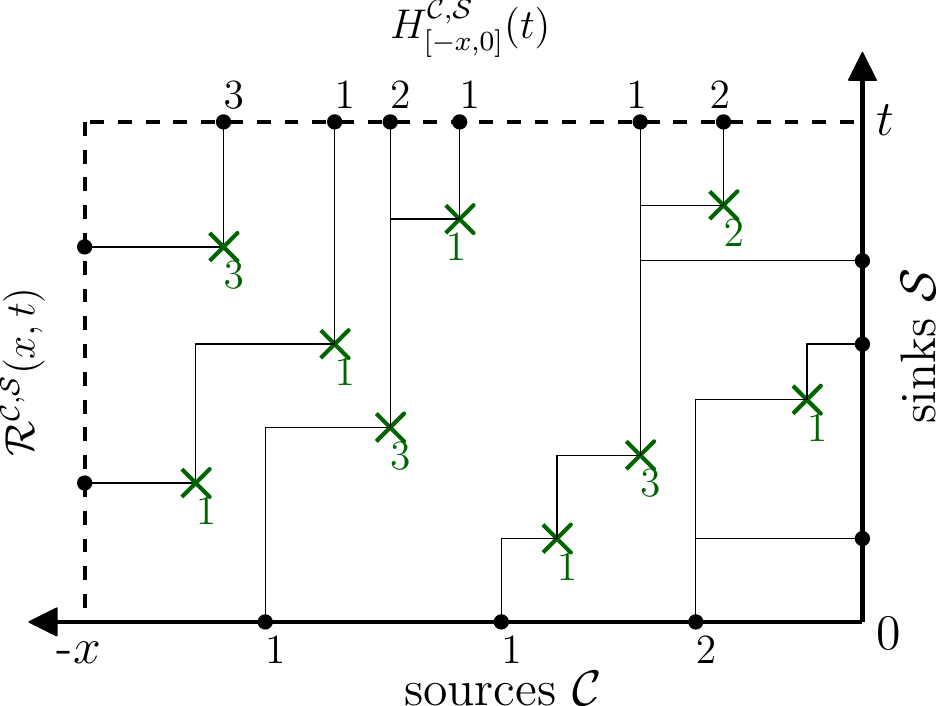}
\caption{Graphical representation of the $\mu$-Hammersley process $H^{\mathcal{C},\mathcal{S}}_{[-x,0]}$ and the root process $\mathcal{R}^{\mathcal{C},\mathcal{S}}(\cdot,t)$ with sources $\mathcal{C}$ and sinks $\mathcal{S}$ inside the box $[-x,0]\times[0,t]$. \label{fig:DefSourcesSinks} }
\end{center}
\end{figure}
\begin{itemize}
\item The process starts at time $0$ from the configuration of sources \emph{i.e.} $H^{\mathcal{C},\mathcal{S}}(0) =\mathcal{C}$.
\item Similarly to $H(t)$, an atom $(u,t,\nu)$ in $\Xi$ with $u\in (0,1)$  removes a life from the nearest alive particle on its left in $H^{\mathcal{C},\mathcal{S}}(t^-)$ (if any), and creates in $H^{\mathcal{C},\mathcal{S}}(t)$ a new particle at position $u$ with $\nu$ lives.
\item At times $s_i$, we remove a life from the largest alive particle in $H^{\mathcal{C},\mathcal{S}}(s_i^-)$. 
\end{itemize}
When both $\mathcal{C}$ and $\mathcal{S}$ are empty, we recover the initial $\mu$-Hammersley thus $H = H^{\emptyset,\emptyset}$. In general, one should think of the sources as the configuration resulting from atoms below the lower boundary of the box and the sinks as mimicking the effects of the atoms on the right of the box.

We can do a similar construction for the root process. In this case,  the sources $\mathcal{C}=(-u_i,\nu_i)_{i\ge 1}$  form a sequence in $(-\infty,0)\times \N$ with $u_i<u_{i+1}$ and the sinks $\mathcal{S}=\{s_i\}_{i\ge 1}$ are a set of points in $(0,1)$. We define the root process $(\mathcal{R}^{\mathcal{C},\mathcal{S}}(x))_{x\geq 0}$  with sources $\mathcal{C}$ and sinks $\mathcal{S}$ as the Markov measure-valued process on $(0,1)$ defined by:
 \begin{itemize}
 \item The process starts from the configuration given by the sinks \emph{i.e.} $\mathcal{R}^{\mathcal{C},\mathcal{S}}(0) =\mathcal{S}$.
\item Similarly to $\mathcal{R}(x)$, an atom $(-u,s,\nu)\in \Xi$ with $s<1$ creates, at time $u$, a new root with height $s$ while removing from $\mathcal{R}^{\mathcal{C},\mathcal{S}}(u)$ the $\nu$ lowest particles with height greater than $s$ (all of them if there are less than $\nu$).
\item For each source $(-u,\nu)\in \mathcal{C}$, we remove at time $u$ the $\nu$ lowest particles from $\mathcal{R}^{\mathcal{C},\mathcal{S}}(u)$  (all of them if there are less than $\nu$).
 \end{itemize}
Just as for the $\mu$-Hammersley process $H$ (resp. the root process $ \mathcal{R}$) without source nor sink, we can extend the previous definitions to larger boxes. We denote by $H^{\mathcal{C},\mathcal{S}}_{[a,b]}(t)$  and  $\mathcal{R}^{\mathcal{C},\mathcal{S}}(x,t)$ these processes. See Figure \ref{fig:DefSourcesSinks} for an illustration.

\section{Geometric case}\label{Sec:GeometricCase}

\begin{figure}
\begin{center}
\includegraphics[height=6cm]{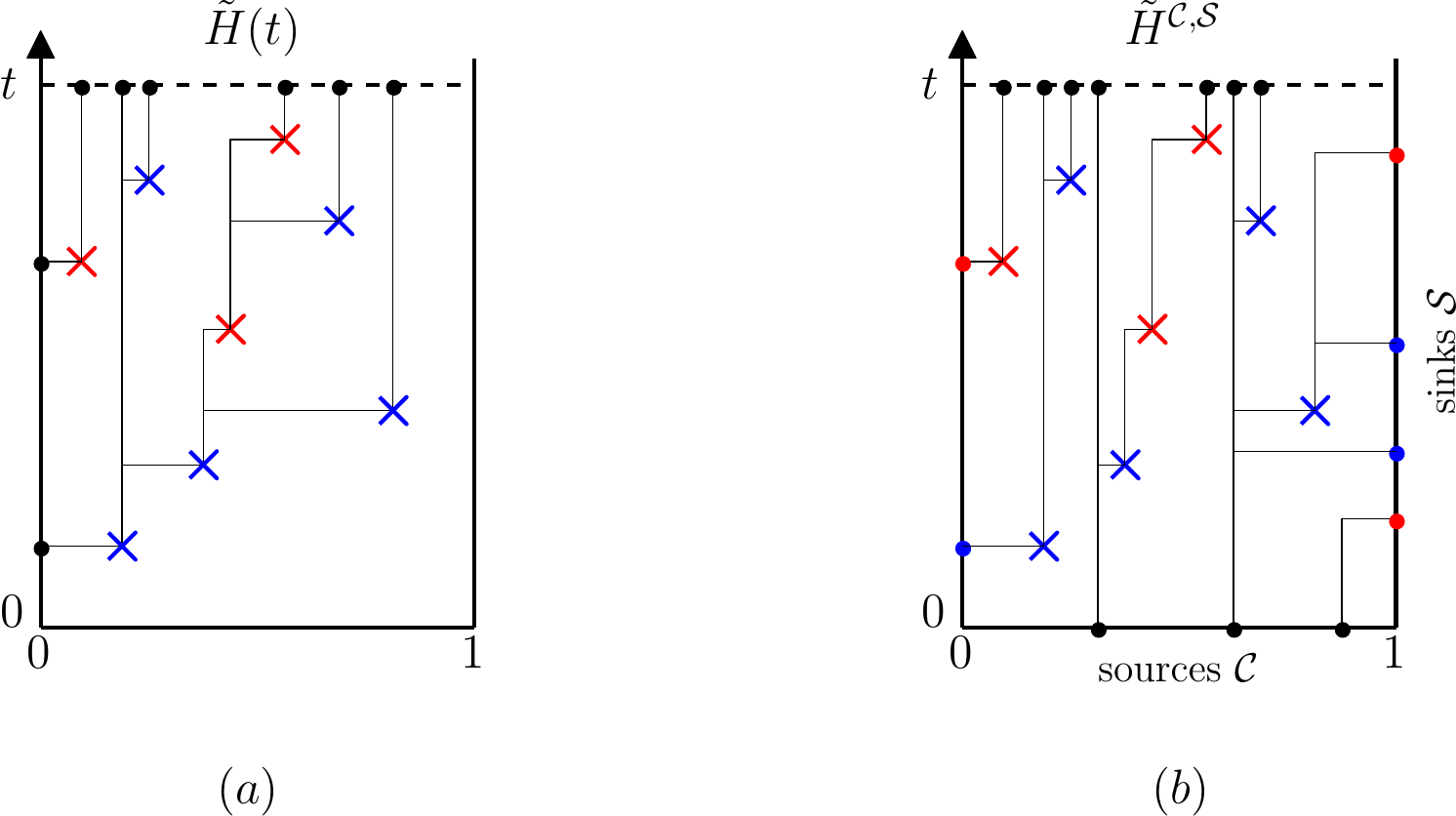}
\caption{\label{fig:GeometriqueRougeBleu}$(a)$ Representation of the geometric $\mu$-Hammersley process $\tilde{H}$. $(b)$ The process $\tilde{H}^{\mathcal{C},\mathcal{S}}$ with sources $\mathcal{C}$ and coloured sinks $\mathcal{S}$.}
\end{center}
\end{figure}

In this section, we concentrate our study on the case where the reproduction law $\mu$ of the trees is geometric with parameter $\alpha \in (0,1)$ \emph{i.e}
$$
\mu(i) = \alpha (1-\alpha)^{i-1}\quad\hbox{for $i=1,2,\ldots$}
$$
This family of distributions plays a special role for this model and it  is the only case for which we are able to derive explicit formula for the stationary measure of the process and by extension, compute the exact value of  $c_\mu$ as well as several other quantities. We do no know whether there are other families of reproduction laws which are exactly solvable. 

The reason why the geometric law is particularly interesting is that we can simplify our model by exploiting its lack of memory. Indeed, instead of attaching a positive number of lives to each atom of our point process, we can now consider that the number of lives of a particle is undecided at its creation and that we will flip a (biased $\alpha$) coin each time another atom attaches to it to decide whether it should stay alive or die. The main point here is that this coin will be held by the child particle instead of the father so that each particle need only to store a Bernoulli random variable instead of a $\mu$-distributed number of lives. 

More precisely, we can consider a Poisson measure on $(0,1)\times (0,+\infty)$ with unit Lebesgue intensity where each atom is coloured independently in blue (resp. red) with probability $(1-\alpha)$ (resp. $\alpha$). Given this coloured Poisson measure, we construct the particle system on $(0,1)$  using the following dynamic
\begin{itemize}
\item The process begins without any particle.
\item A blue atom $(u,t)$ adds a new  particle at position $u$  at time $t$ without interacting with the other particles. 
\item A red atom $(u,t)$ acts as an atom in the usual Hammersley process \emph{i.e.},  at time $t$, it adds a new particle at position $u$ but
kills the particle which has the largest position smaller than $u$ (if any).
\end{itemize}
It is elementary to check that this construction gives the position of the particles of a $\mu$-Hammersley process $H$ with geometric($\alpha$) offspring distribution. We will denote this process by $\tilde{H}$.

\begin{rem}
This construction does not keep track of the dead particles nor the remaining number of lives of each particle. However, since $\mu$ is geometric, the lack of memory of the law implies that, for each $t$, we can reconstruct the number of remaining lives of the particles in $H(t)$ by sampling i.i.d geometric random variables.
\end{rem}

\subsection{Stationarity from bottom to top}

In the case of the usual Hammersley's line process,  Cator and Groeneboom \cite{CatorGroeneboom} showed that $H^{\mathcal{C},\mathcal{S}}$ is stationary when the sources and sinks are independent PPP with constant intensity, respectively $\lambda du$ and $\frac{1}{\lambda}dt$ for some $\lambda>0$. In our setting, the tree structure of the genealogy implies that the intensity of the $\mu$-Hammersley process increases with time so we cannot hope to get the same kind of stationary measures. However, the next result shows that, if we start from a uniform PPP and choose the sinks process correctly, the $\mu$-Hammersley process $H^{\mathcal{C},\mathcal{S}}$ remains a uniform PPP at all times.

\begin{theo}[\textbf{``pseudo-stationary'' measure with sources and sinks}]\label{Th:SourcesStationnaires}
Fix $\lambda>0$. Let the sources process $\mathcal{C}_\lambda$ be a PPP on $(0,1)\times \N$ with constant intensity $\lambda du \otimes \mbox{Geom}(\alpha)$. Let the sinks process $\mathcal{S}_\lambda$ be an inhomogeneous PPP on $(0,\infty)$ with intensity $\frac{ds}{\lambda +(1-\alpha)s}$ . Then, for each $t$, the position of the particles in $H^{\mathcal{C}_\lambda,\mathcal{S}_\lambda}(t)$ is an homogeneous PPP with intensity $(\lambda +(1-\alpha)t)du$ and each particle has a geometric number of lives remaining, independent of everything else. 
\end{theo}

See Figure \ref{Fig:SourcesStationnaires} (a) for an illustration of this result.

\begin{rem}\begin{itemize}
 \item In the degenerated case $\alpha=1$, we recover the stationarity of the usual Hammersley process with sources and sinks. 
\item The theorem above is stated for the $\mu$-Hammersley process on the interval $[0,1]$. However, using the scaling property, it is straightforward to extend the result to $H^{\mathcal{C}_\lambda,\mathcal{S}_\lambda}_{[a,b]}$ for any $a<b$. Then, using the compatibility relation as $a\to -\infty$, we get the same result for the half line $\mu$-Hammersley process $H^{\mathcal{C}_\lambda,\mathcal{S}_\lambda}_{(-\infty,b]}$.
\end{itemize}
\end{rem}

\begin{figure}
\begin{center}
\includegraphics[height=7cm]{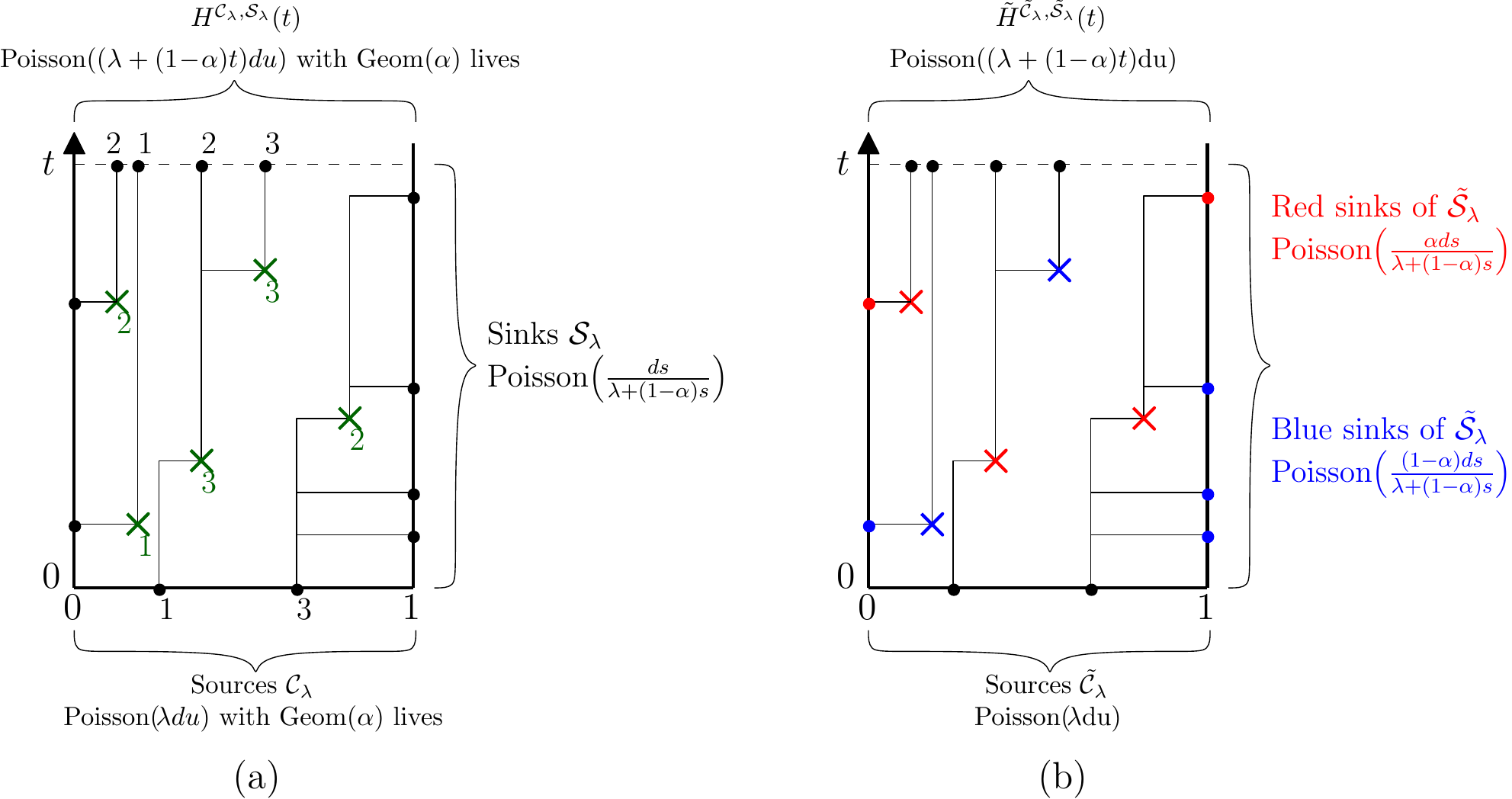}
\caption{\label{Fig:SourcesStationnaires} (a) Illustration of Theorem \ref{Th:SourcesStationnaires} and (b) its equivalent representation with the red/blue construction.}
\end{center}
\end{figure}

\begin{proof} We first translate the statement of the theorem in term of the red/blue representation $\tilde{H}$. Thus, we need to consider the process $\tilde{H}$ where sources and sinks are added. Since particles in $\tilde{H}$ do not carry their number of lives, the sources process is, in this case just a point process on $(0,1)$. On the other hand, the sinks process is such that each sink also carries with it a color, red or blue, which decides whether or not it shall kill its father. We must prove the following result:
\begin{itemize}
\item Suppose that $\tilde{\mathcal{C}}_\lambda$ is a PPP on $(0,1)$ with constant intensity $\lambda du$. Suppose that $\tilde{\mathcal{S}}_\lambda$ on $[0,\infty)$ is independent of $\tilde{\mathcal{C}}_\lambda$ and is an inhomogeneous PPP  with intensity $\frac{ds}{\lambda +(1-\alpha)s}$ where each sink is coloured independently in red with probability $\alpha$ or in blue with probability $1-\alpha$. Then, at all time $t$, $\tilde{H}^{\tilde{\mathcal{C}}_\lambda,\tilde{\mathcal{S}}_\lambda}(t)$ is a PPP with intensity $(\lambda +(1-\alpha)t)du$.
\end{itemize}
See Figure \ref{Fig:SourcesStationnaires} (b) for an illustration. Let us remark that the blue sinks do not, in fact, affect the particle system. However, they do change the graphical representation inside a box by adding horizontal lines so they cannot be dispensed with.

The proof uses the fact that the $\tilde{H}^{\tilde{\mathcal{C}}_\lambda,\tilde{\mathcal{S}}_\lambda}$ may be seen as a superposition of two simpler processes. The blue atoms form a PPP with constant intensity $(1-\alpha)$ that adds up to $\tilde{H}^{\tilde{\mathcal{C}}_\lambda,\tilde{\mathcal{S}}_\lambda}$,  linearly increasing its intensity at rate $(1-\alpha)$. On the other hand, the red atoms  form a PPP with constant intensity $\alpha$ acting with the classical Hammersley dynamic. As we already mentioned, the Poisson distribution of intensity $\lambda>0$ is stationary for the usual Hammersley process provided that the sinks are distributed according to a PPP of intensity $\alpha/\lambda$ (here $\alpha$ is the rate of the red atoms).

Therefore, suppose that at time $t$, $\tilde{H}^{\tilde{\mathcal{C}}_\lambda,\tilde{\mathcal{S}}_\lambda}(t)$ is distributed as a PPP of intensity $\Lambda_t du \defeq(\lambda+(1-\alpha)t)du$, then during the interval time $[t,t+dt]$, the  red atoms arrive at rate $\alpha$ and the red sinks arrive at rate $\alpha/\Lambda_t$, hence leaving invariant the process. On the other hand, during the same time interval, the blue atoms increase the intensity of  $\tilde{H}^{\tilde{\mathcal{C}}_\lambda,\tilde{\mathcal{S}}_\lambda}(t)$ from $\Lambda_t du$ to $(\Lambda_t+ (1-\alpha)dt) du$. Moreover, the law of $\tilde{H}^{\tilde{\mathcal{C}}_\lambda,\tilde{\mathcal{S}}_\lambda}(t)$ is characterized by its infinitesimal generator. Using the linearity property of generators, the generator of $\tilde{H}^{\tilde{\mathcal{C}}_\lambda,\tilde{\mathcal{S}}_\lambda}(t)$ is obtained as the sum of the generator of the red and blue processes which concludes the proof. 
\end{proof}

Theorem \ref{Th:SourcesStationnaires} describes the behaviour of the $\mu$-Hammersley process starting from an initial set of particles distributed as a uniform PPP with intensity $\lambda du$. We wish to take $\lambda = 0$ in order to understand how the process evolves when it starts from an initially empty set of particles. This is not possible in the case of the usual Hammersley process ($\alpha = 1$) as it would require to let the intensity of the sinks process to go to $+\infty$ everywhere. However, we can do it whenever $\alpha < 1$ by considering a sinks process $\mathcal{S}_0$ that has an accumulation point at $0$. Let us point out that allowing an infinite number of sinks at the origin is not a problem when defining the $\mu$-Hammersley process $H^{\emptyset,\mathcal{S}_0}$ on $[0,1]$ that starts without particles. This is because particles in the $\mu$-Hammersley process are only created by the atoms of $\Xi$ hence we only need to consider the sinks located above the minimum height of this set of atoms. More precisely, for each realization of $\Xi$, there exists $\varepsilon>0$  such that the box $[0,1]\times(0,\varepsilon]$ contains no atom of $\Xi$. Then, the $\mu$-Hammersley process is simply empty until time $\varepsilon$ and is subsequently constructed using the standard construction (thus, all the sinks below $\varepsilon$ are ignored).

\begin{cor}[\textbf{``pseudo-stationary'' measure starting from $0$}]\label{cor:sourcefromzero}
Suppose that the sinks process $\mathcal{S}_0$ is an inhomogeneous PPP with intensity $\frac{ds}{(1-\alpha)s}$ on $(0,\infty)$. For every $t$, $H^{\emptyset,\mathcal{S}_0}(t)$  is an homogeneous PPP with intensity $(1-\alpha)t du$ where each particle has a geometric number of remaining lives, independent of everything else. 
\end{cor}

\begin{proof}
Fix $t>0$. We can couple $\mathcal{S}_0$ with the family of sinks processes $\mathcal{S}_\lambda$ with finite intensity $\frac{ds}{\lambda + (1-\alpha)s}$ in such way that $\mathcal{S}_\lambda \subset \mathcal{S}_{\lambda'} \subset \mathcal{S}$ for every $\lambda>\lambda'>0$. Moreover, for each realization of these sinks processes and for any $\varepsilon>0$, there is some $\lambda_0$ such that, for $\lambda < \lambda_0$, the sinks processes $\mathcal{S}_0$ and $\mathcal{S}_\lambda$ have the same roots over the interval $[\varepsilon,t]$. On the other hand, the minimum height of the atoms of $\Xi$ inside $[0,1]\times(0,\infty)$ is strictly positive. If this height is larger than $\varepsilon$, then the $\mu$-Hammersley processes $H^{\emptyset,\mathcal{S}_0}$ and $H^{\emptyset,\mathcal{S}_\lambda}$ coincide up to time $t$. Thus, this happens with probability going to $1$ as $\lambda$ tend to $0$ which concludes the proof of the corollary. 
\end{proof}

\subsection{Stationarity of the root process}
We now turn our attention to the root process. Our next result states that, starting from the same sources $\mathcal{C}_\lambda$ and sinks $\mathcal{S}_\lambda$ as in the previous section, the root process $\mathcal{R}^{\mathcal{C}_\lambda,\mathcal{S}_\lambda}$ is stationary. Recall from Section \ref{section-rootprocess} that for each fixed $t$, $\mathcal{R}^{\mathcal{C}_\lambda,\mathcal{S}_\lambda}$ is a Markov process indexed by "time" $-x$ and taking its values in the set of point measures on $(0,t)$.

\begin{theo}[\textbf{Stationary measure for the root process}]\label{Th:RacinesStationnaires}

Fix $\lambda,t>0$. Suppose that the sources process $\mathcal{C}_\lambda$ is a PPP on $(-\infty,0]\times \N$ with  intensity $\lambda du \otimes \mbox{Geom}(\alpha)$ and the sinks process $\mathcal{S}_\lambda$ on $(0,t)$ is an inhomogeneous PPP with intensity $\frac{ds}{\lambda +(1-\alpha)s}$. Then, the root process with sources and sinks $\mathcal{R}^{\mathcal{C}_\lambda,\mathcal{S}_\lambda}(\cdot,t)$ is stationary. Namely, for all $x\ge 0$,  $\mathcal{R}^{\mathcal{C}_\lambda,\mathcal{S}_\lambda}(x,t)$ is an inhomogeneous PPP with intensity $\frac{ds}{\lambda +(1-\alpha)s}$.
\end{theo}

\begin{proof} The proof we provide here does not use the red/blue representation of the geometrically distributed $\mu$-Hammersley process since it is not clear to us how we could exploit this representation. Instead, we make a direct computation of the infinitesimal generator to show the stationarity. Let us mention that this approach may also be used to give a standalone proof of Theorem \ref{Th:SourcesStationnaires}  which does not rely on previous results concerning Hammersley's line process.

First we make a space-time dilation of the box in order to map the vertical inhomogeneous PPP $\mathcal{S}_\lambda$ with intensity $\frac{ds}{\lambda +(1-\alpha)s}$ to a  PPP with unit intensity while simultaneously renormalizing  the horizontal homogeneous process $\mathcal{C}_\lambda$ to have rate $1$. 
More precisely, set
$$
f_\lambda(s)\defeq \tfrac{1}{1-\alpha}\left(\log (\lambda+(1-\alpha)s)-\log(\lambda)\right)
$$
and consider the image of the graphical representation  by the application
$$
F_\lambda:\begin{array}{r c l}
(-\infty,0]\times [0,t] & \to & (-\infty,0]\times [0,f_\lambda(t)]\\
(-y,s) &\mapsto &(-\lambda y,f_\lambda(s)).
\end{array}
$$
Subsequently to this transformation, the new graphical representation on $(-\infty,0]\times[0,f_\lambda(t)]$ is such that (see Figure \ref{Fig:RacinesStationnaires}):
\begin{itemize}
\item The sources on $(-\infty,0]$ are distributed as a PPP with intensity $1$.
\item The sinks on $ [0,f_\lambda(t)]$ are distributed as a PPP with intensity $1$.
\item The atoms inside the strip are distributed as an inhomogeneous PPP with intensity 
$$
\frac{du}{\lambda}\times \frac{ds}{f_\lambda'\left(f_\lambda^{(-1)}(s)\right)}=e^{(1-\alpha)s}du ds.
$$
\end{itemize}
Since the mapping is non-decreasing in each coordinate, the dynamic of the image of the root process is unchanged: each atom/source kills a geometric number of roots above it (and adds a new root in case of an atom). Thus, we must now prove that, in this new setting, the PPP with unit intensity is stationary for the dynamics.

\begin{figure}
\begin{center}
\includegraphics[height=5.5cm]{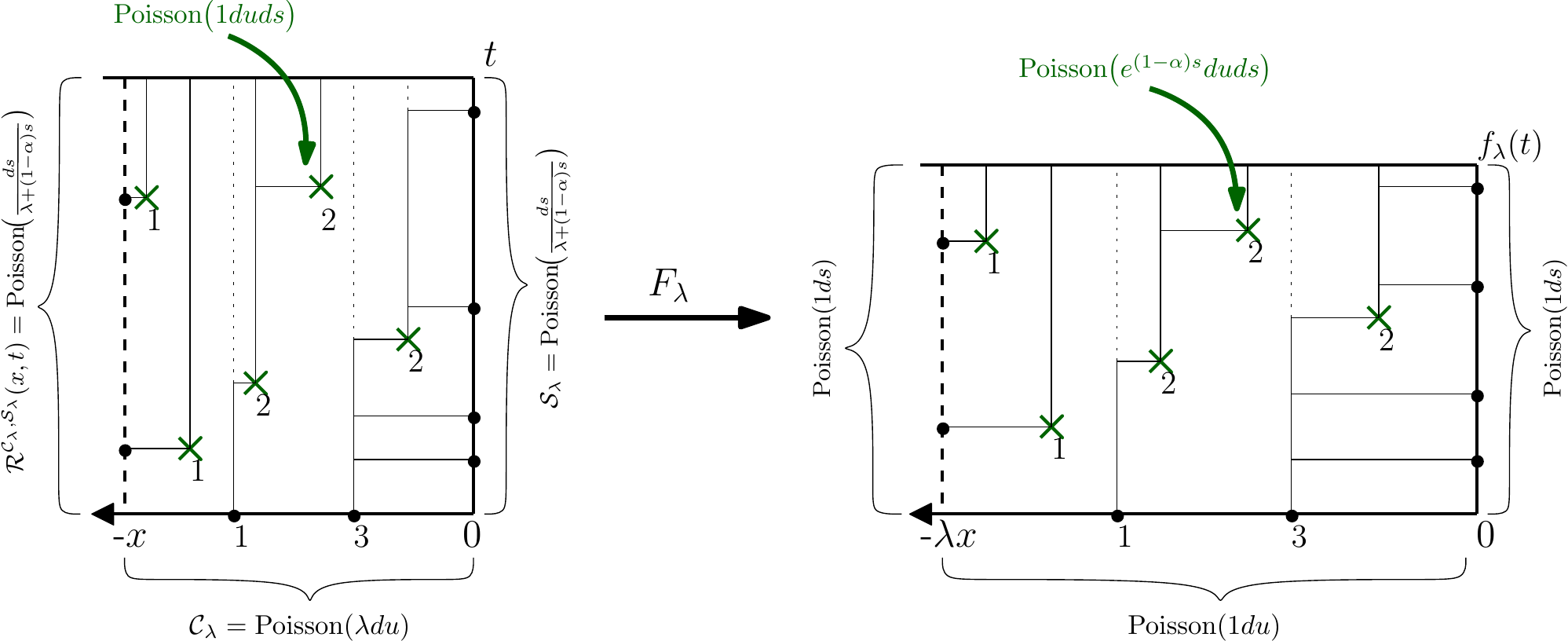}
\caption{\label{Fig:RacinesStationnaires} Illustration of Theorem \ref{Th:RacinesStationnaires} and mapping by $F_\lambda$.}
\end{center}
\end{figure}

We fix  $t$ throughout the proof and we set $T=f_\lambda(t)$. Let 
$$\xi \defeq \sum_i \delta_{r_i}$$
be finite point measure on $(0,T)$ whose atoms are all distinct and ordered increasingly $0 < r_1 < r_2 \ldots < r_n < T$. Given $s\geq0$, we define $r^s_1,r^s_2,\ldots$ to be the (possibly empty) increasing sequence of atoms of $\xi$ located after position $s$. For any positive test function $\phi$, we define the functional
$$
L_\phi(\xi)\defeq\exp\Big(-\sum_{r\in \xi} \phi(r)\Big). 
$$
Let $\mathcal{A}$ be the infinitesimal generator of the mapped root process. According its dynamic, we can write (with the slight abuse of notation $r_{j} = \emptyset$ for $j>n$),
\begin{align*}
\mathcal{A}L_\phi(\xi) = & \int_{0}^{T} e^{(1-\alpha)s}\mathrm{ds} \sum_{k=1}^{+\infty} \alpha(1-\alpha)^{k-1} \left(L_\phi\left(\xi\cup \{s\}\setminus \{r_{1}^s,r_{2}^s,\dots,r_{k}^s\}\right)- L_\phi(\xi) \right)\\
& +  \sum_{k=1}^{+\infty} \alpha(1-\alpha)^{k-1} \left(L_\phi\left(\xi \setminus \{r_{1},r_{2},\dots,r_{k}\}\right)- L_\phi(\xi) \right).
\end{align*}
The first term corresponds to the contribution of an atom located at height $s$. The second term represents the contribution of a source. In both terms, $k$ is the geometric number of roots removed. We use the notation $\beta \defeq 1-\alpha$ to slightly shorten the following formulas. Then, we rewrite the previous equation in the form
\begin{align}
\mathcal{A}L_\phi(\xi) = & \int_{0}^{T} e^{\beta s}\mathrm{ds} \sum_{k=1}^{+\infty} \alpha\beta^{k-1} L_\phi\left(\xi\cup \{s\}\setminus \{r_{1}^s,r_{2}^s,\dots,r_{k}^s\}\right) &\text{[\textbf{atom}]} \label{eq:termatom}\\
&+  \sum_{k=1}^{+\infty} \alpha\beta^{k-1} L_\phi\left(\xi \setminus \{r_{1},r_{2},\dots,r_{k}\}\right)&\text{[\textbf{source}]}\label{eq:termsource}\\ 
&-\left(1+\int_{0}^{T} e^{\beta s}\mathrm{ds}\right)L_\phi(\xi). \label{eq:termreste}
\end{align}

Let us now assume that $\xi$ is a homogeneous PPP with intensity $1$ on $[0,T]$ which is independent of the $\mu$-Hammersley process. We prove that, for every test function $\phi$, we have $\mathbb{E}\left[\mathcal{A}L_\phi(\xi)\right]=0$. We examine the source and atom terms of the previous equation separately. 

\medskip
\noindent\textbf{Contribution of an atom.}
We condition on $\xi$ having $n$ particles at positions $0<r_1<\dots <r_n<T$. This happens with probability $e^{-T}dr_1\ldots dr_n$ An atom that falls in $(r_i,r_{i+1})$ removes particles $r_{i+1},\dots,r_{(i+k)\wedge n}$ where $k$ is the value of the geometric random variable associated with this atom. Thus, splitting the integral depending on  where the atom lies, we get
\begin{multline*}
\E\left[\int_{0}^{T} e^{\beta s}\mathrm{ds} \sum_{k=1}^{+\infty} \alpha\beta^{k-1} L_\phi\left(\xi\cup \{s\}\setminus \{r_{1}^s,r_{2}^s,\dots,r_{k}^s\}\right)\right]\\
\begin{aligned}
=e^{-T}  \sum_{n=0}^\infty \sum_{i=0}^n \int_{0< r_1<\dots < r_{n}<T} \hspace{-1.5cm}dr_1\ldots dr_n \int_{r_i}^{r_{i+1}} \!\!\!\!ds \bigg(e^{\beta s}\sum_{k=1}^{n-i-1} \alpha \beta^{k-1} L_\phi(r_1,\ldots,r_i,s,r_{i+k+1},\ldots,r_{n})
& \\
+ e^{\beta s }(\beta^{n-i-1}\wedge 1) L_\phi(r_1,\ldots,r_i,s)\bigg). &
\end{aligned}
\end{multline*}
The term $\beta^{n-i-1}\wedge 1$ in the equation above comes from the fact that the formula $\beta^{k-1}$ for the probability that the geometric random variable is at least $k$ fails for $k=0$. We gather the two integrals in the sum above to integrate over $n+1$ ordered variables. This yields
\begin{multline*}
\E\left[\int_{0}^{T} e^{\beta s}\mathrm{ds} \sum_{k=1}^{+\infty} \alpha\beta^{k-1} L_\phi\left(\xi\cup \{s\}\setminus \{r_{1}^s,r_{2}^s,\dots,r_{k}^s\}\right)\right]\\
\begin{aligned}
 = & e^{-T}  \sum_{n=0}^\infty \sum_{i=0}^n \int_{0< v_1<\dots < v_{n+1}<T} \hspace{-2cm} dv_1 \ldots dv_{n+1} \bigg(e^{\beta v_{i+1}}\sum_{k=1}^{n-i-1} \alpha \beta^{k-1} L_\phi(v_1,\ldots,v_i,v_{i+1},v_{i+k+2},\ldots,v_{n+1})\bigg)\\
 &+ e^{-T}  \sum_{n=0}^\infty \sum_{i=0}^n \int_{0< v_1<\dots < v_{n+1}<T} \hspace{-2cm} dv_1\ldots dv_{n+1}  
\; e^{\beta v_{i+1}}(\beta^{n-i-1}\wedge 1) L_\phi(v_1,\ldots,v_i,v_{i+1}). 
 \end{aligned}
\end{multline*}
Let $A$ (resp.$B$) denote the first (resp. second) term of the right-hand side. Using the equality 
$$\int_{a<w_1<w_2<\dots <w_k<b} \hspace{-1.8cm} dw_1\ldots dw_k= \frac{(b-a)^k}{k!}$$ 
and integrating over the $k$ free variables $v_{i+2},\ldots,v_{i+k+1}$ and inverting the summation, we get
\begin{eqnarray*}
A\!\!&=&\!\! e^{-T}   \sum_{i=0}^\infty \sum_{k=1}^{\infty}\sum_{n=k+i+1}^\infty\int_{0< u_1<\dots < u_{n+1-k}<T}\hspace{-2.5cm} du_1 \ldots du_{n+1-k} \left(e^{\beta u_{i+1}} \alpha \beta^{k-1} L_\phi(u_1,\ldots,u_{n+1-k})\frac{(u_{i+2}-u_{i+1})^k}{k!}\right)\\
&=&\!\! e^{-T}   \sum_{i=0}^\infty \sum_{k=1}^{\infty}\sum_{l=i+2}^\infty\int_{0< u_1<\dots < u_{l}<T}  \hspace{-1.5cm} du_1 \ldots du_{l} \left(e^{\beta u_{i+1}} \alpha \beta^{k-1} L_\phi(u_1,\ldots,u_{l})\frac{(u_{i+2}-u_{i+1})^k}{k!}\right)\\
&=&\!\! \frac{\alpha}{\beta}e^{-T}   \sum_{l=2}^\infty \sum_{i=0}^{l-2}\int_{0< u_1<\dots < u_{l}<T}  \hspace{-1.5cm} du_1 \ldots du_{l} \left( e^{\beta u_{i+1}}  L_\phi(u_1,\ldots,u_{l})(e^{\beta( u_{i+2}- u_{i+1})}-1)\right)\\
&=&\!\! \frac{\alpha}{\beta}e^{-T}   \sum_{l=2}^\infty \int_{0< u_1<\dots < u_{l}<T}   \hspace{-1.5cm} du_1 \ldots du_{l} \;L_\phi(u_1,\ldots,u_{l})(e^{\beta u_{l}}- e^{\beta u_{1}})\\
&=&\!\! \frac{\alpha}{\beta}e^{-T}   \sum_{l=1}^\infty \int_{0< u_1<\dots < u_{l}<T}  \hspace{-1.5cm} du_1 \ldots du_{l}\; L_\phi(u_1,\ldots,u_{l})(e^{\beta u_{l}}- e^{\beta u_{1}}).
\end{eqnarray*}
Performing similar operations with the second term, we obtain
\begin{eqnarray*}
B&=& e^{-T}   \sum_{i=0}^\infty \sum_{n=i}^\infty\int_{0< u_1<\dots < u_{i+1}<T} \hspace{-2cm} du_1 \ldots du_{i+1} \left(e^{\beta u_{i+1}} L_\phi(u_1,\ldots,u_{i+1})\frac{(T-u_{i+1})^{n-i}}{(n-i)!}(\beta^{n-i-1}\wedge 1)\right)\\
&= & e^{-T}   \sum_{i=1}^\infty \sum_{m=0}^\infty\int_{0< u_1<\dots < u_{i}<T}  \hspace{-1.7cm} du_1 \ldots du_{i} \left( e^{\beta u_{i}}  L_\phi(u_1,\ldots,u_{i})\frac{(T-u_{i})^{m}}{m!} (\beta^{m-1}\wedge 1)\right)\\
& = & \frac{e^{-T}}{\beta} \sum_{i=1}^\infty \sum_{m=0}^\infty\int_{0< u_1<\dots < u_{i}<T} \hspace{-1.7cm}du_1 \ldots du_{i} \left(e^{\beta u_{i}} L_\phi(u_1,\ldots,u_{i}) \frac{(T-u_{i})^{m}}{m!}\beta^m\right) \\
&& \qquad + e^{-T}(1-\frac{1}{\beta})   \sum_{i=1}^\infty \int_{0< u_1<\dots < u_{i}<T} \hspace{-1.7cm}du_1 \ldots du_{i}\;e^{\beta u_{i}} L_\phi(u_1,\ldots,u_{i}) \\
&=&\frac{e^{-T+\beta T}}{\beta}   \sum_{i=1}^\infty \int_{0< u_1<\dots < u_{i}<T} \hspace{-1.7cm} du_1 \ldots du_{i} \;L_\phi(u_1,\ldots,u_{i}) \\
&&\qquad-\frac{\alpha}{\beta}e^{-T}  \sum_{i=1}^\infty \int_{0< u_1<\dots < u_{i}<T} \hspace{-1.7cm}du_1 \ldots du_{i} \;e^{\beta u_{i}} L_\phi(u_1,\ldots,u_{i}).
\end{eqnarray*}

\medskip
\noindent\textbf{Contribution of a source.} 
We split the computation of \eqref{eq:termsource} according whether or not the source kills all the particles.
\begin{multline*}
\mathbb{E}\left[  \sum_{k=1}^{+\infty}\alpha \beta^{k-1} L_\phi\left(\xi\setminus \{r_{1},\dots,r_{k}\}\right)\right]\\
= \alpha\sum_{k=1}^{+\infty} \P\{|\xi|\le k\}\beta^{k-1}+\alpha \sum_{k=1}^{+\infty} \sum_{n=k+1}^\infty \beta^{k-1}e^{-T} \int_{0< r_{k+1}<\dots < r_{n}<T} \hspace*{-2cm} dr_{k+1}\dots dr_{n} \; L_\phi(r_{k+1},\ldots,r_n)\frac{r_{k+1}^k}{k!}
\end{multline*}
where we integrated over the $k$  free variables to obtain the second term.
Denote by $A'$ (resp. $B'$) be the first (resp. second) term of the right-hand side.
We find that
\begin{eqnarray*}
A'&=& \alpha\sum_{k=1}^{+\infty} \sum_{n=0}^k e^{-T} \frac{T^k}{k!}\beta^{k-1}\\ &=&\frac{e^{-T}}{\beta}(e^{\beta T}-\alpha)
\end{eqnarray*}
and 
\begin{eqnarray*}
B'& = &\alpha \sum_{k=1}^{+\infty} \sum_{m=1}^\infty \beta^{k-1}e^{-T} \int_{0< u_{1}<\dots < u_{m}<T}\hspace*{-1.7cm}  du_{1}\dots du_{m}\; L_\phi(u_1,\ldots,u_m)\frac{u_1^k}{k!} \\
&=&\frac{\alpha}{\beta}e^{-T}\sum_{m=1}^\infty  \int_{0< u_{1}<\dots < u_{m}<1} \hspace*{-1.7cm}  du_{1}\dots du_{m}\; L_\phi(u_1,\ldots,u_m)(e^{\beta u_1}-1).
\end{eqnarray*}
Putting everything together, we get
\begin{multline*}
A+B+A'+B'\\
\begin{aligned}
&\begin{aligned}
=&\;\;\frac{\alpha}{\beta}e^{-T}   \sum_{l=1}^\infty \int_{0< u_1<\dots < u_{l}<T}  \hspace{-1.5cm} du_1 \ldots du_{l}\; L_\phi(u_1,\ldots,u_{l})(e^{\beta u_{l}}- e^{\beta u_{1}}) \\
&+ \frac{e^{-T+\beta T}}{\beta}   \sum_{i=1}^\infty \int_{0< u_1<\dots < u_{i}<T} \hspace{-1.7cm} du_1 \ldots du_{i} \;L_\phi(u_1,\ldots,u_{i}) -\frac{\alpha}{\beta}e^{-T}  \sum_{i=1}^\infty \int_{0< u_1<\dots < u_{i}<T} \hspace{-1.7cm}du_1 \ldots du_{i} \;e^{\beta u_{i}} L_\phi(u_1,\ldots,u_{i})\\
&+ \frac{e^{-T}}{\beta}(e^{\beta T}-\alpha) \\
&+\frac{\alpha}{\beta} e^{-T}\sum_{m=1}^\infty  \int_{0< u_{1}<\dots < u_{m}<1} \hspace*{-1.7cm}  du_{1}\dots du_{m}\; L_\phi(u_1,\ldots,u_m)(e^{\beta u_1}-1)
\end{aligned}\\
&= \frac{e^{-T}}{\beta}(e^{\beta T}-\alpha) + (\frac{e^{-T+\beta T}}{\beta} -\frac{\alpha e^{-T}}{\beta})  \sum_{i=1}^\infty \int_{0< u_1<\dots < u_{i}<T}\hspace{-1.6cm} du_1 \ldots du_{i}\; L_\phi(u_1,\ldots,u_{i}) \\
&=\frac{e^{-T}}{\beta}(e^{\beta T}-\alpha) \sum_{i=0}^\infty \int_{0< u_1<\dots < u_{i}<T} \hspace{-1.6cm}du_1 \ldots du_{i}\; L_\phi(u_1,\ldots,u_{i})  \\
&= \frac{1}{\beta}(e^{\beta T}-\alpha) \mathbb{E}[L_\phi(\xi)]
\end{aligned}
\end{multline*}
which is exactly the opposite of the expectation of \eqref{eq:termreste}. Therefore, we conclude that $\mathbb{E}\left[\mathcal{A}L_\phi(\xi)\right]=0$ for any test function $\phi$ which in turn implies that the law of $\xi$ is stationary for the root process, completing the proof of the theorem.
\end{proof}

Similarly to Corollary \ref{cor:sourcefromzero}, we want to take $\lambda=0$ in Theorem \ref{Th:RacinesStationnaires} to understand the stationary measure when there is no source. This can be done by considering again the point process $\mathcal{S}_0$ with an accumulation point at the origin. Just as for the $\mu$-Hammersley process, there is no problem in defining the root process $\mathcal{R}^{\emptyset,\mathcal{S}_0}(\cdot,t)$ since, at any time $x$, only the roots above the atom of $\Xi$ with minimum height in $(-x,0)\times (0,t)$ may be modified. Moreover, we can even let $t$ go to $\infty$ since the root processes $\mathcal{R}^{\emptyset,\mathcal{S}_0}(\cdot,t)$ are compatible for different values of $t$. 

\begin{cor}[\textbf{Stationary measure for the root process without source}]\label{cor:sinkstationnaireinfinie}
Fix $0<t\leq \infty$. Let $\mathcal{S}_0$ be  an inhomogeneous PPP on $(0,t)$ with intensity $\frac{ds}{(1-\alpha)s}$. Then, the root process $\mathcal{R}^{\emptyset,\mathcal{S}_0}(\cdot,t)$ without source and with initial sinks distribution $\mathcal{S}_0$ is stationary.
\end{cor}
Observe that since $\int_0^{\varepsilon} \tfrac{ds}{(1-\alpha)s}=+\infty$ there are infinitely many sinks in every interval $(0,\varepsilon)$. Thus $\mathcal{S}_0$ has an accumulation point at the origin, as wanted.

\begin{proof} We only need to prove the result for $t<\infty$. Again, the proof goes by uniform approximation. We construct  $\mathcal{R}^{\emptyset,\mathcal{S}_0}(\cdot,t)$ together with the root processes $\mathcal{R}^{\mathcal{C}_\lambda, \mathcal{S}_\lambda}(\cdot,t)$ with sources $\mathcal{C}_\lambda$ distributed as a PPP with intensity $\lambda du \otimes \mbox{Geom}(\alpha)$ and starting from an inhomogeneous PPP $\mathcal{S}_{\lambda}$ with intensity $\frac{ds}{\lambda +(1-\alpha)s}$. All theses processes are constructed on the same space and using the same PPP $\Xi$. Moreover, we couple the initial sinks processes in such way that $\mathcal{S}_\lambda \subset \mathcal{S}_0$. Now, if the initial configurations $\mathcal{S}_\lambda$ and $\mathcal{S}_0$ coincide above height $s$ \emph{i.e.}
\begin{equation}\label{eq:couplRoot0}
\mathcal{S}_\lambda \cap [s,t] =  \mathcal{S}_0 \cap [s,t],
\end{equation}
and if there is no source in the interval $[-x,0]$ then, for any $y\leq x$, 
\begin{equation*}
\mathcal{R}^{\mathcal{C}_\lambda, \mathcal{S}_\lambda}(y,t) \cap [s,t] = \mathcal{R}^{\emptyset, \mathcal{S}_0}(y,t) \cap [s,t].
\end{equation*}
The corollary now follows from the fact that, $x$ and $s$ being fixed, the probability of finding a source on $[-x,0]$ goes to $0$ when $\lambda\to 0$ while, simultaneously, the probability of \eqref{eq:couplRoot0} being true goes to $1$.  
\end{proof}

\subsection{Geometric Hammersley Process on $\R$. Stationnary Half plane representation}\label{sec:halfplanegeom}

In Section \ref{section-poisson}, we defined the $\mu$-Hammersley process on a bounded interval  $[a,b]$ \emph{i.e} we  constructed the graphical representation $\mathcal{G}_{a,b}$ inside a strip $[a,b]\times[0,\infty)$. As we already noticed, we can grow this strip from the left by letting $a$ tends to $-\infty $ without modifying  the representation already constructed. Thus, for every realization of the initial PPP $\Xi$ on the upper half-plane $\mathbb{H} \defeq \R\times (0,\infty)$, we can construct a graphical representation $\mathcal{G}_b$ for the quarter-plane $\mathbb{H}_b \defeq (-\infty,b]\times (0,\infty)$. This provides a construction for the $\mu$-Hammersley process on the half-line $(-\infty,b]$ for any $b$. We wish now to define the process on the whole line \emph{i.e} we want construct a graphical representation on $\mathbb{H}$. We cannot carelessly  let $b$ go to $\infty$ since increasing the right boundary may modify the representation inside whole box. However, it turns out that the limit as $b$ goes to infinity always exists and defines a translation invariant graphical representation on $\mathbb{H}$. More precisely, we devote this section to proving the following result which provides a complete picture of the limit.

\begin{theo}[\textbf{Stationary half plane representation: geometric case}]\label{Theo:HalfPlane}
Assume that $\mu$ is the geometric distribution with parameter $\alpha$ and recall that $\Xi$ denotes a PPP on $\mathbb{H}\times\N$ with intensity $du \otimes dt \otimes \mu$. Then, as $b$ goes to $+\infty$, the graphical representations $\mathcal{G}_b$ on $\mathbb{H}_b$ converge locally, almost surely, to a random graphical representation $\mathcal{G}_\infty$ on $\mathbb{H}$. That is, for a.s. every realization of $\Xi$ and for any closed box $\mathbb{B} \defeq [x,y]\times [s,t] \subset \mathbb{H}$, the restriction of $\mathcal{G}_b$ to $\mathbb{B}$ becomes constant for $b$ large enough. The random variable $\mathcal{G}_\infty$ has the following properties:
\begin{enumerate}
\item The law of $\mathcal{G}_\infty$ is characterized by its trace on finite boxes. For $\mathbb{B} \defeq [x,y]\times [s,t] \subset \mathbb{H}$, the law of $\mathcal{G}_\infty$ restricted to $\mathbb{B}$ is that of the stationary $\mu$-Hammersley process with sources and sinks where:
\begin{itemize}
\item[\textup{(a)}] The sources and sinks processes are independent. 
\item[\textup{(b)}] The sinks form an inhomogeneous PPP on $\{y\}\times [s,t]$ with intensity $\frac{1}{(1-\alpha)\tau}d\tau$.
\item[\textup{(c)}] The sources form a PPP on $[x,y]\times \{s\}$ with intensity $(1-\alpha)s du$ and each source is given an independent geometric number of lives with parameter $\alpha$.
\end{itemize}
\item The law of $\mathcal{G}_\infty$ is invariant by horizontal translation. 
\item $\mathcal{G}_\infty$ defines almost surely a single infinite random tree embedded in $\mathbb{H}$ such that
\begin{itemize}
\item[\textup{(a)}] The set of vertices of the tree is exactly the positions of the atoms of $\Xi$ and the number of children of each vertex is the initial 
number of lives of the atom. 
\item[\textup{(b)}] The tree is rooted at $(-\infty,0)$. It grows upward and to the right \emph{i.e.} if a point $(a',b')$ belongs to the progeny of $(a,b)$, then $a'\geq a$ and $b' \geq b$.
\item[\textup{(c)}] Every (forward) infinite ray admits vertical asymptote \emph{i.e.} converges toward $(x,\infty)$ for some finite $x$.
\end{itemize}
\end{enumerate}
\end{theo}

\begin{figure}
\includegraphics[width=16cm]{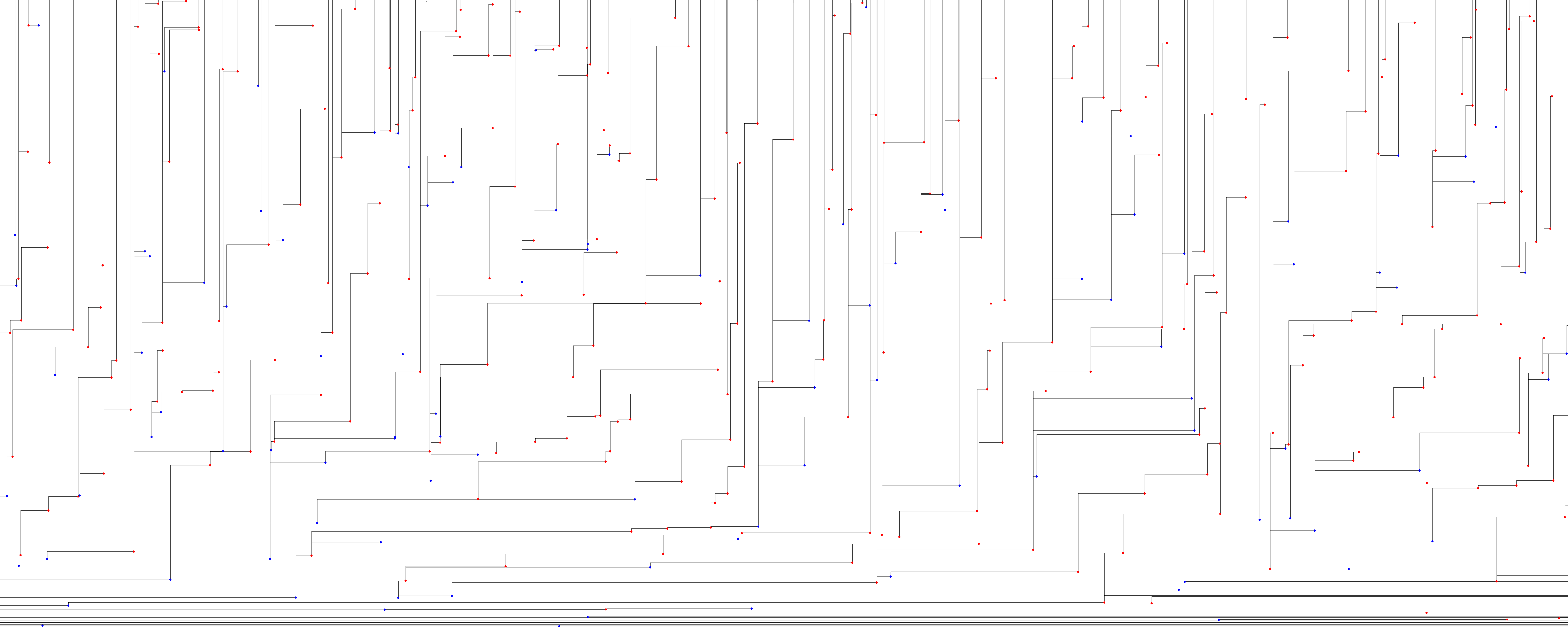}
\caption{\label{Fig:simHalfPlane} Simulation of the half-plane stationary representation $\mathcal{G}_\infty$ with geometric distribution with parameter $\alpha = 2/3$. The window displayed is $[0,30]\times (0,12]$. The vertical lines exiting the top boundary form a PPP with intensity $4du$ while the horizontal line exiting the left/right sides form PPPs with intensity $\frac{3dt}{t}$.}
\end{figure}

\begin{proof}
\textbf{Convergence and characterization of $\mathcal{G}_\infty$}. 
Let us first explain why a random variable with law $\mathcal{G}_\infty$ exists. Let $\mathbb{B} \defeq [x,y]\times [s,t] \subset \mathbb{H}$ and denote by  $\mathcal{G}_{\mathbb{B}}$ the stationary graphical representation in $\mathbb{B}$ where sources and sinks are added according to $1.$ (a) - (c). Let $\mathbb{B}' \defeq [x',y']\times [s',t'] \subset \mathbb{B}$ be a sub-box of $\mathbb{B}$. We claim that the restriction of  $\mathcal{G}_{\mathbb{B}}$ to $\mathbb{B}'$ is, again, distributed as the stationary representation with sources and sinks given by $1.$ (a) - (c). This is easier to explain with the help of Figure \ref{Fig:explicationKolmo}. According to Theorem \ref{Th:RacinesStationnaires}, the sinks process along the segment $\{y'\}\times [s,t]$ is again an inhomogeneous PPP with intensity $\frac{1}{(1-\alpha)\tau}d\tau$. Since it is Poisson, its restrictions on the segments $A$, $B$ and $C$ are independent. On the other hand, using Theorem \ref{Th:SourcesStationnaires} for the Hammersley process inside the box $[x,y']\times [s,s']$, we find that the sources process of $\mathbb{B}'$ on the segments $D$ (and also $E$) is, as required, an homogeneous PPP with intensity $(1-\alpha)s'du$ which depends on the sinks process on $A$ but is independent of the sinks process on $B$.  

This compatibility relation between boxes  together with Kolmogorov's extension Theorem imply that there exists a random variable $\mathcal{G}_\infty$ on $\mathbb{H}$ whose restriction has the required properties. However, this 
abstract construction does not tell us if there is a deterministic algorithm to construct  $\mathcal{G}_\infty$ once the set of atoms of $\Xi$ is fixed. The theorem above states that this can be achieved by taking the limit of the graphical constructions (without source nor sink) inside boxes of increasing size. 

\begin{figure}
\begin{center}
\includegraphics[height=6cm]{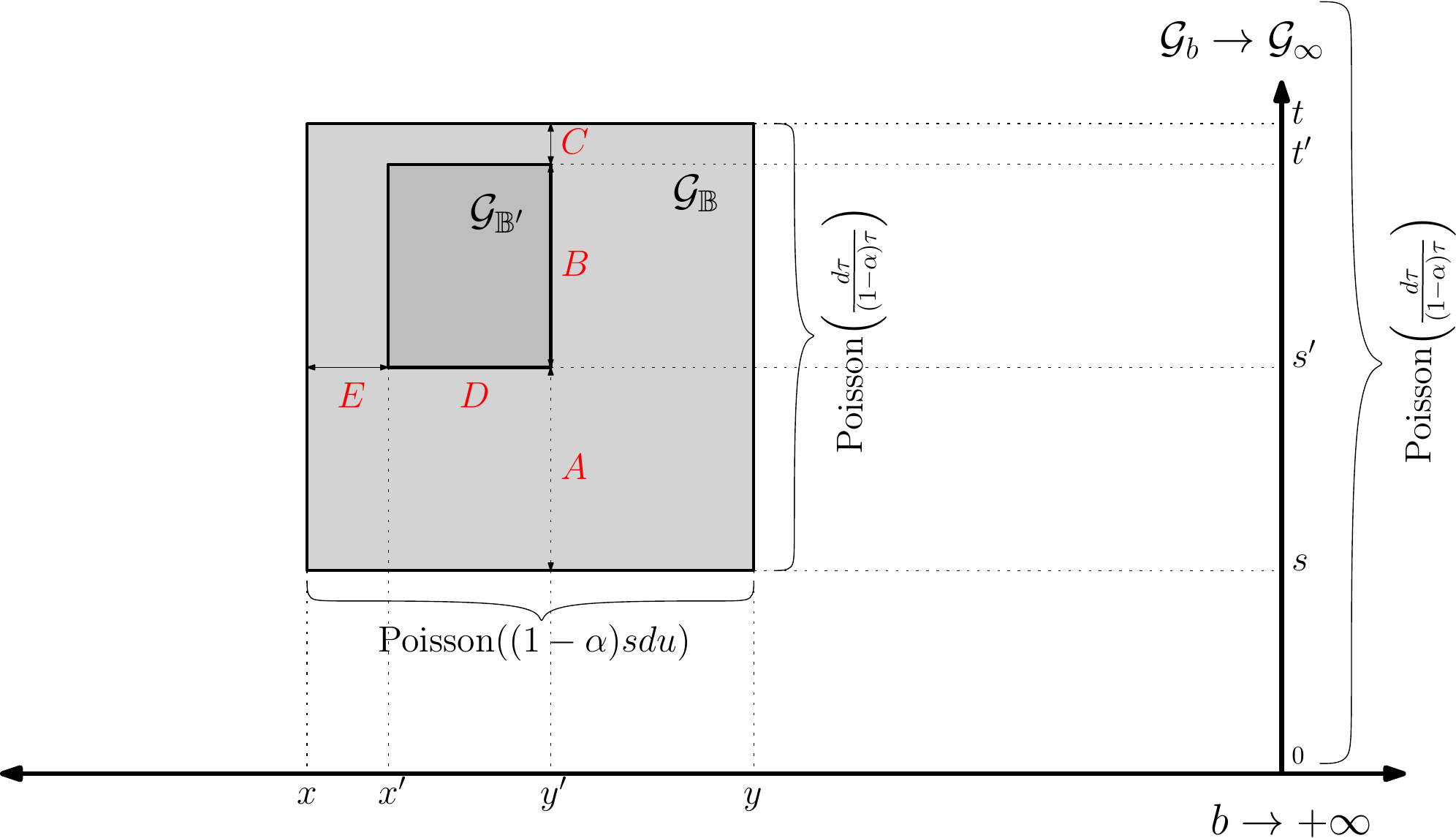}
\end{center}
\caption{\label{Fig:explicationKolmo}Graphical explanation of the restriction property.}
\end{figure}

In order to prove that the restriction of  $\mathcal{G}_b$ to a finite box $\mathbb{B}\defeq [x,y]\times [s,t] \subset \mathbb{H}$ becomes constant for large $b$, it suffices to show that, if we fix a realization of $\Xi$, then
the set of horizontal lines in $\mathcal{G}_b$ which intersect the right boundary $\{y\}\times[s,t]$ of $\mathbb{B}$ stays constant for $b$ large enough.

$\Xi$ being fixed, we can assume without loss of generality that $s>0$ is small enough so that there is no atom of $\Xi$ inside $[x,y]\times (0,s]$. This implies that there can never be a vertical line that intersects the lower boundary $[x,y]\times \{s\}$ of $\mathbb{B}$ in any graphical representation $\mathcal{G}_b$. In other words, the restriction of $\mathcal{G}_b$ to $\mathbb{B}$ has no source. Thus, the sinks on the right boundary $\{y\}\times[s,t]$ together with the atoms of $\Xi$ inside $\mathbb{B}$ uniquely determine the representation inside the box. In particular, if the sinks process becomes constant for large $b$, so does the whole representation inside $\mathbb{B}$.

Thus, we will show that, for any fixed $0<s<t$, the set of heights of the horizontal lines intersecting the vertical segment $\{0\}\times [s,t]$ in $\mathcal{G}_b$ becomes constant for large $b$ and follows the law of an inhomogeneous PPP with intensity $\frac{1}{(1-\alpha)\tau}d\tau$. We denote by $\bar{\mathcal{R}}_{[s,t]}(b)$ this set of heights and by extension $\bar{\mathcal{R}}_{(0,t]}(b)$ the set of  heights of the horizontal lines crossing the vertical segment $\{0\}\times (0,t]$ in $\mathcal{G}_b$. Of course, we have  $$\bar{\mathcal{R}}_{[s,t]}(b) =\bar{\mathcal{R}}_{(0,t]}(b)\cap [s,t].$$

Let us first point out that the $\bar{\mathcal{R}}_{(0,t]}(b)$ is non decreasing in $b$ for the partial inclusion order. Indeed, as explained in Section \ref{sect-augmentationdroite}, growing the box to the right can add new roots (this corresponds to the case of the second class particle exiting by the left border of the box) but cannot remove existing ones. Hence, this process converges almost surely to some possibly infinite limit $\bar{\mathcal{R}}_{(0,t]}(\infty)$. We need to prove that the number of roots above any level $s>0$ remains bounded and characterizes the law of the limit. 

For any fixed $b$, the set of roots $\bar{\mathcal{R}}_{(0,t]}(b)$ has the same law as the root process without source nor sink $\mathcal{R}(b,t)$:
\begin{equation}\label{Eq:majstationnaire}
\bar{\mathcal{R}}_{(0,t]}(b)\overset{\mbox{\tiny{law}}}{=}\mathcal{R}(b,t),
\end{equation}
($\bar{\mathcal{R}}_{(0,t]}(\cdot)$ corresponds to growing from the right whereas $\mathcal{R}(\cdot,t)$ corresponds to growing from the left).
Besides, it is clear that the dynamic of $\mathcal{R}(\cdot,t)$ is monotone with respect to its starting configuration: if we start the process from two initial configurations, one of them containing the other, then at all times, the set of roots maintains this inclusion property. Thus, if we construct simultaneously the root process $\mathcal{R}(\cdot,t) $ starting from no particle and the root process $\mathcal{R}^{\emptyset,\mathcal{S}}(\cdot,t) $ starting from an inhomogeneous PPP $\mathcal{S}_0$ on $(0,t)$ with intensity $\frac{1}{(1-\alpha)\tau}d\tau$ and we use the same PPP $\Xi$, we get, for any  $b\ge 0$, 
$$\mathcal{R}^{}(b,t) \subset\mathcal{R}^{\emptyset,\mathcal{S}_0}(b,t). $$ 
Using the stationarity of $\mathcal{R}^{\emptyset,\mathcal{S}}(\cdot,t)$ stated in Corollary \ref{cor:sinkstationnaireinfinie}, we find that, for any $b\ge 0$, the law of  $\mathcal{R}(b,t)$ is stochastically dominated by $\mathcal{S}$. According to \eqref{Eq:majstationnaire}, the same also holds true for  $\bar{\mathcal{R}}_{(0,t]}(b)$ for each $b$. Taking the increasing limit as $b$ tends to infinity, we conclude that the graphical representation inside finite boxes fixates for $b$ large and we also obtain the desired upper bound on the limit distribution of the sinks. 

We use a truncation argument to prove the matching lower bound. Let $\mathcal{S}_\lambda$ be a PPP on $(0,t]$ with intensity $\frac{1}{\lambda+(1-\alpha)\tau}d\tau$. Let $\mathcal{C}_\lambda$ be a sources process distributed according to a PPP with intensity $\lambda du$ where each source is given a geometric number of lives of parameter $\alpha$. We construct the two root processes $\mathcal{R}^{\mathcal{C}_\lambda,\emptyset}(\cdot,t)$ and $\mathcal{R}^{\mathcal{C}_\lambda,\mathcal{S}_\lambda}(\cdot,t)$ on the same space using the atoms of $\Xi$ and $\mathcal{C}_\lambda$. We will need the following facts:
\begin{itemize}
\item Adding sources can remove roots but cannot create them. Thus, for any $b\ge 0$, 
\begin{equation}\label{Eq:majavecsource}
\mathcal{R}^{\mathcal{C}_\lambda,\emptyset}(b,t) \subset \mathcal{R}(b,t). 
\end{equation}
\item Adding sources does not affect the monotonicity property of the root process with respect to the initial configuration (sources are equivalent to atoms of $\Xi$ with height $0$). Thus, for any $b\ge 0$, we have 
$$\mathcal{R}^{\mathcal{C}_\lambda,\emptyset}(b,t) \subset \mathcal{R}^{\mathcal{C}_\lambda, \mathcal{S}_\lambda}(b,t). $$ 
\end{itemize}
The processes $\mathcal{C}_\lambda$ and $\mathcal{S}_\lambda$ are such that $\mathcal{R}^{\mathcal{C}_\lambda, \mathcal{S}_\lambda}(\cdot,t)$ is stationary. Since the intensity of $S_\lambda$ is bounded, it follows that $\liminf_{b\to\infty}|\mathcal{R}^{\mathcal{C}_\lambda, \mathcal{S}_\lambda}(b,t)| < \infty$ a.s. Now, for any $K$, the probability for $\mathcal{R}$ to go from $K$ roots to $0$ root in a unit of time is bounded below by a constant depending only on $K$ (this happens whenever there are at least $K$ sources in $\mathcal{C}_\lambda$ but no atom of $\Xi$ during a unit time period). Thus, an easy application of the conditional Borel-Cantelli Lemma shows that $\mathcal{R}^{\mathcal{C}_\lambda, \mathcal{S}_\lambda}(\cdot,t)$ reaches $\emptyset$ a s. This means that there exists a.s. some random $b_0$ such that 
$$\mathcal{R}^{\mathcal{C}_\lambda,\emptyset}(b_0,t)= \mathcal{R}^{\mathcal{C}_\lambda, \mathcal{S}_\lambda}(b_0,t).$$ 
which in turn implies that
$$\forall b>b_0, \quad\mathcal{R}^{\mathcal{C}_\lambda,\emptyset}(b,t)= \mathcal{R}^{\mathcal{C}_\lambda, \mathcal{S}_\lambda}(b,t), $$ 
proving that the law of  $ \mathcal{R}^{\mathcal{C}_\lambda,\emptyset}(b,t)$ converges, as $b$ tends to infinity, to  that of $\mathcal{S}_\lambda$.
Combining this with \eqref{Eq:majstationnaire} and \eqref{Eq:majavecsource}, we find that, for any $\lambda>0$, the law of $\bar{\mathcal{R}}_{(0,t]}(\infty)$ stochastically dominates  $\mathcal{S}_\lambda$. Letting $\lambda$ tend to $0$, we conclude that the law of the sinks on the right boundary of any finite box are given by $1.$ (b). 

Thus, we have proved that the trace of $\mathcal{G}_\infty$ on any quarter-plane $\mathbb{H}_b$ has the same law as the graphical representation in $\mathbb{H}_b$ obtained by putting the sinks sequence $\mathcal{S}_0$ on $\{b\}\times(0,\infty)$. From this point on, it follows directly from Corollary \ref{cor:sourcefromzero} that the sources on the bottom side of any box are distributed according to $1.$ (c). Finally, the fact that sinks and sources processes on any box are independent is shown with the exact same arguments used to prove that the sources and sinks processes of a box remain independent after   restriction \emph{c.f.} Figure \ref{Fig:explicationKolmo}.

\medskip

\textbf{Properties of $\mathcal{G}_\infty$.}
The invariance of $\mathcal{G}_\infty$ by horizontal translation follows directly from the invariance of the PPP $\Xi$ used to construct it. We prove that $\mathcal{G}_\infty$ is connected. This result is equivalent to the following statement concerning the root process $\mathcal{R}^{\emptyset,\mathcal{S}_0}(\cdot,t)$ without source and with stationary sinks distribution $\mathcal{S}_0$: 
\begin{itemize}
\item[(A)]
For two roots $r,r' \in \mathcal{S}_0$, there exists $b_1>0$ such that $r,r'$ have the same common ancestor in $\mathcal{R}^{\emptyset,\mathcal{S}_0}(b_1,t)$.
\end{itemize}
To see why (A) must be true, call $r(b)$ and $r'(b)$ the respective ancestor of $r$ and $r'$ in $\mathcal{R}^{\emptyset,\mathcal{S}_0}(b,t)$. We assume that $r < r'$ so that $r(b) \leq r'(b)$ for all $b$. Let $N(b)$ denote the number of roots located on the segment $[r(b),r'(b)]$. 
We observe that $N$ is non-increasing in $b$ since each atom of $\Xi$ encountered removes, at least one root above it. Moreover, if $N$ is strictly larger than $1$, then it decreases strictly every time $\mathcal{R}^{\emptyset,\mathcal{S}}$ encounters an atom of $\Xi$ with at least two lives located between $r(b)$ and the root just below it. Since $\Xi$ is a unit PPP, this happens infinitely often a.s. hence (A) holds and $\mathcal{G}_\infty$ is a single tree almost surely. By construction, this tree grows upward and to the right and satisfies $3.$ (a). Since it contains all the point of $\Xi$, it must necessarily be rooted at $(-\infty,0)$. 

The last statement $3.$ (c) requires additional work and is proved in Proposition \ref{Prop:LimiteXt}. Yet, we stated it in the theorem nevertheless in order to gather all the properties of $\mathcal{G}_\infty$ together.
\end{proof}

\begin{figure}
\begin{subfigure}{.5\textwidth}
  \centering
 \includegraphics[width=.8\linewidth]{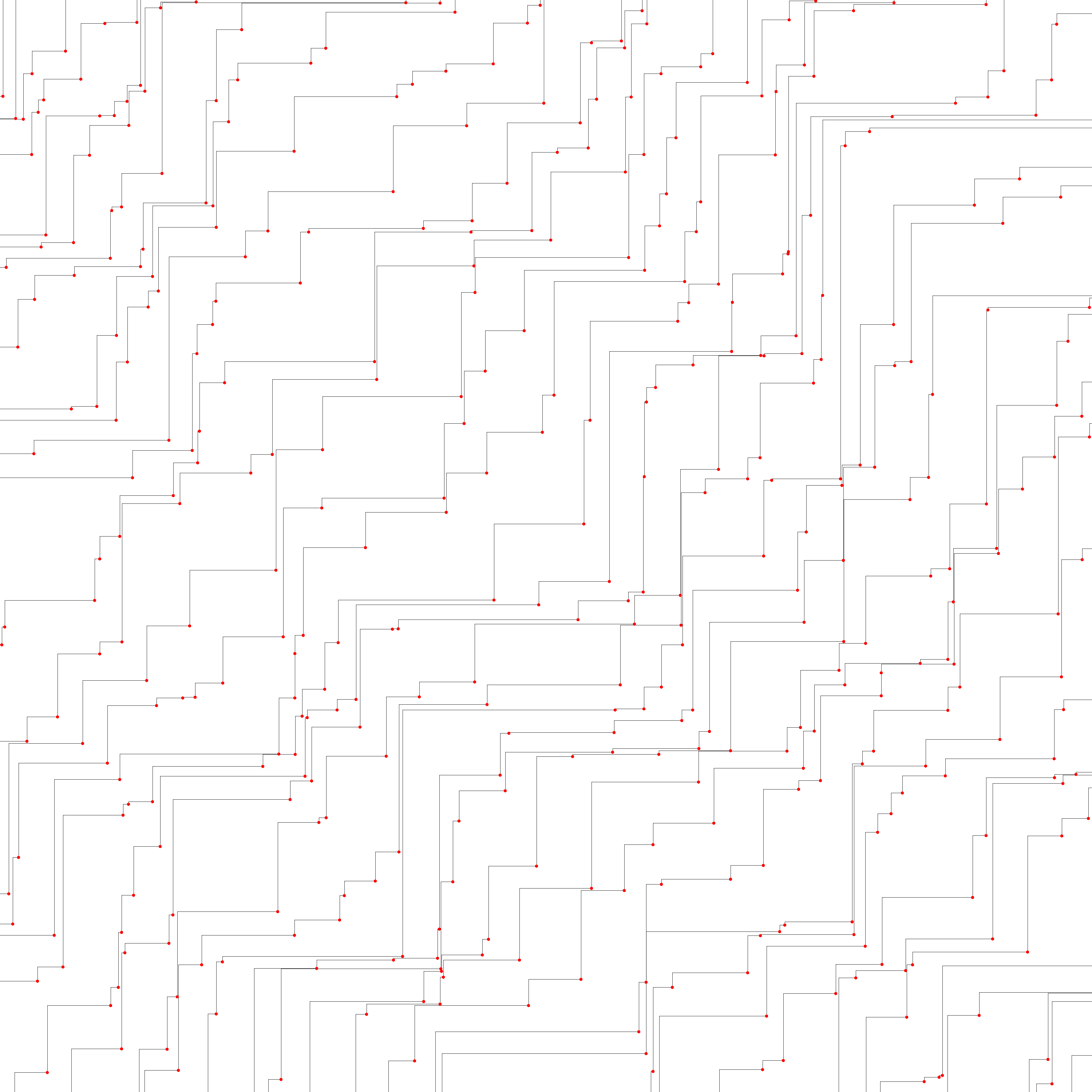}
  \caption{$\mbox{sources}\sim \mbox{sinks} \sim \mbox{Poisson}(1du)$}
\end{subfigure}%
\begin{subfigure}{.5\textwidth}
  \centering
\includegraphics[width=.8\linewidth]{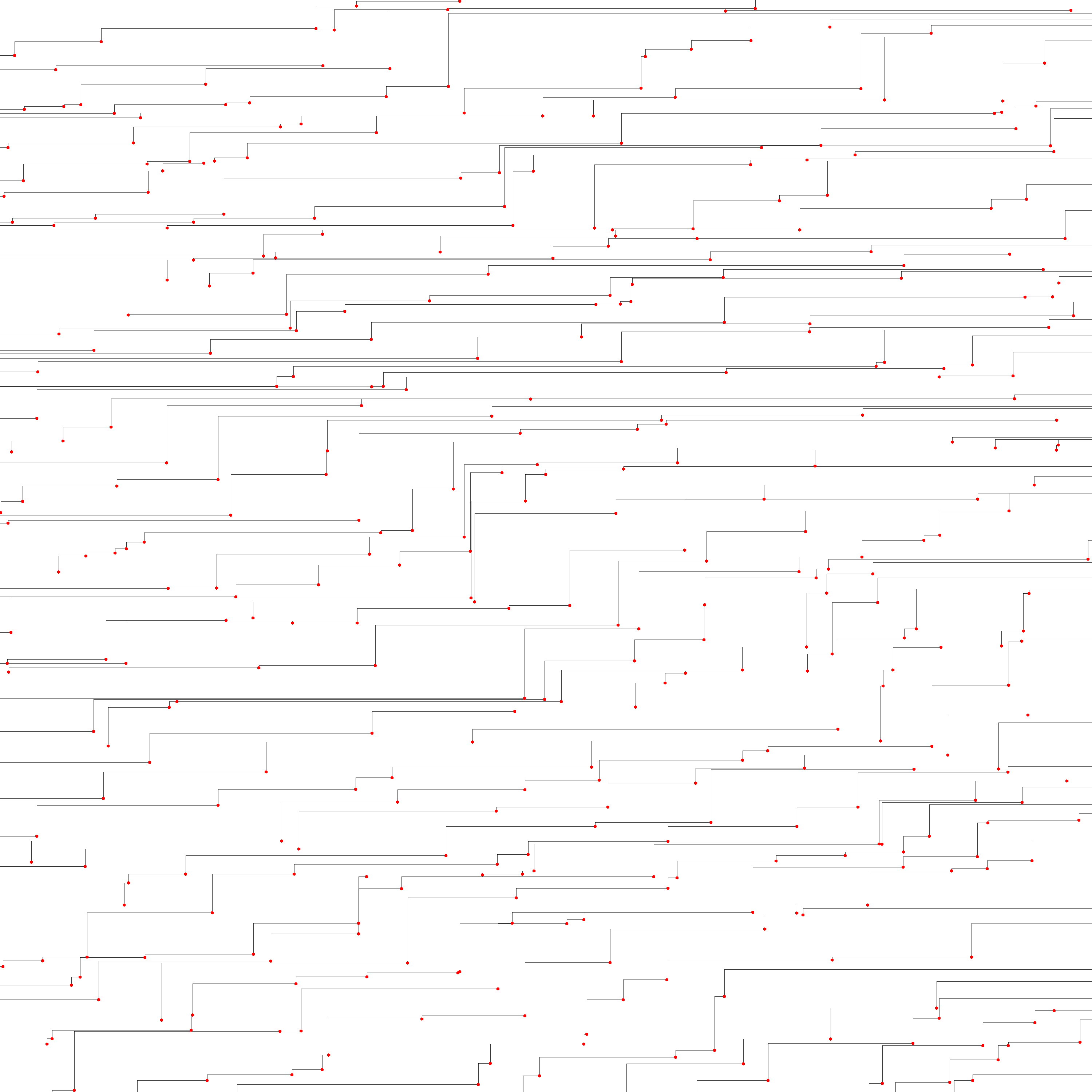}
  \caption{$\mbox{sources}\sim\mbox{Poisson}(\frac{du}{3})$, $\mbox{sinks}\sim\mbox{Poisson(3dt)}$}
\end{subfigure}
\caption{\label{Fig:simHammersleyLine}Hammersley's line process inside the box $[0,20]\times[0,20]$ for two different stationary measures.}
\end{figure}

\begin{rem}\label{Rem:HammersleyLines}
\begin{itemize}
\item In the case of the usual Hammersley's line process, it is also possible to construct an infinite stationary graphical representation. However, this need to be done over the whole plane $\mathbb{R}$ instead of the half plane $\mathbb{H}$ (looking at the representation in $\mathbb{H}_b$  and letting $b\to\infty$ yields a degenerated limit). Moreover, the limit is not unique because it depends on the way the sequence of boxes grows to fill the whole plane. More precisely, if we consider quarter planes of the form $[-x,x]\times[-y,y]$, the limit law will depend on $\lambda \defeq \lim y/x$. This parameter gives the ``angle'' of the Hammersley lines: the sources and sinks processes on the side of any box are independent homogeneous PPP with respective intensity $\lambda du$ and $\frac{dt}{\lambda}$. See Figure \ref{Fig:simHammersleyLine}. In the case of the geometric $\mu$-Hammersley process, there is no such latitude so the limiting law $\mathcal{G}_\infty$ is unique. 

\item The representation $\mathcal{G}_\infty$ defines the stationary $\mu$-Hammersley process on $\R$ starting from no particle at time $0$. However, we can also consider the process starting from any homogeneous PPP with intensity $\lambda du$ by simply looking at the process after time $\lambda/(1-\alpha)$. In other words, we can construct it from the graphical representation restricted to the half-plane $(-\infty,+\infty)\times[\frac{\lambda}{1-\alpha},+\infty)$. 
\end{itemize}
\end{rem}

\subsection{The point of view of the particle.}

\subsubsection{The case of Hammersley's line process}\label{Sec:PalmClassique}

In this subsection only, we focus on the classical Hammersley process with lines instead of trees. In our setting, this corresponds to the degenerated case $\mu = \delta_1$. As we already mentioned in Remark \ref{Rem:HammersleyLines}, it is possible to define a stationary version of the process on the whole line $\mathbb{R}$ starting from any homogeneous PPP with intensity $\lambda>0$ \emph{c.f.} \cite{AldousDiaconis95}. This means that, at each time, the distribution of the process is still a Poisson process with intensity $\lambda$. We will 
prove a similar result for its Palm measure. Even though this result is probably already known, we could not find a reference so we include here a proof for completeness.  We denote by $(X_i(t),\; i\in \mathbb{Z})$ the position of the particles at time $t$ with the convention that
\begin{itemize}
\item The initial indexing is such that $\ldots < X_0(0) < 0 < X_1(0) < X_2(0) < \ldots$.
\item Each index $i$ corresponds to an Hammersley line \emph{i.e} each time a particle appears, it take the index of the one it killed.
\end{itemize}
For each $i$, define $Z_i(t)  = (z^1_i(t),z^2_i(t),\ldots)$ to be the sequence of gaps between particles on the right of the $i$th Hammersley line at time $t$:
$$
\forall j\ge 1, \quad z^{j}_i(t)=X_{i+j}(t)-X_{i+j-1}(t).
$$

\begin{prop}[\textbf{Environment seen from a particle: Hammersley's line process}]\label{Prop:StationaritePalm}
For every $t>0$ and $i\geq 1$, the sequence $Z_i(t)$ is i.i.d. with exponential law of parameter $\lambda$.
\end{prop}

\begin{proof} We assume without loss of generality that $\lambda=1$. We fix some $t\geq 0$ throughout the proof and we write $Z_i=(z^{1}_i,z^{2}_i,\ldots)$ instead of $Z_i(t)=(z^{1}_i(t),z^{2}_i(t),\ldots)$ to shorten the notation. We first note that for all $i,j\ge 1$, sequences $Z_i$ and $Z_j$ have the same distribution. Hence, we just need to prove the proposition for $Z_1$. Moreover, noticing that $z^1_j = z^j_1$, it also follows that $z^1_1,z^2_1,z^3_1,\dots$ have the same distribution. It remains to prove that they are independent with exponential distribution of parameter $1$.
Let $(\tau_1,\tau_2,\ldots)$ denote the successive differences between particles located on $[0,\infty)$ at time $t$. By stationarity of the Hammersley process, $(\tau_1,\tau_2,\ldots)$ is a sequence of i.i.d. exponential random variables with mean $1$. Besides, we have 
$$
Z_i=(\tau_{I+i},\tau_{I+i+1}, \ldots).
$$ 
where $I$ counts the number of particles that jumped over the origin up to time $t$. 
\begin{figure}
\begin{center}
\includegraphics[height=4cm]{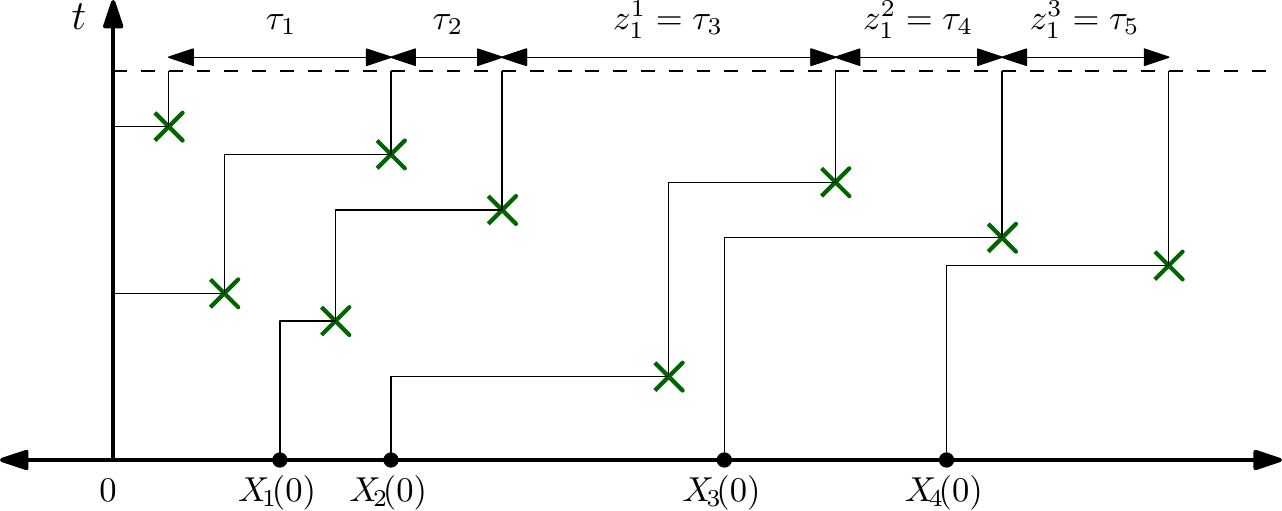}
\caption{Notations for the environment seen from a particle. Here, we have $I=2$.
}
\end{center}
\end{figure}
For any measurable event $A$ such that $\P\{(\tau_{j})_{j\ge 1} \in A\} > 0$  we have
\begin{equation}\label{eq:quelquechose}
\P\{(\tau_{I+j})_{j\ge 1} \in A\} = \P\{ Z_1\in A \} = \P\{ Z_i\in A \} = \P\{(\tau_{I+j+i})_{j\ge 1} \in A \}.
\end{equation}
Moreover
\begin{eqnarray*}
\P\{I=k, (\tau_{I+j+i})_{j\ge 1} \in A\} &=&\P\{I=k, (\tau_{k+j+i})_{j\ge 1} \in A\}\\
&=& \P\{ (\tau_{j})_{j\ge 1} \in A\}\P\{ I=k | (\tau_{k+j+i})_{j\ge 1} \in A\}.
\end{eqnarray*}
The tail $\sigma$-algebra  of the sequence $(\tau_j)$ being trivial, it follows from the backward martingale convergence theorem that
$$
\lim_{i\to \infty} \P\{I=k | (\tau_{k+j+i})_{j\ge 1} \in A\} = \P\{I=k\}.
$$
Thus, we obtain
$$\lim_{i\to \infty} \P\{I=k, (\tau_{I+j+i})_{j\ge 1} \in A\} = \P\{I=k\} \P\{(\tau_{j})_{j\ge 1} \in A\}.$$
Now, setting $f_k(i)\defeq\P\{I\le k, (\tau_{I+j+i})_{j\ge 1} \in A\}$, the previous equality gives
\begin{equation}\label{eq:limiteEnI}
\lim_{i\to \infty} f_k(i)= \P\{I\le k\}\P\{(\tau_{j})_{j\ge 1} \in A\}.
\end{equation}
Obviously, $k\mapsto f_k(i)$ is non-decreasing with
$$\lim_{k\to \infty} f_k(i)=\P\{ (\tau_{I+j+i})_{j\ge 1} \in A \} =\P\{ Z_1\in A\},$$
where we used \eqref{eq:quelquechose} for the last equality. Thus for all $i,k$,we have $f_k(i)\le \P\{Z_1\in A\}$, which, in view of \eqref{eq:limiteEnI}, implies that
$$\P\{(\tau_{j})_{j\ge 1} \in A\}=\lim_{k\to \infty} \lim_{i\to \infty} f_k(i)\le \P\{Z_1\in A\}.$$
Applying the same argument for $A^c$, yields the converse inequality hence
$$ \P\{Z_1\in A\} =\P\{ (\tau_{j})_{j\ge 1} \in A\},$$
which proves that  $Z_1$ is, indeed, a sequence of i.i.d. exponential random variables with mean $1$.
\end{proof}
\subsubsection{The case of the $\mu$-Hammersley process with geometric distribution}

We turn back to the case of the $\mu$-Hammersley process with geometric offspring distribution. As in the previous subsection, we want to describe the law of the (right) environment seen from a particle. 

Recall that  $\mathcal{C}_\lambda$ denotes a homogeneous PPP with intensity $\lambda$ on $\mathbb{R}$ where each atom is given a geometric number of lives of parameter $\alpha$. We consider the $\mu$-Hammersley process $H^{\mathcal{C}_\lambda}_{(-\infty,\infty)}$ on the whole line starting from the sources configuration $\mathcal{C}_\lambda$. Its graphical representation on $\mathbb{H}$ define a set of trees, each one rooted at an atom of $\mathcal{C}_\lambda$. We follow the trajectory of the left border of a tagged tree. More precisely, let $X(0)$ denote the position of the atom in $\mathcal{C}_0$ with minimal value on $[0,\infty)$ and for $t>0$, we define $X(t)$ to be the leftmost particle at time $t$ which is a descendent of $X(0)$. Keeping similar notation as in the previous section, we denote by $Z(t)$ the environment seen on the right of $X(t)$, \emph{i.e}
$$
Z(t)=(z^{1}(t),z^{2}(t),\ldots)
$$
is the sequence of gaps length between the particles on the right of $X(t)$ at time $t$, \emph{c.f.} Figure \ref{Fig:helpdefPointofview}.

\begin{figure}
\begin{center}
\includegraphics[height=4cm]{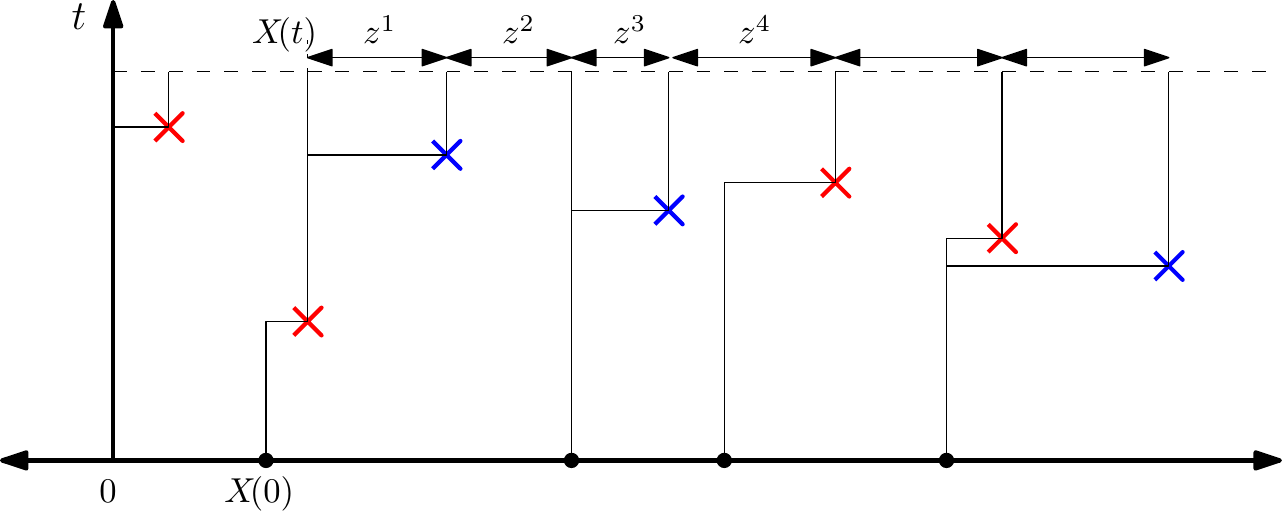}
\caption{\label{Fig:helpdefPointofview}Environment seen from the leftmost particle of the first tree on $[0,\infty)$.}
\end{center}
\end{figure}

\begin{prop}[\textbf{Environment seen from a particle: the geometric case}]\label{Prop:StationaritePalmGeometrique}
For every $t\ge 0$, $Z(t)$ is a sequence of i.i.d. exponential random variables with parameter $\lambda+(1-\alpha)t$.
\end{prop}

\begin{proof} Let us first point out that $Z$ is a Feller process. This follows from the fact that $H^{\mathcal{C}_\lambda}_{(-\infty,\infty)}$ itself has the Feller property. This, in turn, is an consequence of Theorem \ref{Theo:HalfPlane} which shows that we can construct the $\mu$-Hammersley process on the whole line by taking the limit over finite interval. As in Theorem \ref{Th:SourcesStationnaires}, the proof makes use of the fact the geometric $\mu$-Hammersley process may be seen as a superposition an Hammersley's line process and a PPP whose intensity increases linearly. Recall the red/blue construction described in Section \ref{Sec:GeometricCase}. Each atom of the PPP $\Xi$ is now coloured in red (resp. blue) with probability $\alpha$ (resp.  $1-\alpha$) independently of everything else.
\begin{itemize}
\item Red particles behave like in the usual Hammersley process by killing their left neighbour upon arrival.
\item Blue particles enter the process without affecting particles already present.
\end{itemize}
Since red and blue particles form two independent PPP, we can study their action on the infinitesimal generator $\mathcal{A}$ separately. On one hand, Proposition \ref{Prop:StationaritePalm} shows that red atoms leave invariant the law of the process $Z$
starting from any configuration of gaps $(z^1,z^2,\ldots)$ that is i.i.d. with exponential law. On the other hand, the blue atoms do not modify the position $X(t)$ but split each gap $z^i(t)$ at a uniform position, at rate $(1-\alpha)z^i(t)$. Thus, if we consider only the action of the blue atoms starting from an initial sequence of gaps $(z^1,z^2,\ldots)$ i.i.d. with exponential law of parameter $\lambda$, then, at all time $t$, then resulting configuration is still i.i.d. with exponential law of parameter $(1-\alpha)t + \lambda$. Putting these two facts together, we conclude that, $Z(t)$ is also, at all time, distributed as i.i.d. exponential random variables with parameter $\lambda+(1-\alpha)t$.
\end{proof}

Of course, $X(t)$ is non-decreasing in $t$. We can show that it converges to a finite limit a.s., which completes the proof of Assertion $3.$ (c) of Theorem \ref{Theo:HalfPlane}.

\begin{prop}\label{Prop:LimiteXt}
For every $\alpha\in (0,1)$
$$\lim_{t\to \infty} \E[X(t)-X(0)] = \frac{\alpha}{\lambda (1-\alpha)}.
$$
\end{prop}

\begin{proof} Let $f(t)\defeq \E[X(t)-X(0)]$. We are going to compute 
$$f'(t)=\lim_{h\to 0} \frac{1}{h}\E[X(t+h)-X(t)].$$
Fix $t\ge 0$. We estimate the variation of the process $X$ during a small time interval $[t,t+h]$. Blue atoms of $\Xi$ do not change the value of $X$. Only a red atom falling during the time interval $[t,t+h]$ between $X(t)$ and the first particle on its right will increase the value of $X$. Thus, denoting  by $\tau$ the distance between the particle located at $X(t)$ and its right neighbour and by $U$ a uniform random variable in $[0,1]$, independent of $\tau$, it is not difficult to check that
\begin{eqnarray*}
\E[X(t+h)-X(t)]&=&\E \Big[ U \tau \mathbf{1}_{\left\{\mbox{a red atom falls inside $[X(t),X(t+\tau)]\times [t,t+h]$}\right\}} \Big] + o(h)\\
&= & \E[\alpha U \tau^2] h + o(h).
\end{eqnarray*} 
Proposition \ref{Prop:StationaritePalmGeometrique} states that $\tau$ is an exponential random variable of parameter $\lambda+(1-\alpha)t$. Thus, we get 
$$\E[X(t+h)-X(t)] = \frac{\alpha h }{(\lambda+(1-\alpha)t)^2}+ o(h).$$
Taking $h\to0$ yields $f'(t)=\frac{\alpha }{(\lambda+(1-\alpha)t)^2}$. Therefore, we conclude that
$$\lim_{t\to\infty}f(t)= \lim_{t\to\infty} \frac{\alpha t}{\lambda^2+\lambda (1-\alpha)t}  = \frac{\alpha}{ \lambda (1-\alpha)}.$$
\end{proof}

\begin{cor}\label{arbrepuit} Consider the set $\mathcal{E}$ of particles in $\mathcal{C}_\lambda \cap (-\infty,0]$ whose genealogical tree ultimately enters $[0,+\infty)$. We have, $\E[|\mathcal{E}|]<\infty$.
\end{cor}

\begin{proof}
Let $\ldots < x_2 <x_1 < 0$ be the locations of the particles in $\mathcal{C}_\lambda \cap (-\infty,0]$.
\begin{eqnarray*}
\P\{|\mathcal{E}| > n\}&=& \P\{\mbox{the tree rooted at $x_{n+1}$ eventually enters $[0,+\infty)$} \}\\
&\leq& \P\{\mbox{the leftmost particle of the tree rooted at $x_{n}$ eventually enters $[0,+\infty)$} \}\\
&\le & \P\{x_n\ge -n/2\}+\P\Big\{
\begin{array}{c}
\hbox{the leftmost particle of the tree rooted at $x_n$}\\
\hbox{eventually moves by at least $n/2$}
\end{array}
\Big\}\\
&=& \P\{\mathrm{Poisson}(n/2)\ge n\}+\P\Big\{\lim_{t\to\infty}(X(t) - X(0)) \geq n/2\Big\}.
\end{eqnarray*}
One the one hand, we have, for large $n$,
$$\P\{\mathrm{Poisson}(n/2)\ge n\}\le \left(\frac{e^{1/2}}{2}\right)^n.$$
On the other hand, Proposition \ref{Prop:LimiteXt} implies that
$$\sum_n \P\Big\{\lim_{t\to\infty}(X(t) - X(0)) \geq n/2\Big\} <\infty.$$
Therefore $\E[|\mathcal{E}|] = \sum_n \P\{|\mathcal{E}|>n|\} < +\infty$.
\end{proof}

\subsection{Asymptotic of $R(t)$ and $\mathbf{R}(n)$ in the geometric case}

Recall that $R(t)$ is the number of roots (\emph{i.e.} trees) for the $\mu$-Hammersley process $H$ on $[0,1]$ up to time $t$ when there is neither source nor sink. The aim of this section is to prove the first part of Theorem \ref{theocontinu}.
\begin{prop}[\textbf{Theorem \ref{theocontinu} (i) for $R(t)$}]
For a geometric offspring distribution $\mu$ with parameter $\alpha$, we have
$$
\lim_{t\to\infty} \frac{R(t)}{\log t} = \frac{1}{1-\alpha} \quad\text{ a.s. and in $L^1$}.
$$
\end{prop}

\begin{proof}

We first compute the asymptotic number of trees in the stationary case when sources and sinks are added and then compare it to $R(t)$. Let $R^{\mathcal{C}_1,\mathcal{S}_1}(t)$ be the number of trees in $[0,1]\times [0,t]$ defined by the $\mu$-Hammersley process with a geom$(\alpha)$ number of lives, when 
\begin{itemize}
\item $\mathcal{C}_1$ is distributed as a PPP with intensity $1$ on $[0,1]$ and each source is given an independent geometric number of lives with parameter $\alpha$. 
\item $\mathcal{S}_1$ is a PPP with intensity $\frac{ds}{1+(1-\alpha)s}$ on $ [0,+\infty)$.
\end{itemize}
Notice that a tree can start either at a root on the $y$-axis or at a source on the $x$-axis (see for example Figure \ref{Fig:SourcesStationnaires}) thus
$$
R^{\mathcal{C}_1,\mathcal{S}_1}(t)= |\{\hbox{roots on the $y$-axis with height smaller than $t$}\}| + |\mathcal{C}_1|.
$$
The number of sources $|\mathcal{C}_1|$ is a Poisson random variable with parameter $1$ that does not depend on $t$. By Theorem \ref{Th:RacinesStationnaires}, the roots on the $y$-axis are distributed as an inhomogeneous PPP with intensity $\frac{ds}{1+(1-\alpha)s}$. Thus, the number of roots is a Poisson random variable with parameter
$$
\int_0^t \frac{ds}{1+(1-\alpha)s}=\frac{1}{1-\alpha} \log((1-\alpha)t),
$$
from which we easily deduce that
$$\frac {R^{\mathcal{C}_1,\mathcal{S}_1}(t)}{\log t} \rightarrow \frac{1}{1-\alpha} \qquad a.s. \mbox{ and in } L^1.$$ 

To prove the upper bound for $R(t)$, we compare $R(t)$ with  $R^{\mathcal{C}_1,\mathcal{S}_1}(t)$ when both processes are constructed using the same PPP $\Xi$. Clearly, adding sinks only increases the number of trees. A source with $\nu$ lives merges at most $\nu$ trees together, thus decreases the number of trees by, at most $\nu-1$. Therefore, we get the upper bound
\begin{equation}\label{eq:theolowerbound} 
R(t)\leq R^{\mathcal{C}_1,\mathcal{S}_1}(t)+ \sum_{(u,\nu) \in \mathcal{C}_1} (\nu-1).
\end{equation}
The number of sources being $\mathrm{Poisson}(1)$ and each of them having a geometric number of lives, we deduce that
\begin{equation*} 
\limsup_{t \to \infty} \frac{R(t)}{\log t}\le \lim_{t \to \infty} \frac{R^{\mathcal{C}_1,\mathcal{S}_1}(t)}{\log t} = \frac{1}{1-\alpha}\quad\hbox{a.s.}
\end{equation*}
The lower bound is more delicate. Recall that $\Xi$ denotes the PPP used to construct the process $R(t)$. We now consider another copy $\Xi'$ independent of $\Xi$. The sources $\mathcal{C}_1$ and sinks $\mathcal{S}_1$ processes are assumed to be independent of $\Xi$ and $\Xi'$. Let  $H'^{\mathcal{C}_1,\mathcal{S}_1}$ denote the stationary $\mu$-Hammersley process with sources and sinks constructed from $\Xi'$.  Let $\mathcal{T}'$ denote the set of trees in the graphical representation of $H'^{\mathcal{C}_1,\mathcal{S}_1}$ on $[0,1]\times[0,\infty)$ which contain a source or a sink. Since the number of sources is finite, the number of trees that contain a source is also finite. Besides, the number of trees that contain a sink has the same law as the random variable $|\mathcal{E}|$ introduced in Corollary \ref{arbrepuit}  (with $\lambda=1$). This proves that $\E[|\mathcal{T}'|] < +\infty$ so that $|\mathcal{T}'|<\infty$ is finite a.s. Consider now the set $\mathcal{A}$ of atoms of $\Xi'\cup \mathcal{C}_1\cup \mathcal{S}_1$ which belong to a tree in $\mathcal{T}$. Let $f:[0,1]\mapsto [0,\infty]$ denote the lowest non-decreasing function such that $\mathcal{A}$ belongs to its hypograph. See Figure \ref{Fig:proofps} for an illustration. 
\begin{figure}[!ht]
\begin{center}
\includegraphics[height=10.5cm]{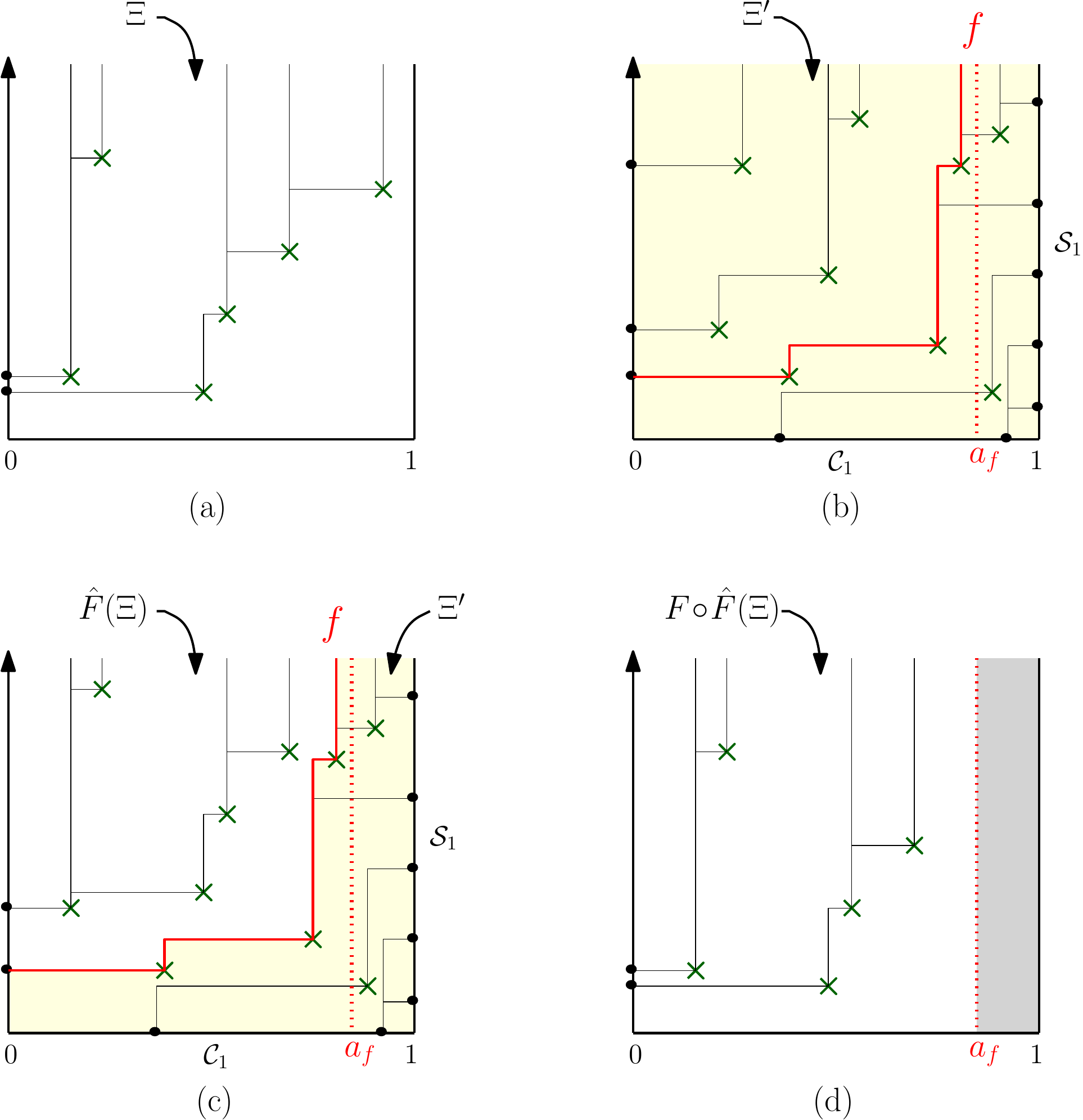}
\caption{\label{Fig:proofps}Proof of the lower bound. (a) The initial graphical representation constructed with $\Xi$. (b) The representation defined by the auxiliary PPP $\Xi'$. The function $f$ is in red and its asymptote at $a_f$ is signified by the dotted red line (here, $\mathcal{T}' =3$) . (c) The graphical representation with $\hat{\Xi}$, the yellow part is the hypograph of $f$ and contains the atoms of $\Xi'$. The upper part is the image by $\hat{F}$ of the atoms of the initial PPP $\Xi$. (d) Going back to $\Xi$ using the mapping $F$.}
\end{center}
\end{figure}
As in Proposition \ref{Prop:couplagesousdomaine}, set
$$a_f \defeq \inf\{x,\; f(x) = +\infty\}$$ and 
define the domain $\mathcal{D}_f  \defeq \{ (x,t) \in (0,a_f)\times (0,+\infty),\; t > f(x)\}$ and its complement  $\mathcal{D}^c_f  \defeq \{ (x,t) \in (0,1)\times (0,+\infty),\; t \le f(x)\}$. Consider
the mapping
\begin{equation*}
\hat{F} : \begin{array}{ccc}
(0,a_f)\times(0,+\infty) \times \mathbb{N} &\to& \mathcal{D}_f \times \mathbb{N}  \\
(x,t,\nu) &\mapsto& (x,t+f(x),\nu).
\end{array}
\end{equation*}
We define a new set of points $\hat{\Xi}\subset (0,1)\times (0,\infty)\times \N$ constructed from $\Xi$ and $\Xi'$ by keeping the atoms of $\Xi'$ on $\mathcal{D}^c_f$ and the atoms of $\hat{F}(\Xi)$ on $\mathcal{D}_f$:
$$
\hat{\Xi} \defeq \left(\Xi'\cap (\mathcal{D}^c_f\times \N)\right)\cup \hat{F}(\Xi).
$$
Let $\hat{\mathcal{T}}$ be the set of trees in the graphical construction using the atoms of $(\hat{\Xi},\mathcal{C}_1,\mathcal{S}_1)$ that contain a source or a sink. We remark that, in fact, $\hat{\mathcal{T}} = \mathcal{T}'$. This means that, given $\hat{\Xi}$ together with $\mathcal{C}_1$ and $\mathcal{S}_1$, we can recover $f$ and the set $\Xi'\cap (\mathcal{D}^c_f\times \N)$.  In particular, given $f$, the point process $\Xi'$ restricted to $\mathcal{D}_f\times \N$ is independent of the triplet $(\Xi'\cap (\mathcal{D}^c_f\times \N),\mathcal{C}_1,\mathcal{S}_1)$ and has the law of a PPP with unit spatial intensity on $\mathcal{D}_f\times \N$. Hence, we get the equality in law 
$$\Big(\Xi'\cap \{\mathcal{D}_f\times \N\},\;\Xi'\cap \{\mathcal{D}^c_f\times \N\},\;\mathcal{C}_1,\;\mathcal{S}_1\Big)\overset{\mbox{\tiny{law}}}{=} 
\Big(\hat{F}(\Xi),\;\Xi'\cap \{\mathcal{D}^c_f\times \N\},\;\mathcal{C}_1,\;\mathcal{S}_1\Big).$$
In particular, we deduce that $\hat{\Xi}$ is a PPP on $(0,1)\times (0,\infty) \times \N$ with unit spatial intensity and is independent of the sources and sinks processes $\mathcal{C}_1$ and $\mathcal{S}_1$. Hence, the process $\hat{H}^{\mathcal{C}_1,\mathcal{S}_1}$ constructed from the atoms of $\hat{\Xi}$ and with sources $\mathcal{C}_1$ and sinks $\mathcal{S}_1$ has the law of the stationary $\mu$-Hammersley process with sources and sinks. Consider now the number of roots $\hat{R}^{\mathcal{C}_1,\mathcal{S}_1}(t)$ created before time $t$. We can write
\begin{equation}\label{Eq:mino1}
\hat{R}^{\mathcal{C}_1,\mathcal{S}_1}(t)= |\{\text{trees of $\hat{\mathcal{T}}$ created before $t$}\}|+ |\{\text{trees not in  $\hat{\mathcal{T}}$ created before $t$}\}|.
\end{equation}
As in Proposition \ref{Prop:couplagesousdomaine}, we use the notation $R(\xi)$ to denote the number of roots in the graphical representation constructed from a finite set of atoms $\xi$. We also set $B_t \defeq [0,1]\times[0,t]$. With this convention, the number of roots created before time $t$ that do not belong to  trees in $\hat{\mathcal{T}}$ is $R(\hat{F}(\Xi)\cap B_t)$. 
Define 
\begin{equation*}
F : \begin{array}{ccc}
 \mathcal{D}_f \times \mathbb{N} &\to& (0,a_f)\times(0,+\infty) \times \mathbb{N} \\
(x,t,\nu) &\mapsto& (x,t-f(x),\nu),
\end{array}
\end{equation*} 
so that $F\circ \hat{F}=\mbox{Id}$ on $(0,a_f)\times(0,+\infty) \times \mathbb{N}$. According to Proposition \ref{Prop:couplagesousdomaine}, we have 
\begin{equation}\label{Eq:mino2}
R(\hat{F}(\Xi)\cap B_t)\le R(\Xi\cap \{(0,a_f)\times(0,t)\}). 
\end{equation} 
Moreover, as explained in Section \ref{sect-augmentationdroite}, adding the atoms in the strip $[a_f,1)\times (0,t)$ can only increase the number of roots hence
\begin{equation}\label{Eq:mino3}
R(\Xi\cap \{(0,a_f)\times(0,t)\})\le R(\Xi\cap \{(0,1)\times(0,t)\})=R(t). 
\end{equation}
Combining \eqref{Eq:mino1}, \eqref{Eq:mino2} and \eqref{Eq:mino3}, we deduce that
\begin{equation}\label{eq:theoupperbound} 
R(t)\ge \hat{R}^{\mathcal{C}_1,\mathcal{S}_1}(t)- |\hat{\mathcal{T}}|
\end{equation}
which, combined with $\E[|\hat{\mathcal{T}}|]<\infty$, shows that
$$\liminf_{t \to \infty} \frac{R(t)}{\log t}\ge  \lim_{t \to \infty} \frac{\hat{R}^{\mathcal{C}_1,\mathcal{S}_1}(t)}{\log t}=\frac{1}{1-\alpha} \quad\hbox{a.s.}$$
The combination of \eqref{eq:theolowerbound} and \eqref{eq:theoupperbound} also implies the $L^1$ convergence. 
\end{proof}

The proof of Theorem \ref{theogeometric} counting the asymptotic number of geometric heaps required by the heap sorting algorithm is now an easy consequence of the previous estimate on $R$.  

\begin{proof}[Proof of Theorem \ref{theogeometric}.]
Recall that $R$ and $\mathbf{R}$ are time changed of each other such that
$$
\mathbf{R}(n)=R({t(n)}),
$$
where
$$t(n)\defeq\inf\left\{t\ge 0, \mathrm{card}\left(\Xi\cap [0,1]\times [0,t]\times \N\right)=n\right\}.$$
Since $\Xi$ has unit Lebesgue intensity, we have $t(n)/n\to 1$ a.s. which directly implies
$$
\frac{\mathbf{R}(n)}{\log n}=\frac{R(t(n))}{\log t(n)}\frac{\log t(n)}{\log n}\underset{n\to\infty}{\longrightarrow} \frac{1}{1-\alpha} \quad\text{a.s.}
$$
Concerning the $L^1$ convergence, we can write, in view of \eqref{eq-discretcontinu}, that 
$$R(t)=\mathbf{R}(K(t))$$
where $K(t)$ is a Poisson Process with rate $1$ independent of $\mathbf{R}$. Since $R(t)/\log t$ converges to $1/(1-\alpha)$ in $L^1$, we have
$$\lim_{t \to +\infty}\E\left[\left(\frac{R(t)}{\log t} -\frac{1}{1-\alpha}\right)^{\!+}\right] = 0.
$$
On the other hand, we can write
\begin{eqnarray*}
\E\left[\left(\frac{R(t)}{\log t} -\frac{1}{1-\alpha}\right)^{\!+}\right]& \ge  & \E\left[\left(\frac{\mathbf{R}(K(t))}{\log t} -\frac{1}{1-\alpha}\right)^{\!+} \mathbf{1}_{\{K(t)\ge t\}}\right]\\
&\ge & \E\left[\left(\frac{\mathbf{R}(\lfloor t \rfloor)}{\log t} -\frac{1}{1-\alpha}\right)^{\!+} \mathbf{1}_{\{K(t)\ge t\}}\right]\\
& = & \E\left[\left(\frac{\mathbf{R}(\lfloor t \rfloor)}{\log t} -\frac{1}{1-\alpha}\right)^{\!+}\right]\P\{ K(t)\ge t\}.
 \end{eqnarray*} 
Since $\P\{K(t)\ge t\} \to 1/2$ as $t$ tend to infinity, we deduce that 
$$
\lim_{n \to +\infty}\E\left[\left(\frac{\mathbf{R}(n )}{\log n} -\frac{1}{1-\alpha}\right)^{\!+}\right] = 0.
$$
Combining this fact with Lemma \ref{lem:cvesperance} shows that the convergence also holds in  $L^1$.
\end{proof}

\section{General trees and regular trees}\label{Sec:CasRegulier}

\subsection{Proof of Theorems \ref{theoesperance} and \ref{theoregulier}}\label{SousSec:PreuveCasRegulier}

Suppose that $\mu$ and $\mu'$ are two offspring distributions such that 
\begin{equation}\label{eq:inegE}
\mathbb{E}[\nu]\leq \mathbb{E}[\nu'].
\end{equation}
It is natural to expect that the heap sorting algorithm of a sequence $\mu'$-distributed should require, in average, less trees than the sorting of a $\mu$-distributed sequence \emph{i.e.} $c_{\mu'} \leq c_{\mu}$. Under the stronger assumption that $\mu$ is stochastically dominated by $\mu'$, this result is straightforward by a trivial coupling. However, it is false in general as we can find two offspring distributions $\nu$ and $\nu'$ such that \eqref{eq:inegE} holds yet $c_\mu < c_{\mu'}$ (\emph{c.f.} Remark \ref{rem:contrex}). 

Even though we cannot order the value of the constants $c_\mu$ between general offspring distributions, we can still prove that, among all reproduction laws with a given expectation, the one with smallest variance is the one that minimizes the average number of trees required. This result is the key to bootstrapping the estimates obtained for the geometric law to general offspring distributions.  

\begin{prop}\label{propcouplage} Let $\mu$ and $\mu'$ denote two offspring distributions on $\N$ with the same expectation. Suppose also that $\mbox{supp}(\mu') \in \{\ell,\ell+1\}$ for some $\ell\in\N$.
Fix $n>0$ and  a finite sequence of labels $(u_i)_{1\le i \le n}$  in $(0,1)$ such that $u_i\neq u_j$ for $i\neq j$. Let $\mathcal{V}=(\nu_i)_{1\le i\le n}$ and  $\mathcal{V}'=(\nu'_i)_{1\le i\le n}$ be two independent sequences of i.i.d. random variables with respective law $\mu$ and $\mu'$.
 We denote by $\mathbf{R}_{\mathcal{V}}$ (resp. $\mathbf{R}_{\mathcal{V}'}$) the number of trees obtained by the heap sorting algorithm for the sequence of  $(u_i,\nu_i)_{1\le i \le n}$  (resp. $(u_i,\nu'_i)_{1\le i \le n}$). Then we have 
\begin{equation}\label{eq:coulpres}
\E[\mathbf{R}_{\mathcal{V}'}]\le \E[\mathbf{R}_{\mathcal{V}}].
\end{equation}
\end{prop}

We emphasize that the sequence of labels $(u_i)$ in the proposition above is deterministic. Hence, the expectation in \eqref{eq:coulpres} is taken only on the randomness of the number of lives $\mathcal{V}$ and $\mathcal{V}'$.

In order to prove this result, we first compare the number of trees needed when only the number of lives of a single term of the sequence $(u_i)$ differs. More precisely, fix $i_0 \in \{1,\ldots,n\}$ and a deterministic sequences of positive integers $(v_i)_{i \neq i_0}$. For $m \in \N$ we denote by $\mathbf{R}_m$ the number of trees required by the heap sorting algorithm to sort the sequence $(u_i, v^m_i)$ where each label $u_i$ has $v_i$ lives except label $u_{i_0}$ which is given $m$ lives:
$$
v^m_i \defeq \left\{
\begin{array}{ll}
v_i & \hbox{if $i\neq i_0$,}\\
m & \hbox{if $i = i_0$.}
\end{array}
\right.
$$

\begin{lem}\label{lemmecouplage} The (deterministic) sequence $(\mathbf{R}_m)_{m\ge 1}$ has the following properties
\begin{itemize}
\item[(i)] For all $m>1$, we have $\mathbf{R}_{m}-\mathbf{R}_{m-1}\in \{-1,0\}$.
\item[(ii)] If $\mathbf{R}_{m}-\mathbf{R}_{m-1}=0$ then for all $m'\ge m$, $\mathbf{R}_{m'}-\mathbf{R}_{m'-1}=0$.
\end{itemize}
\end{lem}
\begin{proof} By construction the heap sorting algorithm yields the minimum number of trees needed to sort a sequence of labels with a prescribed numbers of lives. Thus, it is clear that removing one life from a label can either leave the number of trees unchanged or increase this number by one. Hence, Point (i) is straightforward.

To prove (ii), we use a strategy similar to that explained in Section \ref{sect-augmentationdroite}. For $m\ge 1$, consider the graphical representation of the sequence of atoms $(u_i,i,v^m_i)_{1\le i\le n}$ in $[0,1]\times[0,n+1]\times\N$. We call these graphical representations $\mathcal{G}^m$. For $m\ge 2$, we define a \emph{second class particle of index $m$} that tracks the modifications required to compute the graphical representation $\mathcal{G}^{m-1}$ from $\mathcal{G}^{m}$. This second class particle has the following trajectory inside the box $[0,1]\times [0,n+1]$:
\begin{itemize}
\item It starts from the point $(u_{i_0}, i_0 )$.
\item It moves upwards until it either reaches the top of the box or until the vertical line disappears (or becomes dotted if dead particles are represented).
\item Upon reaching the end of a vertical line, it starts moving horizontally to the left until it either reaches the left side of the box or a solid vertical line. In the latter case, it starts again moving upwards and the procedure above is repeated.
\end{itemize}
Just as in Section \ref{sect-augmentationdroite}, it is easy to see that $\mathcal{G}^{m-1}$ may be obtained from $\mathcal{G}^m$ by altering the graphical representation along the path of this particle.  In particular, 
\begin{itemize}
\item If the second class particle of index $m$ exits the box by its top side then $\mathbf{R}_{m-1}=\mathbf{R}_{m}$.
\item If the second class particle of index $m$ exits the box by its left side then $\mathbf{R}_{m-1}=\mathbf{R}_{m}+1$.
\end{itemize}
We now claim that, for any $m$, the trajectory of the second class particle with index $m-1$ is always below that of the second class particle with index $m$. Indeed, the second class particle moves upwards as soon as it crosses a solid vertical line of the graphical representation. Moreover, if we look at the differences between the graphical representations $\mathcal{G}^{m}$ and  $\mathcal{G}^{m-1}$, we see that a solid vertical line may become dotted but the converse is impossible (see Figure \ref{fig:Monotonie2dClasse} for an example). The combination of these two remarks proves the claim. Therefore, if the second class particle of index $m-1$ ends at the top of the box, then the second class particle with index $m$ does too, which proves (ii). 
\begin{figure}
\begin{center}
\includegraphics[height=5cm]{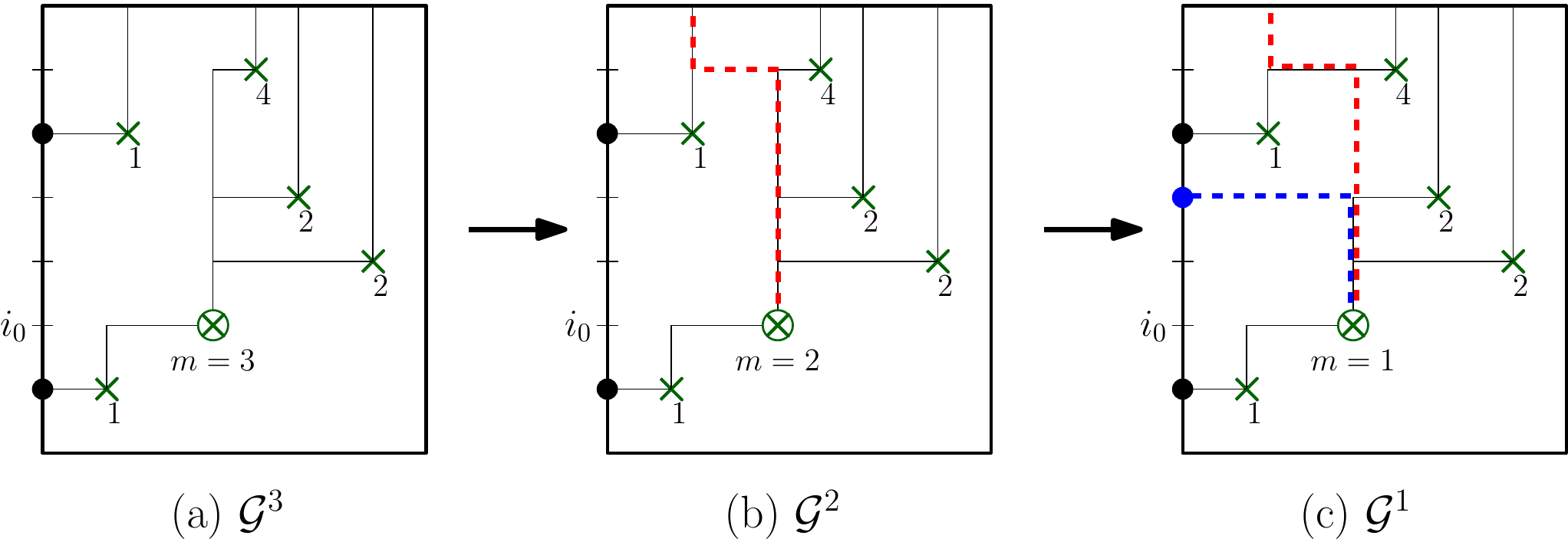}
\caption{Illustration of Lemma \ref{lemmecouplage}. (a) The graphical representation for $m=3$. (b) The second class particle of index $3$ (in dotted red) exits by the top side: no trees are created. (c) The second class particle of index $2$ (in dotted blue) stays below the red trajectory and exits by the left side: a new tree is created.
\label{fig:Monotonie2dClasse}}
\end{center}
\end{figure}
\end{proof}

\begin{proof}[Proof of Proposition \ref{propcouplage}]
As in  Lemma \ref{lemmecouplage},  fix $i_0\in \{1,\ldots,n\}$ and some deterministic sequence of positive integers $(v_i)_{i\neq i_0}$. Let $\nu$ (resp. $\nu'$) be a random variable with distribution $\mu$ (resp. $\mu'$). Denote by $\mathbf{R}_{\nu}$ (resp. $\mathbf{R}_{\nu'}$) the number of trees required to sort the sequence $(u_i)_{1\le i \le n}$ when the number of lives of $u_i$ is $v_i$ for $i\neq i_0$ and $\nu$ (resp. $\nu'$) for $i=i_0$ . We first prove that 
\begin{equation}\label{eq-inegalitevie}
\E[\mathbf{R}_{\nu'}]\le \E[\mathbf{R}_{\nu}].
\end{equation}
To see why this holds, consider the (deterministic) sequence $(\mathbf{R}_m)_{m\ge 1}$ defined in Lemma \ref{lemmecouplage}. We have two cases. 
\begin{itemize}
\item[(a)] $\mathbf{R}_{\ell}=\mathbf{R}_{\ell+1}$. Using that $\mbox{supp}(\nu') \in \{\ell,\ell+1\}$, we get $\E[\mathbf{R}_{\nu'}]=\mathbf{R}_\ell$. Besides, according to (ii) of Lemma \ref{lemmecouplage}, we have  $\mathbf{R}_m\ge \mathbf{R}_\ell$ for all $m\ge 1$. Thus, $\E[\mathbf{R}_{\nu}] \ge \mathbf{R}_\ell =\E[\mathbf{R}_{\nu'}]$.
\item[(b)] $\mathbf{R}_{\ell+1}=\mathbf{R}_{\ell}-1$. According to the previous lemma, we have in this case $\mathbf{R}_\ell=\mathbf{R}_1+1-\ell$. Thus, one the one hand, we find that
\begin{eqnarray*}
\E[\mathbf{R}_{\nu'}] &=& \P\{\nu'=\ell\}\mathbf{R}_\ell+\P\{\nu'=\ell+1\}\mathbf{R}_{\ell+1}\\
&=&\mathbf{R}_\ell-\P\{\nu'=\ell+1\}\\
&=&\mathbf{R}_1+1-\E[\nu'].
\end{eqnarray*}
On the other hand, using (i) of the previous lemma, we get that, for all $m\ge 1$, 
$$\mathbf{R}_m\ge \mathbf{R}_1-(m-1).$$
Thus
\begin{equation*}\label{eqminNV}
\E[\mathbf{R}_{\nu}]\ge \E[\mathbf{R}_1-(\nu-1)]=\mathbf{R}_1+1-\E[\nu].
\end{equation*}
We deduce inequality \eqref{eq-inegalitevie} recalling that $\E[\nu] = \E[\nu']$.
\end{itemize}

More generally, given a sequence of positive integers $(v_1,\ldots,v_n)$, denote by $\mathbf{R}_{(v_1,\ldots,v_n)}$ the number of trees required to sort the sequence $(u_i,v_i)_{1\le i \le n}$. We have 
\begin{eqnarray*}
\E[\mathbf{R}_{\mathcal{V}}]&=&\sum_{v_1,\ldots,v_n\ge 1} \prod_{i=1}^n \P\{\nu_i=v_i\}\mathbf{R}_{(v_1,\ldots,v_n)}\\
&=& \sum_{v_2,\ldots,v_n\ge 1} \prod_{i=2}^n \P\{\nu_i=v_i\}\E[\mathbf{R}_{(\nu,v_2,\ldots,v_n)}]\\
&\ge & \sum_{v_2,\ldots,v_n\ge 1} \prod_{i=2}^n \P\{\nu_i=v_i\}\E[\mathbf{R}_{(\nu',v_2,\ldots,v_n)}]\\
&= & \sum_{v_1,\ldots,v_n\ge 1} \P\{\nu'_1=v_1\} \prod_{i=2}^n \P\{\nu_i=v_i\}\mathbf{R}_{(v_1,\ldots,v_n)}
\end{eqnarray*}
where we used \eqref{eq-inegalitevie} for the inequality. Hence, we conclude, by induction, that
\begin{equation*}
\E[\mathbf{R}_{\mathcal{V}}] \ge \sum_{v_1,\ldots,v_n\ge 1}  \prod_{i=1}^n \P\{\nu'_i=v_i\}\mathbf{R}_{(v_1,\ldots,v_n)}=\E[\mathbf{R}_{\mathcal{V}'}]
\end{equation*}
\end{proof}

\begin{rem}\label{rem:contrex} We give an example of two offspring distributions $\mu$ and $\mu'$ such that 
$$
\E[\nu] < \E[\nu'] \quad\hbox{yet}\quad c_\mu < c_{\mu'}.
$$  
Let $\mu$ to be the geometric distribution with parameter $\frac{4}{21}$. Let $\mu'$ be the distribution such that $\mu'(1) = \mu'(10) = \frac{1}{2}$. Then, we have $\E[\nu] = \frac{21}{4} < \frac{11}{2} = \E[\nu']$.
We can lower bound $c_{\mu'}$ using Equation \eqref{limDn} and the fact that $\E[D_n]$ is non-decreasing. In particular, we have $\E[D_2] = \frac{1}{4}$ hence $c_{\mu'} \geq \frac{5}{4}$. But, on the other hand, we have $c_\mu = \frac{21}{17}$ according to Theorem \ref{theogeometric}. 
\end{rem}

\begin{proof}[Proof of Theorems \ref{theoesperance}, \ref{theoregulier} and (ii) of Theorem \ref{theocontinu}] According to Lemma \ref{lem:cvesperance}, for any offspring distribution $\mu$, the constant 
$$c_\mu \defeq \lim_{n\to\infty} \frac{\E[\mathbf{R}(n)]}{\log n}= \lim_{t\to\infty} \frac{\E[R(t)]}{\log t}\in (1,\infty]$$ 
exists. It remains to establish an upper bound. We first consider the case of the regular trees $\mu=\delta_k$ and show that $c_\mu$ is no larger than $k/(k-1)$.

Let $\mathcal{U}=(U_i)_{1\le i \le n}$ be an i.i.d. sequence of uniform random variables on $[0,1]$. Denote by $\mathbf{R}(n)$ the number of trees required by the heap sorting algorithm for the sequence $(U_i,\nu_i)_{1\le i \le n}$ when the number of lives $(\nu_i)_{1\le i \le n}$ are i.i.d with geometric distribution with mean $k$. Similarly, let $\mathbf{R}'(n)$ be the number of trees needed to sort the same sequence $\mathcal{U}$ when each label $U_i$ has exactly $k$ lives. Using Proposition \ref{propcouplage}, we have for every realization of $\mathcal{U}$, 
$$\mathbf{R}'(n)\le \E[\mathbf{R}(n) | \mathcal{U}].$$
Recalling that $\E[\mathbf{R}(n)]/\log n$ converges to $1/(1-1/k)$ and taking the expectation with respect to $\mathcal{U}$ yield
$$\limsup_{n\to \infty} \frac{\E[\mathbf{R}'(n)]}{\log n}\le \frac{1}{1-1/k}$$
which concludes the proof of Theorem \ref{theoregulier}.

It remains to prove Theorem \ref{theoesperance} \emph{i.e.} to show that $c_\mu$  is finite for any reproduction law $\mu$ on $\mathbb{N}$ that is not the Dirac mass at $1$. Fix such an offspring distribution and let $\mu'$ be the probability distribution on $\{1,2\}$ such that $\mu'(1)=\mu(1)$. Let also $\mu''$ denote the geometric law with same expectation as $\mu'$. Finally, let $\mathbf{R}(n)$, $\mathbf{R}'(n)$ and $\mathbf{R}''(n)$ denote the number of trees required to sort the sequence $\mathcal{U}$ when the reproduction law is respectively $\mu$, $\mu'$ and $\mu''$. Since $\mu$ dominates $\mu'$, we can couple these processes on the same space in a trivial way such that $\mathbf{R}(n)\le \mathbf{R}'(n)$ a.s.  Using again  Proposition \ref{propcouplage} and Theorem \ref{theogeometric}, we conclude that $$\E[\mathbf{R}(n)]\le \E[\mathbf{R}'(n)]\le \E[\mathbf{R}''(n)]\sim \frac{\log n}{1-1/c} $$ where $c$ is the expectation of $\mu''$.

\end{proof}

\subsection{Half plane representation for any measure $\mu$}

In Section \ref{sec:halfplanegeom}, we showed that, in the geometric  case, if we take the graphical representations $\mathcal{G}_b$ on quarter planes $\mathbb{H}_b \defeq (-\infty,b]\times (0,\infty)$, then, as $b\to+\infty$, the representations converge locally to a random stationary representation on the whole upper half-plane $\mathbb{H}\defeq \R\times (0,\infty)$. We now prove that the same result holds for any offspring distribution.

\begin{theo}[\textbf{Stationary half plane representation: general offspring distribution}]\label{Theo:HalfPlaneGeneral}
Assume that $\mu\neq \delta_1$ and recall that $\Xi$ denotes a PPP on $\mathbb{H}\times\N$ with intensity $du \otimes dt \otimes \mu$. Then, as $b$ goes to $+\infty$, the graphical representations $\mathcal{G}_b$ on $\mathbb{H}_b$ converge locally, almost surely, to a random graphical representation $\mathcal{G}_\infty$ on $\mathbb{H}$. That is, for a.s. every realization of $\Xi$ and for any closed box $\mathbb{B} \defeq [x,y]\times [s,t] \subset \mathbb{H}$, the restriction of $\mathcal{G}_b$ to $\mathbb{B}$ becomes constant for $b$ large enough. The random variable $\mathcal{G}_\infty$ has the following properties:
\begin{enumerate}
\item The law $\mathcal{G}_\infty$ is invariant by horizontal translation. 
\item $\mathcal{G}_\infty$ defines almost surely a single infinite random tree embedded in $\mathbb{H}$ such that
\begin{itemize}
\item[\textup{(a)}] The set of vertices of the tree is exactly the positions of the atoms of $\Xi$ and the number of children of each vertex is the initial 
number of lives of the atom. 
\item[\textup{(b)}] The tree is rooted at $(-\infty,0)$. It grows upward and to the right.
\end{itemize}
\end{enumerate}
\end{theo}

\begin{proof}As in the proof of Theorem \ref{Theo:HalfPlane},
in order to prove that the restriction of  $\mathcal{G}_b$ to  a finite box $\mathbb{B}$ becomes constant for large $b$,  we only need to show that, for any $0<s<t$, the set of heights   $\bar{\mathcal{R}}_{[s,t]}(b)$ of the horizontal lines intersecting the vertical segment $\{0\}\times [s,t]$ in $\mathcal{G}_b$ becomes constant for large $b$.
Furthermore, as in the geometric case, $\bar{\mathcal{R}}_{[s,t]}(b)$ is non decreasing in $b$ for the partial inclusion order thus converges to some limit. Hence, we just  need to prove that the number of these horizontal lines remains bounded as $b$ tends to infinity.

For any fixed $b\ge 0$, using a space-time dilation of the box $[0,b]\times [0,t]$ to map it to the box $[0,1] \times [0,bt]$, we see that the number of lines  $|\bar{\mathcal{R}}_{[s,t]}(b)|$ is equal in law to $R(bt)-R(bs)$, the number of roots which appear between times $bs$ and $bt$ in an $\mu$-Hammersley defined on the strip $[0,1]\times [0,\infty)$. We now show that 
  $$ \limsup_{b \to \infty}\E[R(bt)-R(bs)]<\infty,$$
which in turn will imply 
$$ \lim_{b \to \infty} |\bar{\mathcal{R}}_{[s,t]}(b) |<\infty \quad {a.s.}$$
proving the convergence of the graphical representation as $b$ tends to infinity. 
According to \eqref{eq-discretcontinu}, we  can write  $$R(bt)-R(bs)=\mathbf{R}(K(bt))-\mathbf{R}(K(bs)),$$ where $K(\cdot)$ is a Poisson process with unit intensity independent of $\mathbf{R}(\cdot)$. 

Recall the recurrence equation \eqref{recDn} between $\E[\mathbf{R}(n+1)]$ and $\E[\mathbf{R}(n)]$. Using that $\E[D_n+1]$ is non-decreasing with $n$ and converges to $c_\mu< \infty$ (\emph{c.f.} Equation \eqref{limDn} and Theorem \ref{theoesperance}), we get
$$\E[\mathbf{R}(n+1)]\le \E[\mathbf{R}(n)] + \frac{c_\mu}{n+1}.$$
This yields 
$$\E[R(bt)-R(bs)]\le c_\mu \E\left[\sum_{n=K(bs)}^{K(bt)-1}\frac{1}{n+1}\right] = c_\mu \E\left[\sum_{n=1}^{K(bt)-K(bs)}\frac{1}{n+K(bs)}\right].$$
Set $K'\defeq K(bt)-K(bs)$. We split the expectation above according to whether $K(bs)$ is smaller of larger than $bs/2$. Using the fact that $K'$  is independent of $K(bs)$ and a has Poisson distribution with parameter $b(t-s)$ we find that
\begin{eqnarray*}
\E[R(bt)-R(bs)] &\le & c_\mu \E\left[\mathbf{1}_{\big\{K(bs)\ge \frac{bs}{2}\big\}}\sum_{n=1}^{K'}\frac{1}{n+bs/2}\right]+ c_\mu\E\left[\mathbf{1}_{\big\{K(bs)< \frac{bs}{2}\big\}}\sum_{n=1}^{K'}\frac{1}{n}\right]\\
& \le & c_\mu \E\left[\log\Big( \frac{bs}{2}+ K'\Big)-\log\Big(\frac{bs}{2}\Big)\right]+ c_\mu \P\Big\{K(bs)<\frac{bs}{2}\Big\} \E\Big[1+ \log K'\Big]\\
& \le & c_\mu\left(\log( \frac{bs}{2}+  \E [K'])-\log(\frac{bs}{2})\right)+ c_\mu \P\Big\{K(bs)<\frac{bs}{2}\Big\} \Big(1+ \log  \E[K']\Big)\\
& \le & c_\mu \log\Big(\frac{2t-s}{s}\Big)+ c_\mu \P\Big\{K(bs)<\frac{bs}{2}\Big\} \Big(1+ \log  (b(t-s))\Big).
\end{eqnarray*}
 Since  $\P\{K(bs)<\frac{bs}{2}\}$ decreases exponentially in $b$, we get, for any $t>s>0$
$$\limsup_{b\to \infty} \E[|\bar{\mathcal{R}}_{[s,t]}(b)|] =  \limsup_{b\to \infty}   \E[R(bt)-R(bs)] < c_\mu \log\Big(\frac{2t-s}{s}\Big) < \infty.$$ 
Concerning $2.$, we remark that the arguments invoked to show that $\mathcal{G}_\infty$ defines a unique tree during proof of Theorem \ref{Theo:HalfPlane} do not use the specific form of the offspring distribution. Hence the exact same arguments also hold here \emph{mutatis mutandis}. Finally, the other properties of the tree are straightforward. 
\end{proof}

\begin{rem}
\begin{itemize}
\item Although we do not prove it, it seems intuitively clear that, just as in the geometric case, each forward ray in the infinite stationary tree defined by $\mathcal{G}_\infty$ must converge to some finite vertical asymptote.
\item As in Theorem \ref{Theo:HalfPlane}, the law of the restriction of $\mathcal{G}_\infty$ to a finite box $\mathbb{B} \defeq [x,y]\times [s,t]$ corresponds to the law of the graphical representation inside $\mathbb{B}$ with sources and sinks processes $\mathcal{C}$ and  $\mathcal{S}$ which are independent of the atoms inside $\mathbb{B}$ but whose joint law depends on the location of the box inside the upper half plane $\mathbb{H}$. However, contrarily to the geometric case, we do not believe that $\mathcal{C}$ and $\mathcal{S}$ should be independent anymore. This additional complexity makes challenging an explicit computation of the marginal laws.
\end{itemize}
\end{rem}

\section{Comments and questions}

There are several natural questions left open concerning the heap patience sorting algorithm  and Hammersley's tree process. Below are a few that we think might be worth further investigation. 
\begin{enumerate}
\item  Theorem \ref{Theo:HalfPlaneGeneral} states that, for any offspring distribution $\mu\neq \delta_1$, the stationary half-plane representation $\mathcal{G}_\infty$ exists. In particular, the set $\bar{\mathcal{R}}_{(0,+\infty)}(\infty)$ of the heights of the horizontal lines in $\mathcal{G}_\infty$ that intersect the vertical half-line $\{0\}\times (0,\infty)$ necessarily defines a stationary measure for the root process. Computing this stationary measure would yield the exact value of $c_\mu$ since we must have
$$
\left|\bar{\mathcal{R}}_{(1,t)}(\infty)\right| \underset{t\to\infty}{\sim} c_\mu \log t \quad\hbox{a.s.}
$$ 
This would be particularly interesting in the case of regular trees. However, as we already mentioned, this measure may have a complicated form and seems difficult to track. Are there other distributions apart from the geometric law that are exactly solvable ?  

\item In the regular case $\mu=\delta_k$, we showed that $1<c_\mu<k/(k-1)$. However, these bounds are not sharp. In particular, in the binary case $k=2$, some non rigorous heuristic leads us to believe $c_\mu$ should be strictly smaller than the golden ratio, disproving Istrate and Bonchis'conjecture.  On the other hand, for large $k$, it should be possible to improve the upper bound since we believe that $c_\mu$ is of order $1+1/2^k$ when $k$ becomes large. 
\item  Except for the geometric case, we have proved the convergence of $\mathbf{R}(n)/\log n$ only in average. However, the same result must also hold almost surely for any offspring distribution. This is a work in progress.
\item In this paper, we assumed that the trees are without leaves (\emph{i.e} $\mu(0)=0$ so a label cannot be stillborn). Yet, it seems interesting to lift up this restriction. Then, there should exist three different regimes according whether the G-W tree is sub-critical, critical or super-critical. In the sub-critical case, the number of heaps will increase linearly  since each tree contains, in average, a finite number of labels. In the super-critical case, the number of trees should again have logarithmic growth.  The critical case seems the most interesting one. We expect the number of trees to be polynomial in the number of labels. Does the number of heaps increase as $\sqrt{n}$ as in Ulam's problem ? Some computer simulations seem to support this claim. Just as in the super-critical case, we can try to answer this question first for a particular offspring distribution. Here, the natural candidate is the geometric distribution starting from $0$ with parameter $1/2$. It seems that, again, explicit calculations may be achievable and could also provide interesting insight to the general case. This is also a work in progress.    
 
\item Finally, we studied here the number of heaps created but, from a algorithmic point of view, other quantities might be relevant. The initial question of Byers \emph{et al.} \cite{Byersetal} concerns the probability of a finite sequence to be \emph{heapable} \emph{i.e.} to be sorted into a single tree. With our notations, this question is equivalent to finding the asymptotic of $\P\{\mathbb{R}(n) = 1\}$ as $n$ tends to infinity. Of course, we have $\P\{\mathbb{R}(n) = 1\} \leq 1/n$ since the first entry must be the minimum of the first $n$ labels. With a little more work, it is not difficult to show that, in the binary case,
$$
e^{-n^\alpha} \leq \P\{\mathbb{R}(n) = 1\} \leq  \frac{c}{n^2}
$$  
for some $c>0$ and some $\alpha <1$ but we do not know the exact asymptotic. Is it polynomial in $n$ ? 
\end{enumerate}

\begin{ack}
The authors warmly
thank  Nathanaël Enriquez  for stimulating    discussions on the topic.
\end{ack}

\end{document}